\documentclass[12pt]{amsart}

\usepackage[cp1251]{inputenc}
\usepackage[T2A]{fontenc}
\usepackage[english]{babel}
\usepackage{hyperref}
\usepackage{amsfonts}
\usepackage{amsbsy}
\usepackage{amssymb}
\usepackage{mathrsfs}
\usepackage{enumitem}
\usepackage{pdfsync}
\usepackage{euscript}
\usepackage{tikz-cd}

\renewcommand{\abovecaptionskip}{0pt}
\renewcommand{\belowcaptionskip}{6pt}
\makeatletter

\renewcommand{\@makecaption}[2]{
\vspace{\abovecaptionskip}%
\sbox{\@tempboxa}{#1. #2}%
\global\@minipagefalse \hbox to \hsize {{\scshape \hfil #1.
#2\hfil}} \vspace{\belowcaptionskip}}

\makeatother

\newcommand{\ad}{\mathrm{ad}}

\DeclareMathOperator{\Supp}{\mathrm{Supp}}

\DeclareMathOperator{\Hom}{\mathrm{Hom}}
\DeclareMathOperator{\ord}{\mathrm{ord}}
\DeclareMathOperator{\Quot}{\mathrm{Quot}}

\DeclareMathOperator{\Spec}{\mathrm{Spec}}
\DeclareMathOperator{\Aut}{\mathrm{Aut}}

\newcommand{\Ker}{\operatorname{Ker}}

\newcommand{\codim}{\operatorname{codim}}

\newcommand{\Der}{\operatorname{Der}}

\newcommand{\pr}{\mathrm{p}}
\newcommand{\sat}{\mathrm{sat}}

\newcommand{\quot}{/\hspace{-0.5ex}/}

\newcommand{\ZZ}{\mathbb Z}
\newcommand{\QQ}{\mathbb Q}

\newtheorem{theorem}{Theorem}

\newtheorem{proposition}[theorem]{Proposition}
\newtheorem{lemma}[theorem]{Lemma}
\newtheorem{corollary}[theorem]{Corollary}
\newtheorem{question}[theorem]{Question}

\theoremstyle{definition}

\newtheorem{definition}[theorem]{Definition}

\theoremstyle{remark}

\newtheorem{remark}[theorem]{Remark}

\makeatletter

\@addtoreset{theorem}{section}

\makeatother

\numberwithin{equation}{section}

\begin{document}

\sloppy

\renewcommand{\proofname}{Proof}
\renewcommand{\abstractname}{Abstract}
\renewcommand{\refname}{References}
\renewcommand{\figurename}{Figure}
\renewcommand{\tablename}{Table}

\title[New and old results on spherical varieties]
{New and old results on spherical varieties\\ via moduli theory}

\author{Roman Avdeev and St\'ephanie Cupit-Foutou}


\address{%
{\bf Roman Avdeev} \newline\indent National Research University
``Higher School of Economics'', Moscow, Russia}

\email{suselr@yandex.ru}

\address{%
{\bf St\'ephanie Cupit-Foutou}
\newline\indent Ruhr-Universit\"at Bochum, NA 4/67,
D-44797 Bochum, Germany}

\email{stephanie.cupit@rub.de}




\begin{abstract}
Given a connected reductive algebraic group $G$ and a finitely generated monoid $\Gamma$ of dominant weights of~$G$, in 2005 Alexeev and Brion constructed a moduli scheme $\mathrm M_\Gamma$ for multiplicity-free affine $G$-varieties with weight monoid~$\Gamma$. This scheme is equipped with an action of an `adjoint torus' $T_\ad$ and has a distinguished $T_\ad$-fixed point~$X_0$. In this paper, we obtain a complete description of the $T_\ad$-module structure in the tangent space of $\mathrm M_\Gamma$ at~$X_0$ for the case where $\Gamma$ is saturated. Using this description, we prove that the root monoid of any affine spherical $G$-variety is free. As another application, we obtain new proofs of uniqueness results for affine spherical varieties and spherical homogeneous spaces first proved by Losev in 2009. Furthermore, we obtain a new proof of Alexeev and Brion's finiteness result for multiplicity-free affine $G$-varieties with a prescribed weight monoid. At last, we prove that for saturated $\Gamma$ all the irreducible components of $\mathrm M_\Gamma$, equipped with their reduced subscheme structure, are affine spaces.
\end{abstract}

\maketitle

\section*{Introduction}

All objects considered in this paper are defined over an algebraically closed field $\Bbbk$ of characteristic~$0$.

Let $G$ be a connected reductive algebraic group. A $G$-variety (that is, an algebraic variety equipped with a regular action of~$G$) is called \textit{spherical} if it is normal and contains a dense orbit for the induced action of a Borel subgroup $B \subset G$. Famous examples of spherical varieties are toric varieties, flag varieties, and symmetric varieties. Due to a combination of numerous works and methods, the  structure theory of spherical varieties is now well understood and has recently led to a full classification of these objects; see \cite[Chapter~5]{Tim} for a review.

In this paper we obtain new results and also recover a number of already known facts on spherical varieties via one single approach---that of moduli theory, which does not involve any classification results in the theory of spherical varieties. Specifically, we are concerned with the moduli theory developed by Alexeev and Brion in~\cite{AB} for affine spherical $G$-varieties and more generally for multiplicity-free affine $G$-varieties.

An affine $G$-variety $X$ is said to be \textit{multiplicity-free} if $X$ is irreducible and the algebra $\Bbbk [X]$ of regular functions on~$X$, regarded as a $G$-module, contains every simple $G$-module with multiplicity at most~$1$. By a theorem of Vinberg and Kimelfeld~\cite{VK78}, an irreducible affine $G$-variety is multiplicity-free if and only if it possesses a dense $B$-orbit. In particular, affine spherical $G$-varieties are characterized as normal multiplicity-free affine $G$-varieties.

Given a multiplicity-free affine $G$-variety~$X$, the $G$-module structure of $\Bbbk[X]$ is encoded in the \textit{weight monoid}~$\Gamma_X$ of~$X$, consisting of all dominant weights $\lambda$ of~$G$ for which $\Bbbk[X]$ contains a simple $G$-submodule $\Bbbk[X]_\lambda$ with highest weight~$\lambda$. This monoid is known to be finitely generated. Besides, $X$ is normal if and only if $\Gamma_X$ is \textit{saturated}, that is, $\Gamma_X$ is the intersection of a lattice with a cone.

One more invariant of a multiplicity-free affine $G$-variety~$X$ is its \textit{root monoid}~$\Xi_X$, which arises from the ring structure of~$\Bbbk[X]$. By definition, $\Xi_X$ is generated by all expressions $\lambda + \mu - \nu$ with $\lambda, \mu, \nu \in \Gamma_X$ such that the linear span of $\Bbbk[X]_\lambda \cdot \Bbbk[X]_\mu$ contains~$\Bbbk[X]_\nu$. Let $\Xi_X^\sat$ denote the saturation of~$\Xi_X$, that is, the intersection of the lattice generated by $\Xi_X$ with the cone spanned by~$\Xi_X$. An important property of the root monoid was discovered by Knop in~\cite{Kn96}, who proved that the monoid $\Xi_X^\sat$ is free.

In~\cite{AB}, Alexeev and Brion constructed and studied a moduli scheme $\mathrm{M}_\Gamma$ for mul\-ti\-plicity-free affine $G$-varieties with prescribed weight monoid~$\Gamma$. This scheme is affine and of finite type; it is equipped with an action of an adjoint torus $T_\ad$ (the quotient of a maximal torus of $G$ by the center of $G$) in such a way that the $T_\ad$-orbits of $\mathrm{M}_\Gamma$ bijectively correspond to the $G$-isomorphism classes of multiplicity-free affine $G$-varieties with weight monoid~$\Gamma$. Various examples of moduli schemes $\mathrm M_\Gamma$ were further studied under different assumptions on the monoid~$\Gamma$. The case of monoids generated by a single element was worked out in~\cite{Jan}; the paper~\cite{BraCu08} dealt with free monoids that are $G$-saturated (the latter means that the monoid consists of all dominant weights of~$G$ lying in a fixed lattice);
several other special instances of free monoids were studied in~\cite{Cu, PvS12, PvS16}. In all these cases, $\mathrm M_\Gamma$ was shown to be an affine space (as a scheme). Finally, in~\cite{BvS} it was proved that for an arbitrary free monoid~$\Gamma$ all the irreducible components of~$\mathrm M_\Gamma$, equipped with their reduced subscheme structure, are affine spaces.

Given an arbitrary finitely generated monoid $\Gamma$ of dominant weights of~$G$, there always exists a multiplicity-free affine $G$-variety $X_0 = X_0(\Gamma)$ with weight monoid~$\Gamma$ such that the linear span of $\Bbbk[X_0]_\lambda \cdot \Bbbk[X_0]_\mu$ coincides with $\Bbbk[X_0]_{\lambda + \mu}$ for all $\lambda, \mu \in \Gamma$. Such varieties were first considered and studied by Vinberg and Popov in~\cite{VP72}. It is known from \cite{AB} that the $T_\ad$-orbit in~$\mathrm M_\Gamma$ corresponding to~$X_0$ is just a $T_\ad$-fixed closed point (still denoted by~$X_0$), hence the tangent space $T_{X_0} \mathrm M_\Gamma$ of $\mathrm M_\Gamma$ at $X_0$ is naturally equipped with the structure of a $T_\ad$-module.

One of the main achievements of this paper is a complete description of the $T_\ad$-module structure of $T_{X_0} \mathrm M_\Gamma$ purely in terms of $\Gamma$ in the case where $\Gamma$ is saturated (see Theorem~\ref{thm_tangent_space}). In particular, we show that $T_{X_0} \mathrm M_\Gamma$ is a multiplicity-free $T_\ad$-module whose weights, up to a sign, belong to a certain finite set $\overline \Sigma(G)$ depending only on~$G$. The set $\overline \Sigma(G)$ turns out to be a subset of the set of spherical roots of~$G$ that is well known in the theory of spherical varieties.

As a first application of our description of $T_{X_0}\mathrm{M}_\Gamma$, we show that, given an arbitrary affine spherical $G$-variety~$X$, every indecomposable element of the monoid $\Xi_X$ is primitive in the lattice $\ZZ \Xi_X$ (see Proposition~\ref{prop_indecomposable}(\ref{prop_indecomposable_b})). Combining this with the above-mentioned result of Knop on the freeness of~$\Xi_X^\sat$, we derive that the monoid $\Xi_X$ itself is free (see Theorem~\ref{thm_root_monoid_is_free}).

As a second application, we obtain a new proof of the following uniqueness result for affine spherical $G$-varieties: up to a $G$-isomorphism, every affine spherical $G$-variety $X$ is uniquely determined by the pair $(\Gamma_X, \Sigma_X)$, where $\Sigma_X$ is the set of \textit{spherical roots} of~$X$, that is, primitive elements of the lattice $\ZZ \Gamma_X$ lying on extremal rays of the cone spanned by~$\Xi_X$ (see Corollary~\ref{crl_uniqueness1}). This fact was first proved by Losev in~\cite{Lo09b}. It is worth noticing that the above-mentioned uniqueness result easily extends to arbitrary multiplicity-free affine $G$-varieties (see Corollary~\ref{crl_uniqueness1_mf}).

As a third application, we derive a new proof of a rule that enables one to determine the set $\overline \Sigma_X$ of free generators of the monoid $\Xi_X^\sat$ of an affine spherical $G$-variety $X$ starting from the set $\Sigma_X$ of spherical roots (see Theorem~\ref{thm_overline_sigma}). This rule was first obtained by Losev in~\cite{Lo09a}. (In fact, Losev's result deals with a much more general situation.)

We point out that in all the three above-mentioned applications our proofs easily reduce to checking certain combinatorial properties of the set of weights of the $T_\ad$-module $T_{X_0} \mathrm M_\Gamma$.

Using elementary additional material on spherical varieties, from the uniqueness result for affine spherical $G$-varieties we derive the uniqueness property for spherical homogeneous spaces first obtained by Losev in~\cite{Lo09a}; see our Theorem~\ref{thm_uniqueness2} for a precise statement.

We note that Losev's proofs of the above-mentioned uniqueness results for affine spherical varieties and spherical homogeneous spaces use Lie-theoretical methods; the already known classification of affine spherical homogeneous spaces comes into play in his approach. It is also worth mentioning that one more independent proof of the uniqueness property for spherical homogeneous spaces follows from a combination of Luna's paper~\cite{Lu01} and Cupit-Foutou's one~\cite{Cu}, the latter dealing with more complicated aspects of moduli theory of affine spherical varieties than in this paper.

Making use of the uniqueness property for affine spherical $G$-varieties, we recover the following result first obtained by Alexeev and Brion in~\cite{AB}: there are only finitely many $G$-isomorphism classes of multiplicity-free affine $G$-varieties with prescribed weight monoid~$\Gamma$ (see Corollary~\ref{crl_finiteness_MF}); equivalently, $\mathrm M_\Gamma$ contains only finitely many $T_\ad$-orbits. The initial proof of this fact given in~\cite{AB} used a vanishing theorem of Knop~\cite{Kn94}.

At last, combining some of the above-mentioned results, we establish the following property suspected by Brion in~\cite{Br13}: for saturated~$\Gamma$, all the irreducible components of~$\mathrm M_\Gamma$, equipped with their reduced subscheme structure, are affine spaces (see Corollary~\ref{crl_irr_comp_AS}). In the above statement, considering the reduced subscheme structure of the irreducible components of~$\mathrm M_\Gamma$ is essential: using the results of the present paper, in~\cite[\S\,7.6]{ACF18} we construct examples of saturated monoids~$\Gamma$ such that $\mathrm M_\Gamma$ is a non-reduced point.

This paper is organized as follows. In \S\,\ref{sect_notation}, we fix notation and conventions used in this paper. In \S\,\ref{sect_preliminaries} we gather some basic facts on multiplicity-free affine $G$-varieties and moduli schemes~$\mathrm M_\Gamma$. In \S\,\ref{sect_tangent_space} we obtain our description of the $T_\ad$-module structure in the tangent space of $\mathrm M_\Gamma$ at $X_0$ whenever $\Gamma$ is saturated. Applications of this description are presented in \S\,\ref{sect_applications}. Appendix~\ref{app_structure_constants} lists sign conventions for Chevalley bases of simple Lie algebras used in~\S\,\ref{subsec_step_2}. In appendix~\ref{app_invariants} we present some information on invariants of spherical homogeneous spaces; this material is needed in \S\S\,\ref{subsec_mf_var}--\ref{subsec_uniqueness2}.

\subsection*{Acknowledgements}

We thank the referees for their valuable comments and suggestions on a previous version of this paper.

The first author was supported by Dmitry Zimin's ``Dynasty'' Foundation and the Guest Program of the Max-Planck Institute for Mathematics in Bonn; he also thanks the Institute for Fundamental Science in Moscow for providing excellent working conditions.

The second author was supported by the SFB/TR 12 of the German
Research Foundation (DFG).

\section{Notation and conventions}
\label{sect_notation}

Throughout this paper, all topological terms relate to the Zariski topology. All subgroups of algebraic groups are assumed to be closed. The Lie algebras of algebraic groups denoted by capital Latin letters are denoted by the corresponding small Gothic letters. A~\textit{variety} is a separated reduced scheme of finite type. A~\textit{$K$-variety} is a variety equipped with a regular action of an algebraic group~$K$. A~\textit{$K$-isomorphism} of two $K$-varieties is a $K$-equivariant isomorphism.

$\ZZ^+ = \lbrace z \in \ZZ \mid z \ge 0 \rbrace$;

$\QQ^+ = \lbrace q \in \QQ \mid q \ge 0 \rbrace$;

$\Bbbk^\times$ is the multiplicative group of the field~$\Bbbk$;

$|X|$ is the cardinality of a finite set~$X$;

$\langle \cdot\,, \cdot \rangle$ is the natural pairing between $\Hom_\ZZ(L, \QQ)$ and~$L$, where $L$ is a lattice;

$V^*$ is the dual of a vector space~$V$;

$K^0$ is the connected component of the identity of an algebraic group~$K$;

$K_x$ is the stabilizer of a point $x$ under an action of a group~$K$;

$\mathfrak X(K)$ is the character group of a group~$K$ (in additive notation);

$k^\chi$ is the value of a character $\chi \in \mathfrak X(K)$ at an element $k$ of a group~$K$;

$Z(K)$ is the center of a group~$K$;

$N_L(K)$ is the normalizer of a subgroup $K$ in a group~$L$;

$\overline Y$ is the closure of a subset $Y$ of a scheme~$X$;

$\Bbbk[X]$ is the algebra of regular functions on a variety~$X$;

$\Bbbk(X)$ is the field of rational functions on an irreducible variety~$X$;

$\Quot A$ is the field of fractions of a commutative algebra $A$ with no zero divisors;

$\Der A$ is the space of derivations of a commutative algebra~$A$;

$[\mathfrak l, \mathfrak l]$ is the derived subalgebra of a Lie algebra~$\mathfrak l$;

$\mathcal O_X$ is the structure sheaf of a scheme~$X$;

$T_x X$ is the tangent space of a scheme $X$ at a closed point $x \in X$;

$G$ is a connected reductive algebraic group;

$B \subset G$ is a fixed Borel subgroup;

$T \subset B$ is a fixed maximal torus;

$U \subset B$ is the unipotent radical of~$B$;

$T_\ad = T / Z(G)$ is the adjoint torus;

$(\,\cdot\,, \cdot)$ is a fixed inner product on~$\mathfrak X(T) \otimes_\ZZ \QQ$ invariant with respect to the Weyl group~$N_G(T) / T$;

$\Delta \subset \mathfrak X(T)$ is the root system of $G$ with respect to~$T$;

$\Delta^+ \subset \Delta$ is the set of positive roots with respect to~$B$;

$\Pi \subset \Delta^+$ is the set of simple roots;

$\alpha^\vee \in \Hom_\ZZ(\mathfrak X(T), \ZZ)$ is the dual root corresponding to a root $\alpha \in \Delta$;

$\Lambda^+ \subset \mathfrak X(T)$ is the monoid of dominant weights with respect to~$B$;

$V(\lambda)$ is the simple $G$-module with highest weight $\lambda \in \Lambda^+$;

$U(\mathfrak g)$ is the universal enveloping algebra of~$\mathfrak g$.

The lattices $\mathfrak X(B)$ and $\mathfrak X(T)$ are identified via restricting characters from~$B$ to~$T$.

The lattice $\mathfrak X(T_\ad)$ is canonically identified with $\ZZ \Pi$.

Highest weight vectors and lowest weight vectors of all simple $G$-modules are considered with respect to~$B$.

For every $\lambda \in \mathfrak X(T)$, we set $\lambda^* = -w_0 \lambda$ where $w_0$ is the longest element of the Weyl group $N_G(T)/T$.

If $V$ is a vector space equipped with an action of a group~$K$, then the notation $V^K$ stands for the subspace of $K$-invariant vectors and, for every character $\chi$ of $K$, the notation $V^{(K)}_\chi$ stands for the subspace of $K$-semi-invariant vectors of weight~$\chi$.

Let $K$ be a group and let $K_1, K_2$ be subgroups of~$K$. We write $K = K_1 \leftthreetimes K_2$ if $K$ is a semidirect product of $K_1, K_2$ with $K_1$ being a normal subgroup of~$K$.

Let $\sigma \in \ZZ \Pi$ and consider the expression $\sigma = \sum \limits_{\alpha \in \Pi} k_\alpha \alpha$, where $k_\alpha \in \ZZ$ for all $\alpha \in \Pi$. The \textit{support} of $\sigma$ is the set $\Supp \sigma = \lbrace \alpha \in \Pi \mid \nobreak k_\alpha \ne \nobreak 0 \rbrace$. The \textit{type} of $\sigma$ is the type of the Dynkin diagram of the set $\Supp \sigma$. When the Dynkin diagram of $\Supp \sigma$ is connected, we number the simple roots in $\Supp \sigma$ as in~\cite{Bo} and denote the $i$th simple root by $\alpha_i$.

For every $\sigma \in \ZZ \Pi \setminus \lbrace 0 \rbrace$, the root subsystem of $\Delta$ with set of simple roots $\Supp \sigma$ is denoted by~$\Delta_\sigma$.

For every subset $F \subset \mathfrak X(T)$, we set $F^\perp = \lbrace \alpha \in \Pi \mid \langle \alpha^\vee, \lambda \rangle = 0 \text{ for all } \lambda \in F \rbrace$. By abuse of notation, for a single element $\lambda \in \mathfrak X(T)$ we write $\lambda^\perp$ instead of $\lbrace \lambda \rbrace^\perp$.

For every $\alpha \in \Delta$, the image of~$\alpha^\vee$ in~$\mathfrak t$ is denoted by~$h_\alpha$.

For every $\alpha \in \Delta$, we fix a nonzero root vector $e_\alpha \in \mathfrak g$ of weight $\alpha$ with respect to the adjoint action of~$T$. Moreover, we assume that the set $\lbrace h_\alpha \mid \alpha \in \Pi \rbrace \cup \lbrace e_\alpha \mid \alpha \in \Delta \rbrace$ is a Chevalley basis of the semisimple Lie algebra $[\mathfrak g, \mathfrak g]$ (for details on Chevalley bases see~\cite[\S\S\,4.1--4.2]{Ca89}).

For every $\alpha, \beta \in \Delta$ with $\alpha + \beta \in \Delta$ we let $N_{\alpha, \beta} \in \lbrace \pm1, \pm2, \pm3, \pm4 \rbrace$ be the corresponding structure constant so that $[e_\alpha, e_\beta] = N_{\alpha, \beta} e_{\alpha + \beta}$. One has $|N_{\alpha, \beta}| = p + 1$ where $p$ is the largest integer such that $\beta - p\alpha \in \Delta$.

Let $Q$ be a finite-dimensional vector space over~$\QQ$.

A subset $\mathcal C \subset Q$ is called a (finitely generated convex) \textit{cone} if there are finitely many elements $q_1, \ldots, q_s \in Q$ such that $\mathcal C = \QQ^+ q_1 + \ldots + \QQ^+ q_s$.

The \textit{dimension} of a cone is the dimension of its linear span.

The \textit{dual cone} of a cone $\mathcal C \subset Q$ is the cone
\[
\mathcal C^\vee = \lbrace \xi \in Q^* \mid \xi(q) \ge 0 \text{\;for all\;} q \in \mathcal C \rbrace.
\]
One always has $(\mathcal C^\vee)^\vee = \mathcal C$.

A \textit{face} of a cone $\mathcal C \subset Q$ is a subset $\mathcal F \subset \mathcal C$ of the form
\[
\mathcal F = \mathcal C \cap \lbrace q \in Q \mid \xi(q) = 0 \rbrace
\]
for some $\xi \in \mathcal C^\vee$. Each face of $\mathcal C$ is again a cone.

An \textit{extremal ray} of a cone $\mathcal C$ is a face of dimension~$1$ having the form $\QQ^+ q$ for some $q \in Q \setminus \lbrace 0 \rbrace$.


\section{Basic material}
\label{sect_preliminaries}

In this section, we collect basic material on multiplicity-free affine $G$-varieties and on Alexeev and Brion's moduli schemes.

\subsection{Spherical $G$-varieties and multiplicity-free affine $G$-varieties}

\begin{definition}
A $G$-variety $X$ is said to be \textit{spherical} if $X$ is normal and possesses a dense (and hence open) $B$-orbit.
\end{definition}

It follows from the definition that every spherical $G$-variety is irreducible.

Given a $G$-variety~$X$, the algebra $\Bbbk[X]$ is naturally equipped with the $G$-module structure given by $(gf)(x) = f(g^{-1}x)$ for all $g \in G$, $f \in \Bbbk[X]$, and $x \in X$.

\begin{definition}
An affine $G$-variety $X$ is said to be \textit{multiplicity-free} if $X$ is irreducible and every simple $G$-module occurs in $\Bbbk[X]$ with multiplicity at most~$1$.
\end{definition}

\begin{theorem}[{\cite[Theorem~2]{VK78}}]
Let $X$ be an irreducible affine $G$-variety. The following conditions are equivalent:
\begin{enumerate}[label=\textup{(\arabic*)},ref=\textup{\arabic*}]
\item
$X$ is multiplicity-free.

\item
$X$ possesses a dense $B$-orbit.
\end{enumerate}
\end{theorem}

\begin{corollary}
Let $X$ be an affine $G$-variety. The following conditions are equivalent:
\begin{enumerate}[label=\textup{(\arabic*)},ref=\textup{\arabic*}]
\item
$X$ is spherical.

\item
$X$ is multiplicity-free and normal.
\end{enumerate}
\end{corollary}

\subsection{The weight monoid}
\label{subsec_weight monoid}

Let $X$ be a multiplicity-free affine $G$-variety.

\begin{definition}
The \textit{weight monoid} of $X$, denoted by~$\Gamma_X$, is the set of all $\lambda \in \Lambda^+$ such that $\Bbbk[X]$ contains a simple $G$-submodule isomorphic to~$V(\lambda)$.
\end{definition}

\begin{remark}
As $\Bbbk[X]$ is an integral domain, the product of two highest weight vectors in $\Bbbk[X]$ is nonzero and hence again a highest weight vector. It follows that $\Gamma_X$ is indeed a submonoid in~$\Lambda^+$.
\end{remark}

For every $\lambda \in \Gamma_X$, we let $\Bbbk[X]_\lambda$ denote the simple $G$-submodule of $\Bbbk[X]$ isomorphic to~$V(\lambda)$, so that
\[
\Bbbk[X] = \bigoplus \limits_{\lambda \in \Gamma_X} \Bbbk[X]_\lambda.
\]

Given a submonoid $\Gamma \subset \mathfrak X(T)$, let $\Bbbk[\Gamma]$ denote the ``semigroup algebra'' of~$\Gamma$, that is, the algebra with basis $\lbrace u_\lambda \mid \lambda \in \Gamma \rbrace$ and multiplication given by $u_\lambda u_\mu = u_{\lambda + \mu}$ for all $\lambda, \mu \in \Gamma$. We equip $\Bbbk[\Gamma]$ with an action of $T$ given by the formula $t \cdot u_\lambda = t^\lambda u_\lambda$ for all $t \in T$ and $\lambda \in \Gamma$. Clearly, the multiplication of $\Bbbk[\Gamma]$ is $T$-equivariant.

\begin{proposition}[{\cite[Theorem~2]{Po86}}] \label{prop_SGalgebra}
There is a $T$-equivariant isomorphism $\Bbbk[X]^U \simeq \Bbbk[\Gamma_X]$.
\end{proposition}

\begin{corollary}
The monoid $\Gamma_X$ is finitely generated.
\end{corollary}

\begin{proof}
As the algebra $\Bbbk[X]$ is finitely generated, so is $\Bbbk[X]^U$ by~\cite[Theorem~3.1]{Ha67} (see also \cite[Corollary~4 of Theorem~4]{Po86}). It remains to apply Proposition~\ref{prop_SGalgebra}.
\end{proof}

\begin{proposition} \label{prop_normality_criterion}
The algebra $\Bbbk[X]$ is integrally closed if and only if so is $\Bbbk[X]^U$.
\end{proposition}

\begin{proof}
This is a particular case of Vust's normality criterion  \cite[\S\,1.2, Theorem~1]{Vu76} (see also \cite[Corollary of Theorem~6]{Po86}).
\end{proof}

\begin{definition}
A monoid $\Gamma \subset \mathfrak X(T)$ is said to be \textit{saturated} if it satisfies the equality $\Gamma = \ZZ \Gamma \cap \QQ^+ \Gamma$ in $\mathfrak X(T) \otimes_\ZZ \QQ$.
\end{definition}

\begin{proposition} \label{prop_normality_saturatedness}
The following conditions are equivalent:
\begin{enumerate}[label=\textup{(\arabic*)},ref=\textup{\arabic*}]
\item
$X$ is normal \textup(and hence spherical\textup).

\item
$\Gamma_X$ is saturated.
\end{enumerate}
\end{proposition}

\begin{proof}
By~\cite[Ch.~I, \S\,1, Lemma~1]{KKMS73}, the algebra $\Bbbk[\Gamma_X]$ is integrally closed if and only if $\Gamma_X$ is saturated. Now the claim follows from Propositions~\ref{prop_SGalgebra} and~\ref{prop_normality_criterion}.
\end{proof}

\subsection{The root monoid and related invariants}
\label{subsec_root_monoid}

Let $X$ be a multiplicity-free affine $G$-variety.

\begin{definition} \label{dfn_root_monoid}
The \textit{root monoid} of~$X$, denoted by~$\Xi_X$, is the monoid in $\mathfrak X(T)$ generated by all expressions $\lambda + \mu - \nu$ with $\lambda, \mu, \nu \in \Gamma_X$ such that the linear span of $\Bbbk[X]_\lambda \cdot \Bbbk[X]_\mu$ contains~$\Bbbk[X]_\nu$.
\end{definition}

It follows from the definition that $\Xi_X$ is a submonoid of~$\ZZ^+ \Pi$. It is known that $\Xi_X$ is finitely generated, see \cite[Proposition~2.13]{AB}.

Let $\Xi^\sat_X$ denote the saturation of~$\Xi_X$, that is, $\Xi^\sat_X = \ZZ \Xi_X \cap \QQ^+ \Xi_X$.

\begin{theorem}[{see~\cite[Theorem 1.3]{Kn96}}]
\label{thm_saturation_is_free} The monoid $\Xi^\sat_X$ is free.
\end{theorem}

According to Theorem~\ref{thm_saturation_is_free}, let $\overline \Sigma_X \subset \ZZ^+ \Pi$ be the set of free generators of the monoid~$\Xi_X^\sat$, that is, the linearly independent set such that
\[
\Xi^\sat_X = \ZZ^+\overline \Sigma_X.
\]

Along with the set $\overline \Sigma_X$, we shall also consider the set~$\Sigma_X$ consisting of primitive elements $\sigma$ of the lattice $\ZZ \Gamma_X$ such that $\QQ^+\sigma$ is an extremal ray of the cone $\QQ^+ \Xi_X \subset \ZZ \Gamma_X \otimes_\ZZ \QQ$. Elements of $\Sigma_X$ are called \textit{spherical roots} of~$X$.

\subsection{The $G$-variety $X_0$}
\label{subsec_X_0}

From now on until the end of \S\,\ref{subsec_characterizations}, $\Gamma \subset \Lambda^+$ is an arbitrary finitely generated monoid.

Fix an arbitrary finite generating system $\mathrm E \subset \Gamma$ and consider the $G$-module
\[
V = V(\mathrm E) = \bigoplus \limits_{\lambda \in \mathrm E}
V(\lambda)^*.
\]
For every $\lambda \in \mathrm E$, fix a lowest weight vector $v_\lambda \in V(\lambda)^*$. Put
\[
x_0 = \sum \limits_{\lambda \in \mathrm E} v_\lambda, \quad O = Gx_0, \quad \text{ and } \quad X_0 = \overline O \subset V.
\]

\begin{theorem}[{\cite[Theorem~6]{VP72}}]
\label{thm_VP}
The following assertions hold:
\begin{enumerate}[label=\textup{(\alph*)},ref=\textup{\alph*}]
\item
up to a $G$-isomorphism, the $G$-variety $X_0$ is independent of the choice of~$\mathrm E$;

\item
$X_0$ is a multiplicity-free affine $G$-variety;

\item
$\Gamma_{X_0} = \Gamma$;

\item \label{thm_VP_d}
$\Xi_{X_0} = \lbrace 0 \rbrace$, that is, the linear span of $\Bbbk[X_0]_\lambda \cdot \Bbbk[X_0]_\mu$ coincides with $\Bbbk[X_0]_{\lambda + \mu}$ for all $\lambda, \mu \in \Gamma_X$.
\end{enumerate}
\end{theorem}

\subsection{The definition of $\mathrm M_\Gamma$}
\label{subsec_definition_of_M_Gamma}

Consider the $G$-module
\begin{equation} \label{eqn_A}
A_\Gamma = \bigoplus\limits_{\lambda \in \Gamma} V(\lambda).
\end{equation}
For every $\lambda \in \Gamma$, fix a highest weight vector $u_\lambda \in V(\lambda)$. Then $A_\Gamma^U = \bigoplus \limits_{\lambda \in \Gamma} \Bbbk u_\lambda.$
We equip $A_\Gamma^U$ with an algebra structure by setting
\begin{equation} \label{eqn_ml_on_AU}
u_\lambda \cdot u_\mu = u_{\lambda + \mu} \; \text{ for all } \; \lambda, \mu \in \Gamma.
\end{equation}
Thus we get a canonical identification
\begin{equation} \label{eqn_can_alg_iso}
A_\Gamma^U \simeq \Bbbk[\Gamma].
\end{equation}

Every scheme $S$ is naturally equipped with the sheaf of $\mathcal O_S$-$G$-modules $\mathcal O_S \otimes_\Bbbk A_\Gamma$.

We consider the contravariant functor
\[
\mathcal M_\Gamma \colon \text{(Schemes)} \to \text{(Sets)}
\]
assigning to each scheme $S$ the set of $\mathcal O_S$-$G$-algebra structures on the sheaf $\mathcal O_S \otimes_\Bbbk A_\Gamma$ that extend the multiplication~(\ref{eqn_ml_on_AU}) on~$A_\Gamma^U$. By~\cite[Proposition~2.10]{AB}, this definition of $\mathcal M_\Gamma$ agrees with that given in~\cite[Definition~1.11]{AB}, see also \cite[\S\,4.3]{Br13}.

The following result is a consequence of \cite[Theorems~1.12 and~2.7]{AB}, see also~\cite[\S\,4.3]{Br13}.

\begin{theorem} \label{thm_M_Gamma}
The functor $\mathcal M_\Gamma$ is represented by an affine scheme $\mathrm M_\Gamma$ of finite type.
\end{theorem}

Let $\mathrm{ML}(A_\Gamma)$ denote the set of all $G$-equivariant multiplication laws on $A_\Gamma$ extending the multiplication~(\ref{eqn_ml_on_AU}) on~$A_\Gamma^U$.

\begin{corollary} \label{crl_cp=ml}
The set of closed points of $\mathrm M_\Gamma$ is in bijection with the set $\mathrm{ML}(A_\Gamma)$.
\end{corollary}

\subsection{Relation of $\mathrm M_\Gamma$ to multiplicity-free affine $G$-varieties with weight monoid~$\Gamma$}

Consider a multiplicity-free affine $G$-variety $X$ with weight monoid~$\Gamma$. In view of Proposition~\ref{prop_SGalgebra}, there is a $T$-equivariant algebra isomorphism
\begin{equation} \label{eqn_isom_T-alg}
\tau \colon \Bbbk[X]^U \xrightarrow{\sim} \Bbbk[\Gamma].
\end{equation}
Identifying $\Bbbk[\Gamma]$ with $A_\Gamma^U$ via~(\ref{eqn_can_alg_iso}), we get a $T$-equivariant isomorphism $\Bbbk[X]^U \xrightarrow{\sim} A_\Gamma^U$. Clearly, the latter isomorphism uniquely extends to a $G$-module isomorphism
\begin{equation} \label{eqn_isom_G-mod}
\Bbbk[X] \xrightarrow{\sim} A_\Gamma.
\end{equation}
Transferring the algebra structure from $\Bbbk[X]$ to $A_\Gamma$ via isomorphism~(\ref{eqn_isom_G-mod}), we obtain a $G$-equivariant multiplication law on $A_\Gamma$ extending the multiplication of~$A_\Gamma^U$.

Let $X_1, X_2$ be two multiplicity-free affine $G$-varieties with weight monoid~$\Gamma$ and fix $T$-equivariant isomorphisms $\tau_i \colon \Bbbk[X_i]^U \xrightarrow{\sim} \Bbbk[\Gamma]$ ($i = 1,2$). We say that the pairs $(X_1, \tau_1)$ and $(X_2, \tau_2)$ are equivalent if there is a $G$-equivariant isomorphism $\Bbbk[X_1] \xrightarrow{\sim} \Bbbk[X_2]$ such that the induced $T$-equivariant isomorphism $\Bbbk[X_1]^U \xrightarrow{\sim} \Bbbk[X_2]^U$ fits into a commutative diagram
\[
\begin{tikzcd}[column sep=0pt]
\Bbbk[X_1]^U \arrow{rr}{\sim} \arrow[swap]{dr}{\tau_1} & & \Bbbk[X_2]^U \arrow{dl}{\tau_2} & \\
 & \Bbbk[\Gamma] &
\end{tikzcd}
\]

Combining the above material with Corollary~\ref{crl_cp=ml}, we get

\begin{proposition} \label{prop_cp=pairs}
The closed points of $\mathrm M_\Gamma$ are in bijection with the equivalence classes of pairs $(X, \tau)$, where $X$ is a multiplicity-free affine $G$-variety with weight monoid~$\Gamma$ and $\tau \colon \Bbbk[X]^U \xrightarrow{\sim} \Bbbk[\Gamma]$ is a $T$-equivariant algebra isomorphism.
\end{proposition}

\subsection{Basic facts on the action of $T_\ad$ on~$\mathrm M_\Gamma$}
\label{subsec_Tad_action}

Let $A_\Gamma$ be as in~\S\,\ref{subsec_definition_of_M_Gamma}. In view of~(\ref{eqn_A}), every multiplication law $m \in \mathrm{ML}(A_\Gamma)$ can be expressed as the sum
\[
m = \sum \limits_{\lambda, \mu, \nu \in \Gamma} m_{\lambda, \mu}^\nu
\]
where each component $m_{\lambda, \mu}^\nu \colon V(\lambda) \otimes V(\mu) \to V(\nu)$ is a $G$-module homomorphism.

\begin{proposition}[{\cite[Proposition~2.11]{AB}}] \label{prop_action_on_M_Gamma}
Modulo the identification of Corollary~\textup{\ref{crl_cp=ml}}, the action of $T_\ad$ on the set of closed points of $\mathrm M_\Gamma$ is described as follows:
\begin{equation} \label{eqn_action_of_Tad_on_M}
(t \cdot m)_{\lambda, \mu}^\nu = t^{\nu -\lambda - \mu} m_{\lambda, \mu}^\nu \;\; \text{ for all } \;\; t \in T_\ad, m \in \mathrm{ML}(A_\Gamma).
\end{equation}
\end{proposition}

\begin{corollary} \label{crl_the_same_Tad-orbit}
Modulo the identification of Proposition~\textup{\ref{prop_cp=pairs}}, suppose that \textup(the equivalence classes of\textup) two closed points $(X_1, \tau_1)$ and $(X_2, \tau_2)$ of $\mathrm M_\Gamma$ lie in the same $T_\ad$-orbit. Then $X_1$ and $X_2$ are $G$-isomorphic.
\end{corollary}

\begin{theorem}[{\cite[Theorem~1.12 and Lemma~2.2]{AB}}]
\label{thm_bijection}
The $G$-isomorphism classes of multiplicity-free affine $G$-varieties with weight monoid~$\Gamma$ are in bijection with the $T_\ad$-orbits in~$\mathrm M_\Gamma$.
\end{theorem}

The following result is a consequence of Corollary~\ref{crl_the_same_Tad-orbit} and Theorem~\ref{thm_bijection}.

\begin{corollary} \label{crl_orbit_of_X_in_M}
Suppose that $X$ is a multiplicity-free affine $G$-variety with weight monoid~$\Gamma$. Then, modulo the identification of Proposition~\textup{\ref{prop_cp=pairs}}, the closed points of the $T_\ad$-orbit in $\mathrm M_\Gamma$ corresponding to~$X$ are all \textup(equivalence classes of\textup) pairs of the form $(X, \tau)$.
\end{corollary}

According to Theorem~\ref{thm_bijection}, for every multiplicity-free affine $G$-variety $X$ with weight monoid~$\Gamma$ we let $C_X$ denote the closure of the $T_\ad$-orbit in~$\mathrm M_\Gamma$ corresponding to the $G$-isomorphism class of~$X$.

Since $\Xi_{X_0} = \lbrace 0 \rbrace$ by Theorem~\ref{thm_VP}(\ref{thm_VP_d}), it follows from (\ref{eqn_action_of_Tad_on_M}) and Corollary~\ref{crl_orbit_of_X_in_M} that the $T_\ad$-orbit in $\mathrm M_\Gamma$ corresponding to~$X_0$ is a $T_\ad$-fixed (closed) point. In what follows, by abuse of notation, we denote this point by~$X_0$. In particular, $C_{X_0} = \lbrace X_0 \rbrace$.

\begin{theorem}[{\cite[Theorem~2.7]{AB}}] \label{thm_X_0_fixed_point}
The $T_\ad$-fixed point $X_0 \in \mathrm M_\Gamma$ is the unique closed $T_\ad$-orbit in~$\mathrm M_\Gamma$. Equivalently, $X_0$ is contained in each $T_\ad$-orbit closure in~$\mathrm M_\Gamma$.
\end{theorem}

\begin{theorem}[{\cite[Proposition~2.13]{AB}}]
\label{thm_orbit_closure_of_X}
Let $X$ be a multiplicity-free affine $G$-variety with weight monoid~$\Gamma$. The $T_\ad$-orbit closure $C_X \subset \mathrm M_\Gamma$, equipped with its reduced subscheme structure, is a multiplicity-free affine $T_\ad$-variety whose weight monoid is~$\Xi_X$.
\end{theorem}

\begin{corollary} \label{crl_ind_elements}
Under the hypotheses of Theorem~\textup{\ref{thm_orbit_closure_of_X}}, the tangent space $T_{X_0} C_X$ is a multiplicity-free $T_\ad$-module whose set of weights is
\[
\lbrace -\tau \mid \tau \text{ is an indecomposable element of\, } \Xi_X \rbrace.
\]
\end{corollary}

\begin{proof}
By Theorem~\ref{thm_orbit_closure_of_X}, $C_X$ is a multiplicity-free affine $T_\ad$-variety with weight monoid~$\Xi_X$, so that $\Bbbk[C_X] = \bigoplus \limits_{\xi \in \Xi_X} \Bbbk[C_X]^{(T_\ad)}_\xi$. Recall from Theorem~\ref{thm_X_0_fixed_point} that $X_0$ is the unique $T_\ad$-fixed closed point in~$\mathrm M_{\Gamma_X}$ (and hence in~$C_X$), therefore it corresponds to the maximal ideal
\[
I = \bigoplus \limits_{\xi \in \Xi_X \setminus \lbrace 0 \rbrace} \Bbbk[C_X]^{(T_\ad)}_\xi \subset \Bbbk[C_X].
\]
Now the assertion follows from the $T_\ad$-equivariant isomorphism $T_{X_0} C_X \simeq (I / I^2)^*$.
\end{proof}

\subsection{Characterizations of $T_{X_0} \mathrm M_\Gamma$}
\label{subsec_characterizations}

In this subsection we present general facts on the $T_\ad$-module structure of the tangent space $T_{X_0} \mathrm M_\Gamma$.

Let $\mathrm E$, $V$, $v_\lambda$ ($\lambda \in \mathrm E$), $x_0$, $O$, and $X_0$ be as in~\S\,\ref{subsec_X_0}. Let $a^*_G \colon (g,v) \mapsto g*v$ be the natural action of $G$ on~$V$. Given $t \in T$, let $\overline t$  denote the image of $t$ in~$T_\ad$.

According to~\cite[\S\,2.1]{AB}, we define an action $a^*_\ad \colon (\overline t, v) \mapsto \overline t * v$ of $T_\ad$ on~$V$ in the following way. For every $\lambda \in \mathrm E$ and $v \in V(\lambda)^*$, we set
\[
\overline t * v = t^{-\lambda}(t^{-1}*v),
\]
and then extend the action to the whole~$V$. Note that
\begin{equation} \label{eqn_T_ad-inv_in_V}
V^{T_\ad} = \bigoplus \limits_{\lambda \in \mathrm E} \Bbbk v_\lambda.
\end{equation}

We introduce the semi-direct product $\widetilde G = G \leftthreetimes T_\ad$ given by $\overline t g \overline t^{-1} = t^{-1}gt$ for all $t \in T$ and $g \in G$. Then the actions $a^*_G$ and $a^*_\ad$ extend to an action of~$\widetilde G$ on~$V$, which will be denoted by~$a^*$. Observe that $x_0 \in V^{T_\ad}$ by~(\ref{eqn_T_ad-inv_in_V}), and so the orbit $Gx_0$ is $T_\ad$-stable and hence $\widetilde G$-stable. It follows that $X_0$ is $\widetilde G$-stable.

All actions of $\widetilde G$ (resp.~$T_\ad$) that will be considered in the remaining part of this subsection are induced by the action~$a^*$ (resp.~$a_\ad^*$) on~$V$ and its restriction to~$X_0$.

Let $\Omega^1_V$ (resp.~$\Omega^1_{X_0}$) denote the sheaf of differential $1$-forms on~$V$ (resp.~$X_0$). Consider the canonical closed immersion $i \colon X_0 \hookrightarrow V$ and let $\mathcal I$ be the corresponding ideal sheaf. By \cite[Proposition~II.8.12]{Ha77} or \cite[Proposition~6.1.24(d)]{Li02}, there is an exact sequence of coherent sheaves of $\mathcal O_{X_0}$-modules
\[
i^* (\mathcal I/\mathcal I^2) \to i^*\Omega^1_V \to \Omega^1_{X_0} \to 0.
\]
We note that $i^*\Omega^1_V \simeq \mathcal O_{X_0} \otimes_{\Bbbk} V^*$ as $\mathcal O_{X_0}$-modules. Applying $Hom_{\mathcal O_{X_0}}(-, \mathcal O_{X_0})$ to the above exact sequence, we obtain an exact sequence of coherent sheaves of $\mathcal O_{X_0}$-modules
\begin{equation} \label{eqn_sheaves}
0 \to \mathcal T_{X_0} \to \mathcal O_{X_0} \otimes_\Bbbk V \to \mathcal N_{X_0},
\end{equation}
where
\[
\mathcal T_{X_0} = Hom_{\mathcal O_{X_0}}(\Omega^1_{X_0}, \mathcal O_{X_0})
\]
is the tangent sheaf of $X_0$ (that is, the sheaf of $\Bbbk$-derivations of the sheaf~$\mathcal O_{X_0}$) and
\[
\mathcal N_{X_0} = Hom_{\mathcal O_{X_0}}(i^*(\mathcal I/\mathcal I^2), \mathcal O_{X_0})
\]
is the normal sheaf of $X_0$ in~$V$. Taking global sections in~(\ref{eqn_sheaves}) yields an exact sequence of $\Bbbk[X_0]$-$\widetilde G$-modules
\begin{equation} \label{eqn_global_sections}
0 \to H^0(X_0, \mathcal T_{X_0}) \to H^0(X_0, \mathcal O_{X_0} \otimes_\Bbbk V) \to H^0(X_0, \mathcal N_{X_0}) \to T^1(X_0) \to 0,
\end{equation}
where $T^1(X_0)$, called the \textit{space of infinitesimal deformations of~$X_0$}, is by definition the cokernel of the map $H^0(X_0, \mathcal O_{X_0} \otimes_\Bbbk V) \to H^0(X_0, \mathcal N_{X_0})$ in~(\ref{eqn_global_sections}) (see \cite[Exercise~III.9.8]{Ha77}).

The following characterization of the tangent space $T_{X_0} \mathrm M_\Gamma $, which is implicitly contained in~\cite{AB}, has already been mentioned in~\cite[Subsection~4.3]{Br13}. For the reader's convenience, we provide it together with a proof.

\begin{proposition} \label{prop_T1}
There is a $T_\ad$-module isomorphism $T_{X_0} \mathrm M_\Gamma \simeq T^1(X_0)^G$.
\end{proposition}

\begin{proof}
Applying~\cite[Proposition~2.8]{AB}, we obtain an exact sequence of $T_\ad$-modules
\[
0 \to \mathrm{Der}^G(\Bbbk[X_0]) \to \mathrm{Der}^T(\Bbbk[X_0]^U) \to T_{X_0} \mathrm M_\Gamma \to T^1(X_0)^G \to T^1(X_0 \quot U)^T \rightarrow 0,
\]
where $X_0 \quot U = \Spec \Bbbk[X_0]^U$. By~\cite[Proposition~1.15(ii)]{AB}, $T^1(X_0 \quot U)^T$ is trivial. Therefore it remains to prove that the map $\Der^G(\Bbbk[X_0]) \to \Der^T(\Bbbk[X_0]^U)$, given by restricting derivations from $\Bbbk[X_0]$ to $\Bbbk[X_0]^U$, is surjective (and hence an isomorphism). To this end, let $B$ act on $\Bbbk[G]$ by right multiplication and on $\Bbbk[X_0]^U$ in such a way that each $T$-eigenvector of weight~$\lambda$ is multiplied by the character $-\lambda^*$. Then there is a $G$-equivariant isomorphism of algebras
\begin{equation} \label{eqn_isomorphism}
\Bbbk[X_0] \simeq (\Bbbk[G] \otimes_\Bbbk \Bbbk[X_0]^U)^B,
\end{equation}
where $B$-invariants are taken with respect to the diagonal action of~$B$ and the action of $G$ on the right-hand side is induced by that on $\Bbbk[G]$ by left multiplication. It is clear from~(\ref{eqn_isomorphism}) that every $T$-equivariant derivation of $\Bbbk[X_0]^U$ extends to a $G$-equivariant derivation of~$\Bbbk[X_0]$.
\end{proof}

\begin{corollary} \label{crl_exact_sequence_X0}
There is an exact sequence of $T_\ad$-modules
\[
0 \to H^0(X_0, \mathcal T_{X_0})^G \to H^0(X_0, \mathcal O_{X_0} \otimes_\Bbbk V)^G \to H^0(X_0, \mathcal N_{X_0})^G \to T_{X_0} \mathrm M_\Gamma \to 0.
\]
\end{corollary}

\begin{proof}
This follows by taking $G$-invariants in~(\ref{eqn_global_sections}) and applying Proposition~\ref{prop_T1}.
\end{proof}

Given a smooth open subset $Y \subset X_0$, the restrictions to $Y$ of all the sheaves appearing in~(\ref{eqn_sheaves}) are well known to be locally free, hence they may be regarded as the sheaves of sections of vector bundles on~$Y$. More precisely, $\left. \mathcal T_{X_0} \right|_Y$ (resp. $\left. (\mathcal O_{X_0} \otimes_{\Bbbk} V)\right|_Y$, $\left. \mathcal N_{X_0} \right|_Y$) will be regarded as the sheaf of sections of the tangent bundle of~$Y$ (resp. trivial bundle $Y \times V$, normal bundle of $Y$ in~$V$). If, in addition, $Y$ is $G$-stable, then the three vector bundles are $G$-linearized in a natural way.

\begin{proposition} \label{prop_exact_sequence_O}
The exact sequence of $T_\ad$-modules
\[
0 \to H^0(O, \left. \mathcal T_{X_0}\right|_O)^G \to H^0(O, \left. (\mathcal O_{X_0} \otimes_{\Bbbk} V)\right|_O)^G \to H^0(O, \left. \mathcal N_{X_0} \right|_O)^G
\]
identifies with
\[
0 \to (\mathfrak g x_0)^{G_{x_0}} \to V^{G_{x_0}} \to (V/\mathfrak gx_0)^{G_{x_0}}.
\]
\end{proposition}

\begin{proof}
As $O$ is $G$-homogeneous, for every $G$-linearized vector bundle $p \colon F \to O$ the space of its $G$-invariant sections is canonically isomorphic to $(p^{-1}(x_0))^{G_{x_0}}$. Applying this to our three vector bundles yields the claim.
\end{proof}

By~\cite[Lemma~3.9]{Br13}, for every coherent sheaf $\mathcal F$ on $X_0$ the restriction map $H^0(X_0, \mathcal G) \to H^0(O, \left. \mathcal G \right|_O)$ is injective, where $\mathcal G = Hom_{\mathcal O_{X_0}}(\mathcal F, \mathcal O_{X_0}$). Combining this with Corollary~\ref{crl_exact_sequence_X0} and Proposition~\ref{prop_exact_sequence_O} we obtain the following result.

\begin{proposition} \label{prop_large_cd}
There is a commutative diagram of\, $T_\ad$-modules
\begin{equation} \label{eqn_commutative_diagram}
\begin{tikzcd}[column sep = 0.9em, row sep = small]
0 \arrow{r} & H^0(X_0, \mathcal T_{X_0})^G \arrow{r} \arrow{d} & H^0(X_0, \mathcal O_{X_0} \otimes_\Bbbk V)^G \arrow{r} \arrow{d} & H^0(X_0, \mathcal N_{X_0})^G \arrow{r} \arrow{d} & T_{X_0} \mathrm M_\Gamma \arrow{r} & 0 \\
0 \arrow{r} & (\mathfrak gx_0)^{G_{x_0}} \arrow{r} & V^{G_{x_0}} \arrow{r} & (V / \mathfrak g x_0)^{G_{x_0}} & &
\end{tikzcd}
\end{equation}
where the rows are exact and the vertical arrows are injective maps.
\end{proposition}

\begin{proposition} \label{prop_tangent_space_characterization}
There is an exact sequence of\, $T_\ad$-modules
\[
0 \to H^0(X_0, \mathcal N_{X_0})^{\widetilde G} \to H^0(X_0, \mathcal N_{X_0})^G \to T_{X_0} \mathrm M_\Gamma \to 0.
\]
\end{proposition}

\begin{remark} \label{rem_equality}
$H^0(X_0, \mathcal N_{X_0})^{\widetilde G}  = [H^0(X_0, \mathcal N_{X_0})^G]^{T_\ad}$.
\end{remark}

\begin{proof}[{Proof of Proposition~\textup{\ref{prop_tangent_space_characterization}}}]
The claim will follow as soon as we show that the image of the map
\begin{equation} \label{eqn_map}
H^0(X_0, \mathcal O_{X_0} \otimes_\Bbbk V)^G \to H^0(X_0, \mathcal N_{X_0})^G
\end{equation}
in~(\ref{eqn_commutative_diagram}) coincides with $H^0(X_0, \mathcal N_{X_0})^{\widetilde G}$. Since $G_{x_0}$ contains a maximal unipotent subgroup of~$G$, it follows that the space $V^{G_{x_0}}$ is just the linear span of all vectors $v_\lambda$ with $\lambda \in \mathrm E$, which implies
\begin{equation} \label{eqn_V_invariants}
V^{G_{x_0}} = V^{T_\ad}
\end{equation}
by~(\ref{eqn_T_ad-inv_in_V}). Therefore $T_\ad$ acts trivially on $V^{G_{x_0}}$ and hence on $H^0(X_0, \mathcal O_{X_0} \otimes_\Bbbk V)^G$. Thus the image of the map~(\ref{eqn_map}) is contained in
$H^0(X_0, \mathcal N_{X_0})^{\widetilde G}$, and so there is a commutative diagram of vector spaces
\begin{equation} \label{eqn_commutative_diagram1}
\begin{tikzcd}[column sep = small, row sep = small]
H^0(X_0, \mathcal O_{X_0} \otimes_\Bbbk V)^G \arrow{r} \arrow{d} & H^0(X_0, \mathcal N_{X_0})^{\widetilde G} \arrow{d} \\
V^{G_{x_0}} \arrow{r} & {[(V / \mathfrak g x_0)^{G_{x_0}}]^{T_\ad}}
\end{tikzcd}
\end{equation}
where the vertical arrows are injective maps.
Further, note that
\[
H^0(X_0, \mathcal O_{X_0} \otimes_\Bbbk V) \simeq \Bbbk[X_0] \otimes_\Bbbk V
\]
and $\Bbbk[X_0] \otimes_\Bbbk V$ contains $V(\lambda) \otimes V(\lambda)^*$ as a $G$-submodule for every $\lambda \in \mathrm E$. Since $\dim (V(\lambda) \otimes\nobreak V(\lambda)^*)^G \ge 1$ for every $\lambda \in \Lambda^+$, it follows that $\dim H^0(X_0, \mathcal O_{X_0} \otimes_\Bbbk V)^G \ge |\mathrm E|$. On the other hand, it is clear that $\dim V^{G_{x_0}} = |\mathrm E|$. Consequently, the left vertical arrow in~(\ref{eqn_commutative_diagram1}) is an isomorphism. At last, the surjectivity of the natural map $V^{T_\ad} \to (V / \mathfrak g x_0)^{T_\ad}$ along with~(\ref{eqn_V_invariants}) implies that the lower horizontal arrow in~(\ref{eqn_commutative_diagram1}) is a surjective map. The latter already suffices to conclude that the map given by the upper horizontal arrow in~(\ref{eqn_commutative_diagram1}) is also surjective.
\end{proof}

Combining Proposition~\ref{prop_tangent_space_characterization} with Remark~\ref{rem_equality} we obtain

\begin{corollary} \label{crl_weights_are_nonzero}
All $T_\ad$-weights of $T_{X_0} \mathrm M_\Gamma$ are nonzero.
\end{corollary}

\section{The tangent space of $\mathrm M_\Gamma$ at $X_0$}
\label{sect_tangent_space}

Throughout this section, we fix the following notation:

$\Gamma \subset \Lambda^+$ is an arbitrary finitely generated and saturated monoid;

$\mathcal L = \Hom_\ZZ(\ZZ \Gamma, \ZZ)$;

$\mathcal Q = \mathcal L \otimes_\ZZ \QQ = \Hom_\ZZ(\ZZ\Gamma, \QQ)$;

$\mathcal K \subset \mathcal Q$ is the cone dual to $\QQ^+ \Gamma$;

$\mathcal K^1 \subset \mathcal K$ is the set of primitive elements $q$ in $\mathcal L$ such that $\QQ^+ q$ is an extremal ray of~$\mathcal K$;

$\iota \colon \Hom_\ZZ(\mathfrak X(T), \QQ) \to \mathcal Q$ is the natural restriction map.

For every $\sigma \in \ZZ \Gamma$ we define the set
\begin{equation} \label{eqn_K1(sigma)}
\mathcal K^1(\sigma) = \lbrace \varrho \in \mathcal K^1 \mid \langle \varrho, \sigma \rangle > 0 \rbrace.
\end{equation}

\subsection{Statement of the main result}
\label{Statement of the main result}

We first describe the set $\overline \Sigma(G)$. By definition, an element $\sigma \in \mathfrak X(T)$ belongs to $\overline \Sigma(G)$ if and only if $\sigma \in \ZZ^+ \Pi \setminus \lbrace 0 \rbrace$ and the expression of $\sigma$ as a linear combination of the simple roots in $\Supp \sigma$ appears in Table~\ref{table_spherical_roots}. (In row~3 of this table, $\alpha$ and $\beta$ are the two distinct simple roots in $\Supp \sigma$.)

\begin{table}[h]

\caption{The set $\overline \Sigma(G)$} \label{table_spherical_roots}

\begin{tabular}{|c|c|c|c|c|}
\hline

No. & \begin{tabular}{c} Type of \\ $\Supp \sigma$\end{tabular} &
$\sigma$ & $\Pi_\sigma$ & Note\\

\hline

1 & $\mathsf A_1$ & $\alpha_1$ & $\varnothing$ & \\

\hline

2 & $\mathsf A_1$ & $2\alpha_1$ & $\varnothing$ & \\

\hline

3 & $\mathsf A_1 \times \mathsf A_1$ & $\alpha +
\beta$ & $\varnothing$ & \\

\hline

4 & $\mathsf A_r$ & $\alpha_1 + \alpha_2 + \ldots + \alpha_r$ &
\begin{tabular}{c}
$\varnothing$ for $r = 2$; \\
\hline $\alpha_2, \alpha_3, \ldots, \alpha_{r-1}$ \\ for $r \ge 3$
\end{tabular}
& $r \ge 2$\\

\hline

5 & $\mathsf A_3$ & $\alpha_1 + 2\alpha_2 + \alpha_3$
& $\alpha_1, \alpha_3$ & \\

\hline

6 & $\mathsf B_r$ & $\alpha_1 + \alpha_2 + \ldots + \alpha_r$ &
\begin{tabular}{c}
$\varnothing$ for $r = 2$; \\
\hline $\alpha_2, \alpha_3, \ldots, \alpha_{r-1}$ \\ for $r \ge 3$
\end{tabular}
& $r \ge 2$\\

\hline

7 & $\mathsf B_r$ & $2\alpha_1 + 2\alpha_2 + \ldots +
2\alpha_r$ & $\alpha_2, \alpha_3, \ldots, \alpha_r$ & $r \ge 2$\\

\hline

8 & $\mathsf B_3$ & $\alpha_1 + 2\alpha_2 + 3\alpha_3$
& $\alpha_1, \alpha_2$ & \\

\hline

9 & $\mathsf C_r$ & $\alpha_1 + 2\alpha_2 + 2\alpha_3 + \ldots +
2\alpha_{r-1} + \alpha_r$ & $\alpha_3, \alpha_4, \ldots,
\alpha_r$ & $r \ge 3$\\

\hline

10 & $\mathsf D_r$ & $2\alpha_1 + 2\alpha_2 + \ldots +
2\alpha_{r-2} + \alpha_{r-1} + \alpha_r$ & $\alpha_2, \alpha_3,
\ldots, \alpha_r$ & $r \ge 4$ \\

\hline

11 & $\mathsf F_4$ & $\alpha_1 + 2\alpha_2 + 3\alpha_3 +
2\alpha_4$ & $\alpha_1, \alpha_2, \alpha_3$ & \\

\hline

12 & $\mathsf G_2$ & $\alpha_1 + \alpha_2$ & $\varnothing$ & \\

\hline

13 & $\mathsf G_2$ & $4\alpha_1 + 2\alpha_2$ & $\alpha_2$ & \\

\hline

\end{tabular}

\end{table}

Each element $\sigma \in \overline \Sigma(G)$ comes together with a certain subset $\Pi_\sigma \subset \Supp \sigma$, which can be defined as follows:
\begin{equation} \label{eqn_Pi_sigma}
\Pi_\sigma = \lbrace \gamma \in \Supp \sigma \cap \sigma^\perp \mid \sigma - \gamma \notin \Delta^+ \text{ or } \gamma \in \Supp(\sigma - \gamma) \rbrace.
\end{equation}
For the reader's convenience, in Table~\ref{table_spherical_roots} we listed all roots in $\Pi_\sigma$ for each~$\sigma \in \overline \Sigma(G)$. We note that
\begin{itemize}
\item
$\Pi_\sigma = \Supp \sigma \cap \sigma^\perp$ unless $\sigma$ is in rows 6 or~9 of Table~\ref{table_spherical_roots};
\item
$\Pi_\sigma = \lbrace \gamma \in \Supp \sigma \cap \sigma^\perp \mid \sigma - \gamma \notin \Delta^+ \rbrace$ unless $\sigma$ is in row~11 of Table~\ref{table_spherical_roots}.
\end{itemize}

Observe that the set $\overline \Sigma(G)$ is finite and depends only on~$G$.

We set
\[
\Phi (\Gamma)= \lbrace \sigma \in \mathfrak X(T_\ad) \mid -\sigma
\textit{ is a $T_\ad$-weight of~$T_{X_0}\mathrm M_\Gamma$}\rbrace.
\]
In other words, $\Phi(\Gamma)$ is the set of $T_\ad$-weights in the cotangent space of $\mathrm M_\Gamma$ at~$X_0$.

Note that $0 \notin \Phi(\Gamma)$ by Corollary~\ref{crl_weights_are_nonzero}.

\begin{theorem} \label{thm_tangent_space}
The tangent space $T_{X_0} \mathrm M_\Gamma$ is a multiplicity-free $T_\ad$-module. Moreover, an element $\sigma \in \mathfrak X(T_\ad)$ belongs to $\Phi(\Gamma)$ if and only if the following conditions are satisfied:
\begin{enumerate}[label=\textup{($\Phi$\arabic*)},ref=\textup{$\Phi$\arabic*}]
\item \label{Phi1}
$\sigma \in \ZZ\Gamma$;

\item \label{Phi2}
$\sigma \in \overline \Sigma(G)$;

\item \label{Phi3}
$\Pi_\sigma \subset \Gamma^\perp$;

\item \label{Phi4}
if $\sigma = \alpha_1 + \ldots + \alpha_r$ with $\Supp \sigma$ of type $\mathsf B_r$ \textup($r \ge 2$\textup), then $\alpha_r \notin \Gamma^\perp$;

\item \label{Phi5}
if $\sigma = \alpha + \beta$ with $\alpha, \beta \in \Pi$ and $\alpha \perp \beta$, then $\langle \alpha^\vee, \lambda \rangle = \langle \beta^\vee, \lambda \rangle$ for all $\lambda \in \Gamma$;

\item \label{Phi6}
if $\sigma = 2\alpha$ for some $\alpha \in \Pi$ then $\langle \alpha^\vee, \lambda \rangle \in 2\ZZ$ for all $\lambda \in \Gamma$;

\item \label{Phi7}
if $\sigma \notin \Pi$ then for every $\varrho \in \mathcal K^1(\sigma)$ there exists $\delta \in \Pi \setminus \Gamma^\perp$ such that $\iota(\delta^\vee)$ is a positive multiple of~$\varrho$;

\item \label{Phi8}
if $\sigma = \alpha \in \Pi$ then there exist two distinct elements $\varrho_1, \varrho_2 \in \mathcal K \cap \mathcal L$ with the following properties:
\begin{enumerate}[label=\textup{(\alph*)},ref=\textup{\alph*}]
\item \label{Phi8a}
$\langle \varrho_1, \alpha \rangle = \langle \varrho_2, \alpha \rangle = 1$;

\item \label{Phi8b}
$\iota(\alpha^\vee) = b_1\varrho_1 + b_2\varrho_2$ for some $b_1, b_2 \in \QQ^+ \setminus \lbrace 0 \rbrace$;

\item \label{Phi8c}
$\mathcal K^1(\alpha) \subset \lbrace \varrho_1, \varrho_2 \rbrace$.
\end{enumerate}
\end{enumerate}
\end{theorem}

\begin{remark}
In condition~(\ref{Phi8}) it is important that the elements $\varrho_1, \varrho_2$ be distinct.
\end{remark}

\begin{remark}
Conditions (\ref{Phi1})--(\ref{Phi6}) depend only on the lattice $\ZZ \Gamma$, whereas (\ref{Phi7}) and (\ref{Phi8}) are the only conditions involving the cone $\QQ^+ \Gamma \subset \ZZ \Gamma \otimes_\ZZ \QQ$.
\end{remark}

\begin{remark}
In the case where $\Gamma$ is free, a result similar to Theorem~\ref{thm_tangent_space} is proved by Bravi and Van Steirteghem in~\cite{BvS}; see Theorem~4.1, Corollary~4.2, and Corollary~2.17 in loc.~cit. Although our proof and that of loc. cit. follow the same general strategy, below we point out two main differences between the two approaches.

(1) When proving~(\ref{Phi2}) for every $\sigma \in \Phi(\Gamma)$, Bravi and Van Steirteghem establish a more general fact that (\ref{Phi2}) holds for any nonzero weight of the $T_\ad$-module $(V/\mathfrak gx_0)^{G_{x_0}}$. However, their arguments resort to an extensive case-by-case analysis of root systems. On the other hand, in our proof of~(\ref{Phi2}) for~$\sigma \in \Phi(\Gamma)$ we avoid long case-by-case considerations thanks to Proposition~\ref{prop_canonical_form}, which imposes strong restrictions on an element in $(V/\mathfrak gx_0)^{G_{x_0}}$ arising from a $T_\ad$-eigenvector in $T_{X_0} \mathrm M_\Gamma$.

(2) To prove that every $\sigma \in \mathfrak X(T_\ad)$ satisfying (\ref{Phi1})--(\ref{Phi8}) belongs to~$\Phi(\Gamma)$ (for free~$\Gamma$), Bravi and Van Steirteghem use Theorem~\ref{thm_orbit_closure_of_X} and existence results for affine spherical $G$-varieties $X$ with $\Gamma_X = \Gamma$ and $|\Sigma_X|=1$, which trace back to the known classification of so-called wonderful varieties of rank~1, see loc. cit. for details. In our proof, for every $\sigma \in \mathfrak X(T_\ad)$ satisfying (\ref{Phi1})--(\ref{Phi8}) we explicitly construct an element in $(V/\mathfrak gx_0)^{G_{x_0}}$ that gives rise to a $T_\ad$-eigenvector in~$T_{X_0} \mathrm M_\Gamma$ of weight~$-\sigma$.
\end{remark}

\begin{remark}
In \cite{BvS}, elements of $\overline \Sigma(G)$ are referred to as \textit{spherically closed spherical roots} of~$G$.
\end{remark}

We now briefly describe the contents of the remaining part of this section. In \S\,\ref{subsec_preliminaries_for_the_proof} we gather further notation and conventions needed for the proof of Theorem~\ref{thm_tangent_space}. In \S\S\,\ref{subsec_saturatedness}--\ref{subsec_canonical_representatives} we discuss several ingredients for the proof. The proof itself is divided into two steps carried out in \S\,\ref{subsec_step_1} and \S\,\ref{subsec_step_2}, respectively. At the first step we prove that the $T_\ad$-module $T_{X_0} \mathrm M_\Gamma$ is multiplicity-free and every element of $\Phi(\Gamma)$ satisfies conditions~(\ref{Phi1})--(\ref{Phi8}). At the second step we prove that every element $\sigma \in \mathfrak X(T_\ad)$ satisfying (\ref{Phi1})--(\ref{Phi8}) belongs to~$\Phi(\Gamma)$.

\subsection{Preliminaries for the proof of Theorem~\ref{thm_tangent_space}}
\label{subsec_preliminaries_for_the_proof}

In this subsection we set up an additional notation and make several conventions that will be used in our proof of Theorem~\ref{thm_tangent_space}.

We fix an arbitrary finite generating system $\mathrm E \subset \Gamma$. For every $\varrho \in \mathcal K^1$, we set
\[
\mathrm E_\varrho = \lbrace \lambda \in \mathrm E \mid \langle \varrho, \lambda \rangle = 0 \rbrace.
\]
Next, for every $\lambda \in \mathrm E$ we fix a lowest weight vector $v_{\lambda} \in V(\lambda)^*$ and put
\[
V = \bigoplus \limits_{\lambda \in \mathrm E} V(\lambda)^*, \quad x_0 = \sum \limits_{\lambda \in \mathrm E} v_\lambda, \quad O =
Gx_0,\quad\mbox{ and } \quad X_0 = \overline O\subset V.
\]

From \S\,\ref{subsec_characterizations}, recall the group $\widetilde G = G \leftthreetimes T_\ad$ and the action $a_G^*$ (resp. $a_\ad^*$,~$a^*$) of $G$ (resp. $T_\ad$,~$\widetilde G$) on~$V$ under which $X_0$ is stable. Combining Propositions~\ref{prop_large_cd} and~\ref{prop_tangent_space_characterization} we get a diagram of $T_\ad$-equivariant maps
\begin{equation} \label{eqn_just_diagram}
\begin{tikzcd}[column sep = small, row sep = small]
0 \arrow{r} & H^0(X_0, \mathcal N_{X_0})^{\widetilde G} \arrow{r} & H^0(X_0, \mathcal N_{X_0})^G \arrow{r} \arrow{d} & T_{X_0} \mathrm M_\Gamma \arrow{r} & 0 \\
 & & (V / \mathfrak g x_0)^{G_{x_0}} & &
\end{tikzcd}
\end{equation}
where the upper row is exact and the vertical arrow is an injective map. We identify $T_{X_0} \mathrm M_\Gamma$ with the unique $T_\ad$-submodule of $H^0(X_0, \mathcal N_{X_0})^G$ complementary to $H^0(X_0, \mathcal N_{X_0})^{\widetilde G}$ and let $\EuScript{TS}$ denote the image of $T_{X_0} \mathrm M_\Gamma$ in $(V / \mathfrak gx_0)^{G_{x_0}}$, so that
\[
\EuScript{TS} \simeq T_{X_0} \mathrm M_\Gamma
\]
as $T_\ad$-modules.

For our computations with the $T_\ad$-module $(V / \mathfrak gx_0)^{G_{x_0}}$, it will be more convenient to replace the actions $a_G^*$, $a_\ad^*$, and $a^*$ with other ones as described below.

As in~\S\,\ref{subsec_characterizations}, we let $\overline t$ denote the image in $T_\ad$ of an element $t \in T$. Let $\theta \in \Aut G$ be a Weyl involution of $G$ relative to~$T$, that is, $\theta(t) = t^{-1}$ for all~$t \in T$. It is well known that $\theta(B) \cap B = T$. We extend this involution to an involution of $\widetilde G$ by setting $\theta(\overline t) = \overline t^{-1}$ for all $\overline t \in T_\ad$.

We define a new action $a \colon ((g, \overline t), v) \mapsto (g, \overline t) \cdot v$ of $\widetilde G$ on $V$ by $(g, \overline t) \cdot v = \theta(g, \overline t) * v$. Let $a_G$ (resp.~$a_\ad$) denote the restriction of $a$ to $G$ (resp.~$T_\ad$).

Here are the most important features of the new actions.
\begin{enumerate}
\item
The action $a_\ad$ is opposite to the action~$a_\ad^*$. In particular, the set $\Phi(\Gamma)$ is exactly the set of weights of the $T_\ad$-module $\EuScript{TS}$ with respect to the action~$a_\ad$.

\item \label{feature2}
For every $\lambda \in \mathrm E$, the subspace $V(\lambda)^* \subset V$, regarded as a $G$-module with respect to the action~$a_G$, is isomorphic to $V(\lambda)$ with~$v_\lambda$, viewed in~$V(\lambda)$, being a highest weight vector.
\end{enumerate}

From now on, we shall consider the diagram~(\ref{eqn_just_diagram}) only with respect to the actions $a$, $a_G$, and~$a_\ad$. According to~(\ref{feature2}), this implies the following changes in our notation:
\begin{itemize}
\item
$V = \bigoplus \limits_{\lambda \in \mathrm E} V(\lambda)$;

\item
$v_\lambda$ is a highest weight vector of~$V(\lambda)$ for every $\lambda \in \mathrm E$.
\end{itemize}
With the above new notation, $a_G$ becomes the usual action of $G$ on~$V$ and the action $a_\ad$ of $T_\ad$ on~$V$ is given by
\begin{equation} \label{eqn_action_of_Tad_on_V}
\overline t \cdot v = t^\lambda (t^{-1} \cdot v) \;\; \text{ for all } \; t \in T, \; \lambda \in \mathrm E, \text{ and } v \in V(\lambda).
\end{equation}

For every $\lambda \in \mathrm E$, let $\pr_\lambda \colon V \to V(\lambda)$ be the canonical projection.

For every $v \in V$, let $[v]$ denote the image of $v$ under the
natural map $V \to V / \mathfrak g x_0$.

Since the subspace $\mathfrak gx_0 \subset V$ is $T_\ad$-invariant,
for every $T_\ad$-eigenvector $q \in V / \mathfrak g x_0$ there
exists a $T_\ad$-eigenvector $v \in V$ (of the same weight) such
that $[v] = q$. This observation will be always used in our study of
$(V / \mathfrak g x_0)^{G_{x_0}}$.

\subsection{The role of saturatedness of~$\Gamma$}
\label{subsec_saturatedness}

The saturatedness assumption on $\Gamma$ will be essential in our proof of Theorem~\ref{thm_tangent_space}. Firstly, by Proposition~\ref{prop_normality_saturatedness} this assumption guarantees that the variety $X_0$ is normal, which is essentially used in the proof of Proposition~\ref{prop_should_extend_to_oc1} in \S\,\ref{subsec_extension_of_sections}. Secondly, our arguments will often require the following crucial property of saturated~$\Gamma$.

\begin{lemma} \label{lemma_value_is_one}
For every $\varrho \in \mathcal K^1$, there exists $\mu \in \mathrm E$ such that $\langle \varrho, \mu \rangle = 1$.
\end{lemma}

\begin{proof}
It suffices to prove that $\lbrace \nu \in \Gamma \mid \langle \varrho, \nu \rangle = 1 \rbrace \ne \varnothing$. Since $\Gamma$ is saturated, one has
\begin{equation} \label{eqn_saturated}
\Gamma = \lbrace \nu \in \ZZ \Gamma \mid \langle \kappa, \nu \rangle
\ge 0 \text{ for all } \kappa \in \mathcal K^1 \rbrace.
\end{equation}
As $\varrho$ is primitive in~$\mathcal L$, there exists $\nu_0 \in \ZZ \Gamma$ with $\langle \varrho, \nu_0 \rangle = 1$. If $\mathcal K^1 = \lbrace \varrho \rbrace$ then $\nu_0 \in \Gamma$ by~(\ref{eqn_saturated}). Otherwise there exists an element $\eta \in \ZZ^+ \mathrm E_\varrho$ such that $\langle \kappa, \eta \rangle > 0$ for all $\kappa \in \mathcal K^1 \setminus \lbrace \varrho \rbrace$. For each $n \in \ZZ^+$ consider the element $\nu_n = \nu_0 + n\eta$. Clearly, $\langle \varrho, \nu_n \rangle = 1$ for all $n \in \ZZ^+$. In view of~(\ref{eqn_saturated}) one has $\nu_n \in \Gamma$ when $n$ is sufficiently large.
\end{proof}

\subsection{Basic properties
of~$(V / \mathfrak gx_0)^{G_{x_0}}$ and its $T_\ad$-weights}

The material presented in this subsection is more or less known.

The following lemma is obvious.

\begin{lemma} \label{lemma_G_x-fixed}
Suppose that $\mathrm E' \subset \mathrm E$ is a nonempty subset, $x = \sum \limits_{\lambda \in \mathrm E'} v_\lambda \in V$, and $A$ is a $G_x$-module. Then an element $a \in A$ is $G_x$-fixed if and only if the following two conditions hold:
\begin{enumerate}[label=\textup{(\arabic*)},ref=\textup{\arabic*}]
\item \label{lemma_G_x-fixed_1}
$a$ is $T_x$-stable;

\item \label{lemma_G_x-fixed_2}
$e_\delta a = 0$ for all $\delta \in \Delta^+ \cup (\Delta^- \cap \ZZ \mathrm E'^\perp)$.
\end{enumerate}
Moreover, condition~\textup(\ref{lemma_G_x-fixed_2}\textup) is equivalent to
\begin{enumerate}[label=\textup{(\arabic*$'$)},ref=\textup{\arabic*$'$}]
\setcounter{enumi}{1}
\item \label{lemma_G_x-fixed_3}
$e_\delta a = 0$ for all $\delta \in \Pi \cup (- \mathrm E'^\perp)$.
\end{enumerate}
\end{lemma}

\begin{lemma} \label{lemma_straightforward}
Suppose that $\sigma$ is a $T_\ad$-weight of $(V / \mathfrak gx_0)^{G_{x_0}}$. Then
\begin{enumerate}[label=\textup{(\alph*)},ref=\textup{\alph*}]
\item \label{lemma_straightforward_a}
$\sigma \in \ZZ^+ \Pi$;

\item \label{lemma_straightforward_b}
$\sigma \in \ZZ \Gamma$.
\end{enumerate}
\end{lemma}

\begin{proof}
(\ref{lemma_straightforward_a}) This follows from (\ref{eqn_action_of_Tad_on_V}) and basic properties of $T$-weights in a simple $G$-module.

(\ref{lemma_straightforward_b}) It suffices to show that $t^\sigma = 1$ for all $t \in T_{x_0}$. Assume the converse and take $t \in T_{x_0}$ such that $t^\sigma \ne 1$. Let $v \in V^{(T_\ad)}_\sigma$ be such that $[v] \in (V /\mathfrak gx_0)^{G_{x_0}} \setminus \lbrace 0 \rbrace$. As $t^\lambda = 1$ for all $\lambda \in \mathrm E$, one has $t \cdot v = t^{-\sigma} v$. Since $T_{x_0} \subset G_{x_0}$, it follows that $[v] = [t \cdot v] = t^{-\sigma}[v]$ and hence $[v] = 0$, a contradiction.
\end{proof}

\begin{lemma} \label{lemma_edeltav1}
Suppose that $\sigma$ is a nonzero $T_\ad$-weight of $V$ and $v \in V^{(T_\ad)}_\sigma \setminus \lbrace 0 \rbrace$. Then there exists $\delta \in \Pi$ such that $e_\delta v \ne 0$.
\end{lemma}

\begin{proof}
Assume that $e_\delta v = 0$ for all $\delta \in \Pi$. Then $v$ is a sum of highest weight vectors in~$V$. As $\sigma \ne 0$, it follows that $v = 0$, a contradiction.
\end{proof}

\begin{lemma} \label{lemma_edeltav2}
Let $\sigma$ be a $T_\ad$-weight of $(V / \mathfrak gx_0)^{G_{x_0}}$. Suppose that $v \in V^{(T_\ad)}_\sigma$ is such that $[v] \in (V / \mathfrak g x_0)^{G_{x_0}}$. Then $e_\delta v \in V^{(T_\ad)}_{\sigma - \delta} \cap \mathfrak gx_0$ for every $\delta \in \Delta^+ \cup (\Delta^- \cap \ZZ\mathrm E^\perp)$.
\end{lemma}

\begin{proof}
By Lemma~\ref{lemma_G_x-fixed} the condition $\delta \in \Delta^+ \cup (\Delta^- \cap \ZZ\mathrm E^\perp)$ implies $e_\delta v \in \mathfrak g x_0$. Clearly, $e_\delta v$ is a $T_\ad$-eigenvector of weight $\sigma - \delta$.
\end{proof}

\begin{lemma} \label{lemma_intersection}
Suppose that $\sigma \in \ZZ \Pi$. Then
\[
V^{(T_\ad)}_\sigma \cap \mathfrak gx_0 =
\begin{cases}
\Bbbk e_{-\sigma} x_0 & \text{ if } \sigma \in \Delta^+; \\
\mathfrak tx_0 & \text{ if } \sigma = 0; \\
\lbrace 0 \rbrace & \text{ otherwise.}
\end{cases}
\]
\end{lemma}

\begin{proof}
This follows from the decomposition $\mathfrak g = \mathfrak t \oplus \bigoplus \limits_{\delta \in \Delta} \Bbbk e_\delta$ and the fact that $e_\delta x_0 = 0$ for all $\delta \in \Delta^+$.
\end{proof}

\begin{corollary} \label{crl_sigma-root}
Let $\sigma$ be a nonzero $T_\ad$-weight of $(V / \mathfrak gx_0)^{G_{x_0}}$. Let $v \in V^{(T_\ad)}_\sigma$ be such that $[v] \in (V /\mathfrak g x_0)^{G_{x_0}} \setminus \lbrace 0 \rbrace$. Suppose that $\delta \in \Delta^+$ is such that $\delta \ne \sigma$ and $e_\delta v \ne 0$. Then
\begin{enumerate}[label=\textup{(\alph*)},ref=\textup{\alph*}]
\item \label{crl_sigma-root_a}
$\sigma - \delta \in \Delta^+$;

\item \label{crl_sigma-root_b}
$e_\delta v = c e_{-(\sigma - \delta)} x_0$ for some $c \in \Bbbk^\times$.
\end{enumerate}
\end{corollary}

\begin{proof}
This is a direct consequence of Lemmas~\ref{lemma_edeltav2} and~\ref{lemma_intersection}.
\end{proof}

\subsection{Extension of sections}
\label{subsec_extension_of_sections}

For every open subset $Y \subset X_0$, we let $\mathcal N_Y$ denote the restriction of the sheaf $\mathcal N_{X_0}$ to~$Y$.

\begin{proposition} \label{prop_should_extend_to_oc1}
For a section $s \in H^0(O, \mathcal N_O)$, the following conditions are equivalent:
\begin{enumerate}[label=\textup{(\arabic*)},ref=\textup{\arabic*}]
\item
$s$ extends to $X_0$.

\item \label{prop_should_extend_to_oc1_2}
$s$ extends to $O \cup O'$ for each $G$-orbit $O' \subset X_0$ of
codimension~$1$.
\end{enumerate}
\end{proposition}

\begin{proof}
Let $Y \subset X_0$ be the union of all $G$-orbits in $X_0$ of codimension at most $1$. Then $Y$ is an open $G$-stable subset of $X_0$ and $\codim_{X_0} (X_0 \setminus Y) \ge \nobreak 2$. As $X_0$ is normal, by~\cite[Lemma~3.9]{Br13} the restriction map $H^0(X_0, \mathcal N_{X_0}) \to H^0(Y, \mathcal N_Y)$ is an isomorphism. Thus a section $s \in H^0(O, \mathcal N_O)$ extends to $X_0$ if and only if it extends to~$Y$. But the latter is obviously equivalent to~(\ref{prop_should_extend_to_oc1_2}).
\end{proof}

To describe all $G$-orbits in $X_0$ of codimension~$1$, we need some additional notation. First of all, we introduce the set
\begin{equation} \label{eqn_set_P}
\mathcal P = \lbrace \varrho \in \mathcal K^1 \mid \mathrm E^\perp = \mathrm E_\varrho^\perp \rbrace.
\end{equation}
Note that every $\varrho \in \mathcal P$ is not proportional to an element of the form $\iota(\alpha^\vee)$ with $\alpha \in \Pi$. Next, for every $\varrho \in \mathcal P$ we consider the vector
\begin{equation} \label{eqn_z_varrho}
z_\varrho = \sum \limits_{\lambda \in \mathrm E_\varrho} v_\lambda \in V
\end{equation}
and its $G$-orbit $O_\varrho = Gz_\varrho$.

The following result is a consequence of~\cite[Theorems~8 and~9]{VP72}.

\begin{proposition} \label{prop_orbits_of_codimension_1}
The map $\varrho \mapsto O_\varrho$ is a bijection between the set $\mathcal P$ and the $G$-orbits in $X_0$ of codimension~$1$.
\end{proposition}

\begin{corollary} \label{crl_[v]_in_TS}
Suppose that $\sigma$ is a nonzero $T_\ad$-weight of~$V$ and $v \in V^{(T_\ad)}_\sigma$. Let $s \in H^0(O, \mathcal N_O)^G$ be the section defined by $s(x_0) = [v]$. Then the following conditions are equivalent:
\begin{enumerate}[label=\textup{(\arabic*)},ref=\textup{\arabic*}]
\item
$[v] \in \EuScript{TS}$.

\item
$s$ extends to $O \cup O_\varrho$ for all $\varrho \in \mathcal P$.
\end{enumerate}
\end{corollary}

\begin{proof}
This follows from the definition of~$\EuScript{TS}$ (see~\S\,\ref{subsec_preliminaries_for_the_proof}) along with Propositions~\ref{prop_should_extend_to_oc1} and~\ref{prop_orbits_of_codimension_1}.
\end{proof}

In what follows, for every $\varrho \in \mathcal P$ we regard the sheaf $\mathcal N_{O \cup O_\varrho}$ as the sheaf of sections of the normal bundle of $O \cup O_\varrho$ in~$V$. We denote the total space of this bundle by $F_\varrho$ and let
\[
p_\varrho \colon F_\varrho \to O \cup O_\varrho
\]
be the canonical projection.

Fix an arbitrary element $\varrho \in \mathcal P$ and let $\phi_\varrho \colon \Bbbk^\times \to T$ be the one-parameter subgroup of $T$ corresponding to~$\varrho$, that is, $(\phi_\varrho(\xi))^\chi = \xi^{\langle \varrho, \chi \rangle}$ for all $\chi \in \mathfrak X(T)$ and $\xi \in \Bbbk^\times$. For every $\xi \in \Bbbk$, consider the vector $z_\varrho(\xi) \in V$ given by
\[
z_\varrho(\xi) =
\begin{cases}
z_\varrho & \text{ if } \xi = 0; \\
\phi_\varrho(\xi) x_0 & \text{ otherwise.}
\end{cases}
\]
Then we have $z_\varrho(\xi) = z_\varrho + \sum \limits_{\lambda \in \mathrm E \setminus \mathrm E_\varrho} \xi^{\langle \varrho, \lambda \rangle} v_\lambda$ for all $\xi \in \Bbbk$. Note that $z_\varrho(1) = x_0$. It follows from Lemma~\ref{lemma_value_is_one} that the morphism $\Bbbk \to O \cup O_\varrho$ given by $\xi \mapsto z_\varrho(\xi)$ is a closed immersion; we denote its image by~$Z_\varrho$.

\begin{lemma} \label{lemma_tangent_space}
Suppose that $\varrho \in \mathcal P$. Then $T_{z_\varrho} X_0 = \mathfrak gz_\varrho \oplus \Bbbk u_\varrho$ where
\begin{equation} \label{eqn_u_varrho}
u_\varrho = \sum \limits_{\mu \in \mathrm E : \langle \varrho, \mu \rangle = 1} v_\mu.
\end{equation}
\end{lemma}

\begin{proof}
We have $T_{z_\varrho}(Gz_\varrho) = \mathfrak gz_\varrho$ and $T_{z_\varrho}Z_\varrho = \Bbbk u_\varrho$. Since $Gz_\varrho$ has codimension~$1$ in~$X_0$ and $X_0$ is normal, it follows that $\dim \mathfrak g z_\varrho = \dim X_0 - 1$ and $z_\varrho$ is a regular point of~$X_0$. The proof is completed by observing that $u_\varrho \notin \mathfrak gz_\varrho$.
\end{proof}

\begin{proposition} \label{prop_extension_of_sections}
Let $v \in V$ be such that $[v] \in (V /\mathfrak gx_0)^{G_{x_0}}$ and let $s \in H^0(O, \mathcal N_{O})^G$ be the section defined by $s(x_0) = [v]$. Given $\varrho \in \mathcal P$, the following conditions are equivalent:
\begin{enumerate}[label=\textup{(\arabic*)},ref=\textup{\arabic*}]
\item \label{Ext1}
The section $s$ extends to $O \cup O_\varrho$.

\item \label{Ext2}
There exists $\lim \limits_{\xi \to 0} s(z_\varrho(\xi))$, that is, the restriction of $s$ to $O \cap Z_\varrho$ extends to~$Z_\varrho$.
\end{enumerate}
\end{proposition}

\begin{proof}
Obviously, (\ref{Ext1}) implies~(\ref{Ext2}). It remains to prove the converse implication. Put $s(z_\varrho) = \lim \limits_{\xi \to 0} s(z_\varrho(\xi))$. Being in the closure of the $\phi_\varrho (\Bbbk^\times)$-orbit of~$s(x_0)$, the point $s(z_\varrho)$ is $\phi_\varrho(\Bbbk^\times)$-fixed. Further, observe that $G_{x_0} = G_{z_\varrho(\xi)}$ for all $\xi \in \Bbbk^\times$, which implies that every point $s(z_\varrho(\xi))$ with $\xi \in \Bbbk^\times$ is $G_{x_0}$-fixed. Consequently, the point~$s(z_\varrho)$ is $G_{x_0}$-fixed. Next, as $\QQ^+ \mathrm E_\varrho$ is a face of codimension~$1$ of~$\QQ^+ \Gamma$ and $\Gamma$ is saturated, it follows that $\ZZ \Gamma \cap \Ker \varrho = \ZZ \mathrm E_\varrho$, which implies $T_{z_\varrho} = T_{x_0} \cdot \phi_\varrho(\Bbbk^\times)$. Combining this and the equality $\mathrm E^\perp = \mathrm E_\varrho^\perp$ with Lemma~\ref{lemma_G_x-fixed} we obtain $G_{z_\varrho} = G_{x_0} \cdot \phi_\varrho(\Bbbk^\times)$, hence $s(z_\varrho)$ is $G_{z_\varrho}$-fixed. The latter enables us to extend $s$ to $O_\varrho$ by the formula $s(gz_\varrho) = g(s(z_\varrho))$.

To complete the proof it remains to show that the extended map $s \colon O \cup O_\varrho \to F_\varrho$ is a morphism. First, as $s(z_\varrho) = \lim \limits_{\xi \to 0} s(z_\varrho(\xi))$, it follows that the $G$-orbit $Gs(z_\varrho)$ is contained in the closure of the $G$-orbit $Gs(x_0)$ in~$F_\varrho$. Next, since the map $p_\varrho$ is $G$-equivariant and the point $s(x_0)$ is $G_{x_0}$-fixed, it follows that $G_{s(x_0)} = G_{x_0}$ and hence $Gs(x_0) \simeq \nobreak O$. Similarly, $Gs(z_\varrho) \simeq O_\varrho$. In particular, $\dim Gs(z_\varrho) = \dim Gs(x_0) - 1$, and so the set $Gs(x_0) \cup Gs(z_\varrho)$ is open in $\overline{Gs(x_0)}$. Now the restriction of~$p_\varrho$ to $Gs(x_0) \cup Gs(z_\varrho)$ is a bijective morphism onto $O \cup O_\varrho$, whence an isomorphism since $O \cup O_\varrho$ is smooth. Consequently, our map $s \colon O \cup O_\varrho \to Gs(x_0) \cup Gs(z_\varrho)$ is a morphism as required.
\end{proof}

The next proposition is an application of the previous one.

\begin{proposition} \label{prop_section_extensions}
Suppose that $\varrho \in \mathcal P$ and $\sigma \in \ZZ^+ \Pi$. Let $v \in V^{(T_\ad)}_\sigma$ be such that $[v] \in (V / \mathfrak gx_0)^{G_{x_0}} \setminus \lbrace 0 \rbrace$ and let $s \in H^0(O, \mathcal N_{O})^G$ be the section defined by $s(x_0) = [v]$.
\begin{enumerate}[label=\textup{(\alph*)},ref=\textup{\alph*}]
\item \label{prop_section_extensions_a}
If $\langle \varrho, \sigma \rangle > 0$ and $\sum \limits_{\lambda \in \mathrm E_\varrho} \pr_\lambda (v) \notin \mathfrak gz_\varrho$, then $s$ does not extend to~$O \cup O_\varrho$.

\item \label{prop_section_extensions_b}
If $\langle \varrho, \sigma \rangle \le 0$ then $s$ extends to $O \cup O_\varrho$.

\item \label{prop_section_extensions_c}
If $\sum \limits_{\lambda \in \mathrm E_\varrho} \pr_\lambda (v) = 0$ and there exists $\nu \in \mathrm E \setminus \mathrm E_\varrho$ with $\pr_\nu (v) \ne 0$ and $\langle \varrho, \sigma \rangle > \langle \varrho, \nu \rangle$, then $s$ does not extend to $O \cup O_\varrho$.

\item \label{prop_section_extensions_d}
If $\sum \limits_{\lambda \in \mathrm E_\varrho} \pr_\lambda (v) = 0$ and $\langle \varrho, \sigma \rangle = 1$ then $s$ extends to $O \cup O_\varrho$.
\end{enumerate}
\end{proposition}

\begin{proof}
Thanks to Proposition~\ref{prop_extension_of_sections}, in all the cases it is enough to prove the corresponding statement about the existence of $\lim \limits_{\xi \to 0} s(z_\varrho(\xi))$. Before we proceed, let us make some preparations.

Given $w \in V_\sigma^{(T_\ad)}$ and~$\xi \in \Bbbk$, let $[w]_\xi$ denote the image of~$w$ in $V / T_{z_\varrho(\xi)}X_0$. Then for every $\xi \in \Bbbk^\times$ one has
\[
s(z_\varrho(\xi)) = s(\phi_\varrho(\xi) x_0) = \xi^{-\langle \varrho, \sigma \rangle} [\sum \limits_{\lambda \in \mathrm E_\varrho} \pr_\lambda (v)]_\xi + \sum \limits_{\mu \in \mathrm E \setminus \mathrm E_\varrho}\xi^{\langle \varrho, \mu \rangle - \langle \varrho, \sigma \rangle}[\pr_\mu (v)]_\xi.
\]
Let $\mathcal N_{Z_\varrho}$ denote the restriction of the sheaf $\mathcal N_{O \cup O_\varrho}$ to~$Z_\varrho$.

(\ref{prop_section_extensions_a}) Consider the section $s' \in H^0(Z_\varrho, \mathcal N_{Z_\varrho})$ given by
\[
s'(z_\varrho(\xi)) = [\sum \limits_{\lambda \in \mathrm E_\varrho} \pr_\lambda(v)]_\xi + \sum \limits_{\mu \in \mathrm E \setminus \mathrm E_\varrho} \xi^{\langle \varrho, \mu \rangle}[\pr_\mu (v)]_\xi.
\]
Clearly, $s'(z_\varrho) = [ \sum \limits_{\lambda \in \mathrm E_\varrho} \pr_\lambda(v)]_0$. It follows from Lemma~\ref{lemma_tangent_space} that $(\bigoplus \limits_{\lambda \in \mathrm E_\varrho} V(\lambda)) \cap T_{z_\varrho} X_0 = \mathfrak g z_\varrho$, whence the condition $\sum \limits_{\lambda \in \mathrm E_\varrho} \pr_\lambda (v) \notin \mathfrak gz_\varrho$ implies $s'(z_\varrho) \ne 0$. On the other hand, one has $s(z_\varrho(\xi)) = \xi^{-\langle \varrho, \sigma \rangle}s'(z_\varrho(\xi))$ for all $\xi \in \Bbbk^\times$. Since $\langle \varrho, \sigma \rangle > 0$, it follows that $\lim \limits_{\xi \to 0} s(z_\varrho(\xi))$ does not exist.

(\ref{prop_section_extensions_b}) It is easy to see that $\lim \limits_{\xi \to 0} s(z_\varrho(\xi))$ exists and is given by
\[
\lim \limits_{\xi \to 0} s(z_\varrho(\xi)) =
\begin{cases}
\sum \limits_{\lambda \in \mathrm E_\varrho} [\pr_\lambda(v)]_0 &
\text{ if }
\langle \varrho, \sigma \rangle = 0;\\
0 & \text{ if } \langle \varrho, \sigma \rangle < 0.
\end{cases}
\]

(\ref{prop_section_extensions_c}) We may assume that $\langle \varrho, \nu \rangle \le \langle \varrho, \mu \rangle$ for all $\mu \in \mathrm E \setminus \mathrm E_\varrho$ with $\pr_\mu(v) \ne 0$. Consider the section $s' \in H^0(Z_\varrho, \mathcal N_{Z_\varrho})$ given by
\[
s'(z_\varrho(\xi)) =
\sum \limits_{\mu \in \mathrm E \setminus \mathrm E_\varrho : \langle \varrho, \mu \rangle = \langle \varrho, \nu \rangle} [\pr_\mu(v)]_\xi +
\sum \limits_{\mu \in \mathrm E \setminus \mathrm E_\varrho : \langle \varrho, \mu \rangle > \langle \varrho, \nu \rangle} \xi^{\langle \varrho, \mu \rangle - \langle \varrho, \nu \rangle}[\pr_\mu(v)]_\xi.
\]
Clearly, $s'(z_\varrho) = \sum \limits_{\mu \in \mathrm E \setminus \mathrm E_\varrho : \langle \varrho, \mu \rangle = \langle \varrho, \nu \rangle} [\pr_\mu(v)]_0$. Since $\langle \varrho, \sigma \rangle > \langle \varrho, \nu \rangle > 0$, it follows that $\sigma \ne 0$ and hence
\[
\sum \limits_{\mu \in \mathrm E \setminus \mathrm E_\varrho : \langle \varrho, \mu \rangle = \langle \varrho, \nu \rangle} \pr_\mu(v) \notin \Bbbk u_\varrho.
\]
As $\Bbbk u_\varrho = (\bigoplus \limits_{\mu \in \mathrm E \setminus \mathrm E_\varrho} V(\mu)) \cap T_{z_\varrho} X_0$ by Lemma~\ref{lemma_tangent_space}, we find that $s'(z_\varrho) \ne 0$. On the other hand, we have $s(z_\varrho(\xi)) = \xi^{\langle \varrho, \nu \rangle - \langle \varrho, \sigma \rangle} s'(z_\varrho(\xi))$ for all $\xi \in \Bbbk^\times$. Since $\langle \varrho, \nu \rangle < \langle \varrho, \sigma \rangle$, it follows that $\lim \limits_{\xi \to 0} s(z_\varrho(\xi))$ does not exist.

(\ref{prop_section_extensions_d}) Clearly, $\lim \limits_{\xi \to 0} s(z_\varrho(\xi))$ exists and equals $\sum \limits_{\mu \in \mathrm E : \langle \varrho, \mu \rangle = 1} [\pr_\mu(v)]_0$.
\end{proof}

\begin{remark}
In \cite[Theorem~2.8]{PvS16} one can find a specialization of our Proposition~\ref{prop_section_extensions} to the case where $\Gamma$ is free.
\end{remark}

\subsection{Canonical representatives of $T_\ad$-eigenvectors in~$\EuScript{TS}$}
\label{subsec_canonical_representatives}

The main results of this subsection are Propositions~\ref{prop_canonical_form} and~\ref{prop_role_of_NE2}.

\begin{lemma} \label{lemma_non-acute}
For every $\sigma \in \ZZ^+ \Pi \setminus \lbrace 0 \rbrace$ there exists $\delta \in \Supp \sigma$ such that $\langle \delta^\vee, \sigma \rangle > 0$.
\end{lemma}

\begin{proof}
Assuming the converse we find that the angle between any two distinct elements of the set $\lbrace \sigma \rbrace \cup \Supp \sigma$ is non-acute. Since the latter set is contained in a half-space of the $\QQ$-vector space spanned by~$\Supp \sigma$, the elements in $\lbrace \sigma \rbrace \cup \Supp \sigma$ have to be linearly independent, which is not the case.
\end{proof}

Recall from (\ref{eqn_K1(sigma)}) the set $\mathcal K^1(\sigma) \subset \mathcal K^1$ defined for every $\sigma \in \ZZ \Gamma$. As $\Phi(\Gamma) \subset \ZZ \Gamma$ by Lemma~\ref{lemma_straightforward}(\ref{lemma_straightforward_b}), the set $\mathcal K^1(\sigma)$ is also defined for every $\sigma \in \Phi(\Gamma)$.

\begin{lemma} \label{lemma_K1sigma_ne_0}
Suppose that $\sigma \in \ZZ \Gamma \cap (\ZZ^+ \Pi \setminus \lbrace 0 \rbrace)$. Then $\mathcal K^1(\sigma) \ne \varnothing$.
\end{lemma}

\begin{proof}
By Lemma~\ref{lemma_non-acute} there exists $\delta \in \Pi$ with $\langle \delta^\vee, \sigma \rangle > 0$. Now assume $\mathcal K^1(\sigma) = \varnothing$. Then $\langle \varrho, \sigma \rangle \le 0$ for all $\varrho \in \mathcal K^1$, hence $\sigma \in -\Gamma$. The latter yields $\langle \delta^\vee, \sigma \rangle \le 0$, a~contradiction.
\end{proof}

\begin{corollary} \label{crl_K1sigma_ne_0}
Suppose that $\sigma \in \Phi(\Gamma)$. Then $\mathcal K^1(\sigma) \ne \varnothing$.
\end{corollary}

\begin{proof}
This follows from Corollary~\ref{crl_weights_are_nonzero} along with Lemmas~\ref{lemma_straightforward} and~\ref{lemma_K1sigma_ne_0}.
\end{proof}

Recall the subset $\mathcal P \subset \mathcal K^1$ given by~(\ref{eqn_set_P}).

\begin{lemma} \label{lemma_positive_multiple}
Suppose that $\varrho \in \mathcal K^1 \setminus \mathcal P$. Then there exists $\delta \in \Pi$ such that $\iota(\delta^\vee)$ is a positive multiple of~$\varrho$.
\end{lemma}

\begin{proof}
Since $\varrho \notin \mathcal P$, it follows that $\mathrm E^\perp \ne \mathrm E_\varrho^\perp$. Then there exist $\delta \in \Pi$ and $\mu \in \mathrm E \setminus \mathrm E_\varrho$ such that $\langle \delta^\vee, \mu \rangle > 0$ and $\langle \delta^\vee, \lambda \rangle = 0$ for all $\lambda \in \mathrm E_\varrho$. Obviously, $\delta$ possesses the required property.
\end{proof}

\begin{proposition} \label{prop_canonical_form}
Let $\sigma \in \Phi(\Gamma)$, $\varrho \in \mathcal K^1(\sigma)$, $v \in V^{(T_\ad)}_\sigma$, and $[v] \in \EuScript{TS} \setminus \lbrace 0 \rbrace$. \begin{enumerate}[label=\textup{(\alph*)},ref=\textup{\alph*}]
\item \label{prop_canonical_form_a}
If $\sigma \in \Delta^+$ then there exist $v' \in V^{(T_\ad)}_\sigma$ and $c \in \Bbbk$ such that $v' = v - ce_{-\sigma}x_0$ and $\pr_\lambda(v') = 0$ for all $\lambda \in \mathrm E_\varrho$.

\item \label{prop_canonical_form_b}
If $\sigma \notin \Delta^+$ then $\pr_\lambda(v) = 0$ for all $\lambda \in \mathrm E_\varrho$.
\end{enumerate}
\end{proposition}

\begin{proof}
Note that $\sigma \ne 0$ by Corollary~\ref{crl_weights_are_nonzero}. Recall the vector $z_\varrho$ given by~(\ref{eqn_z_varrho}) and set $w = \sum \limits_{\lambda \in \mathrm E_\varrho} \pr_\lambda(v)$. We consider two cases.

\textit{Case}~1: $\varrho \in \mathcal P$. Since $[v] \in \EuScript{TS}$, it follows from Corollary~\ref{crl_[v]_in_TS} and Proposition~\ref{prop_section_extensions}(\ref{prop_section_extensions_a}) that $w \in \mathfrak g z_\varrho$. Applying an analogue of Lemma~\ref{lemma_intersection} for $V^{(T_\ad)}_\sigma \cap \mathfrak gz_\varrho$, we obtain the following:
\begin{itemize}
\item
if $\sigma \notin \Delta^+$ then $w = 0$;

\item
if $\sigma \in \Delta^+$ then $w = c e_{-\sigma}z_\varrho$ for some $c \in \Bbbk$.
\end{itemize}
In the latter case, the vector $v' = v - ce_{-\sigma}x_0$ satisfies $\pr_\lambda(v') = 0$ for all $\lambda \in \mathrm E_\varrho$.

\textit{Case}~2: $\varrho \notin \mathcal P$. Assume that $w \ne 0$. By Lemma~\ref{lemma_positive_multiple}, there exists $\delta \in \Pi$ such that $\iota(\delta^\vee)$ is a positive multiple of~$\varrho$. Then $\langle \delta^\vee, \sigma \rangle > 0$ and $\langle \delta^\vee, \lambda \rangle = 0$ for all $\lambda \in \mathrm E_\varrho$. If $\sigma = \delta$ then $\pr_\lambda(v) \in \Bbbk e_{-\delta} v_\lambda = \lbrace 0 \rbrace$ for every $\lambda \in \mathrm E_\varrho$, which contradicts the assumption $w \ne 0$. So in what follows we assume that $\sigma \ne \delta$. We have $\langle \delta^\vee, \lambda - \sigma \rangle = -\langle \delta^\vee, \sigma \rangle <\nobreak 0$ for all $\lambda \in \mathrm E_\varrho$, therefore $e_\delta w \ne 0$ and hence $e_\delta v \ne 0$. Corollary~\ref{crl_sigma-root} implies that $\sigma - \delta \in \Delta^+$ and $e_\delta v = ce_{-(\sigma - \delta)}x_0$ for some $c \in \Bbbk^\times$. In particular, $e_\delta \pr_\lambda(v) = c e_{-(\sigma - \delta)}v_\lambda$ for all $\lambda \in \mathrm E$.

Let $\mathfrak h \simeq \mathfrak{sl}_2$ be the Lie subalgebra of $\mathfrak g$ generated by $e_\delta$ and $e_{-\delta}$. Fix $\lambda \in \mathrm E_\varrho$ such that $\pr_\lambda(v) \ne 0$. Let $R^\lambda \subset V(\lambda)$ be the $\mathfrak h$-submodule generated by $e_{-(\sigma - \delta)}v_\lambda$. Since $\langle \delta^\vee, \lambda \rangle = 0$, it follows that $R^\lambda$ is a simple $\mathfrak h$-module with highest weight $2l - \langle \delta^\vee, \sigma \rangle$, where $l$ is the maximal integer such that $\sigma - l \delta \in \Delta^+$. Note that $\pr_\lambda(v) \in R^\lambda$ since otherwise the inequality $\langle \delta^\vee, \lambda - \sigma \rangle < 0$ would imply $e_\delta \pr_\lambda(v) \notin R^\lambda$, which is not the case. We conclude that $\pr_\lambda(v) = de_{-\delta}e_\delta \pr_\lambda(v)$ for some scalar $d \in \Bbbk^\times$ that depends only on $\sigma$ and~$\delta$ (and not on~$\lambda$).

It follows from the previous paragraph that
\[
w = cd\sum \limits_{\lambda \in \mathrm E_\varrho} e_{-\delta}e_{-(\sigma - \delta)}v_\lambda = cd\sum \limits_{\lambda \in \mathrm E_\varrho} [e_{-\delta}, e_{-(\sigma - \delta)}]v_\lambda.
\]
Recall that $w \ne 0$, therefore $\sigma \in \Delta^+$ and $w = c' \sum \limits_{\lambda \in \mathrm E_\varrho} e_{-\sigma} v_\lambda$ for some $c' \in \Bbbk^\times$. Now the vector $v' = v - c'e_{-\sigma} x_0$ satisfies $\pr_\lambda(v) = 0$ for all $\lambda \in \mathrm E_\varrho$. Since the assumption $w \ne 0$ implies $\sigma \in \Delta^+$, the proof is completed.
\end{proof}

\begin{lemma} \label{lemma_v_alpha}
Under the assumptions of Proposition~\textup{\ref{prop_canonical_form}}, suppose in addition that $\sigma = \alpha \in \Pi$ and $\pr_\lambda(v) = 0$ for all $\lambda \in \mathrm E_\varrho$. Then
\begin{enumerate}[label=\textup{(\alph*)},ref=\textup{\alph*}]
\item \label{lemma_v_alpha_a}
$\langle \alpha^\vee, \mu \rangle > 0$ for all $\mu \in \mathrm E \setminus \mathrm E_\varrho$;

\item \label{lemma_v_alpha_b}
there exists $c \in \Bbbk^\times$ such that
\[
v = c \sum \limits_{\mu \in \mathrm E \setminus \mathrm E_\varrho} \frac{\langle \varrho, \mu \rangle}{\langle \alpha^\vee, \mu \rangle} e_{-\alpha}v_\mu.
\]
\end{enumerate}
\end{lemma}

\begin{proof}
Consider the expression $v = \sum \limits_{\mu \in \mathrm E \setminus \mathrm E_\varrho} c_\mu e_{-\alpha} v_\mu$, where $c_\mu \in \Bbbk$ for all $\mu \in \mathrm E \setminus \mathrm E_\varrho$. Combining Lemma~\ref{lemma_edeltav1} with Corollary~\ref{crl_sigma-root}(\ref{crl_sigma-root_a}) yields $e_\alpha v \ne 0$. It then follows from Lemmas~\ref{lemma_edeltav2} and~\ref{lemma_intersection} that $e_\alpha v = yx_0$ for some $y \in \mathfrak t$. In particular, for every $\lambda \in \mathrm E_\varrho$ the condition $\pr_\lambda(v) = 0$ implies $\lambda(y) = 0$. Therefore the restriction of~$y$ (regarded as an element of $\mathcal Q \otimes_\ZZ \Bbbk$) to $\ZZ \Gamma \otimes_\ZZ \Bbbk$ is proportional to~$\varrho$, and so
\begin{equation} \label{eqn_one}
e_\alpha v = c\sum \limits_{\mu \in \mathrm E \setminus \mathrm E_\varrho} \langle \varrho, \mu \rangle v_\mu
\end{equation}
for some $c \in \Bbbk^\times$. On the other hand, one has
\begin{equation} \label{eqn_two}
e_\alpha v = \sum \limits_{\mu \in \mathrm E \setminus \mathrm E_\varrho} c_\mu e_\alpha e_{-\alpha} v_\mu = \sum \limits_{\mu \in \mathrm E \setminus \mathrm E_\varrho} c_\mu \langle \alpha^\vee, \mu \rangle v_\mu.
\end{equation}
Comparing~(\ref{eqn_one}) with~(\ref{eqn_two}) we obtain the required results.
\end{proof}

\begin{proposition}\label{prop_role_of_NE2}
Suppose that $\sigma \in \Phi(\Gamma)$ and $\varrho \in \mathcal K^1(\sigma)$. Then the following conditions are equivalent:
\begin{enumerate}[label=\textup{(\arabic*)},ref=\textup{\arabic*}]
\item \label{NE2holds}
$\varrho \in \mathcal P$.

\item \label{sigma_is_simple}
$\sigma \in \Pi$.
\end{enumerate}
\end{proposition}

\begin{proof}
Let $v \in V^{(T_\ad)}_\sigma$ be such that $[v] \in \EuScript{TS} \setminus \lbrace 0 \rbrace$. Taking into account Proposition~\ref{prop_canonical_form}, we may assume that $\pr_\lambda(v) = 0$ for all $\lambda \in \mathrm E_\varrho$.

(\ref{NE2holds})$\Rightarrow$(\ref{sigma_is_simple}) Thanks to Lemma~\ref{lemma_edeltav1} there exists $\alpha \in \Pi$ such that $e_\alpha v \ne 0$. Assume that $\sigma - \alpha \ne 0$. Then Corollary~\ref{crl_sigma-root} implies $\sigma - \alpha \in \Delta^+$ and $e_{\alpha}v = c e_{-(\sigma - \alpha)} x_0$ for some $c \in \Bbbk^\times$. It follows that $\langle (\sigma - \alpha)^\vee, \lambda \rangle = 0$ for all $\lambda \in \mathrm E_\varrho$ and $\langle (\sigma - \alpha)^\vee, \mu \rangle > 0$ for some $\mu \in \mathrm E \setminus \mathrm E_\varrho$. Consequently, $\Supp(\sigma - \alpha) \subset \mathrm E_\varrho^\perp$ and $\Supp(\sigma - \alpha) \not\subset \mathrm E^\perp$, which contradicts~(\ref{NE2holds}). Thus $\sigma = \alpha$.

(\ref{sigma_is_simple})$\Rightarrow$(\ref{NE2holds}) Let $\sigma = \alpha \in \Pi$ and assume that $\varrho \notin \mathcal P$. By Lemma~\ref{lemma_positive_multiple} there exists $\delta \in \Pi$ such that $\iota(\delta^\vee)$ is a positive multiple of~$\varrho$. Then $\langle \delta^\vee, \alpha \rangle > 0$ and hence $\delta = \alpha$. Applying Lemma~\ref{lemma_v_alpha}(\ref{lemma_v_alpha_b}) we obtain $v = c\sum \limits_{\mu \in \mathrm E \setminus \mathrm E_\varrho} e_{-\alpha} v_\mu = ce_{-\alpha}x_0$ for some $c \in \Bbbk^\times$, hence $v \in \mathfrak gx_0$ and $[v] = 0$, a~contradiction.
\end{proof}

\subsection{Proof of Theorem~\ref{thm_tangent_space}: Step~1}
\label{subsec_step_1}

The goal of this subsection is to show that every weight $\sigma \in \Phi(\Gamma)$ satisfies conditions (\ref{Phi1})--(\ref{Phi8}) along with the following one:
\begin{enumerate}
\renewcommand{\labelenumi}{(MF)}
\renewcommand{\theenumi}{MF}

\item \label{MF}
the multiplicity of $\sigma$ in $\EuScript{TS}$ equals~$1$.
\end{enumerate}

For the rest of this subsection, we fix a weight $\sigma \in \Phi(\Gamma)$ and a vector $v \in V^{(T_\ad)}_\sigma$ such that $[v] \in \EuScript{TS} \setminus \lbrace 0 \rbrace$. Recall the set $\mathcal K^1(\sigma)$ given by~(\ref{eqn_K1(sigma)}), which is nonempty by Corollary~\ref{crl_K1sigma_ne_0}.

Property~(\ref{Phi1}) has already been established in Lemma~\ref{lemma_straightforward}(\ref{lemma_straightforward_b}).

\begin{proof}[Proof of~\textup(\ref{Phi7}\textup)] Suppose that $\sigma \in \Phi(\Gamma) \setminus \Pi$ and take any $\varrho \in \mathcal K^1(\sigma)$. Then Proposition~\ref{prop_role_of_NE2} yields $\varrho \notin \mathcal P$. By Lemma~\ref{lemma_positive_multiple}, there exists $\delta \in \Pi$ such that $\iota(\delta^\vee)$ is a positive multiple of~$\varrho$. Clearly, $\delta \notin \Gamma^\perp$.
\end{proof}

It remains to establish properties (\ref{Phi2})--(\ref{Phi6}), (\ref{Phi8}), and~(\ref{MF}). We consider four cases.

\subsubsection{Case $\sigma = \alpha \in \Pi$}

Properties (\ref{Phi2}) and (\ref{Phi3}) hold automatically. Property (\ref{MF}) follows from Proposition~\ref{prop_canonical_form}(\ref{prop_canonical_form_a}) and Lemma~\ref{lemma_v_alpha}(\ref{lemma_v_alpha_b}). It remains to prove~(\ref{Phi8}).

\begin{lemma} \label{lemma_value_is_one_bis}
Suppose that $\varrho \in \mathcal K^1(\alpha)$. Then $\langle \varrho, \alpha \rangle = 1$.
\end{lemma}

\begin{proof}
By Proposition~\ref{prop_canonical_form}(\ref{prop_canonical_form_a}) we may assume that $\pr_\lambda(v) = 0$ for all $\lambda \in \mathrm E_\varrho$. Note that $\varrho \in \mathcal P$ by Proposition~\ref{prop_role_of_NE2}. Then it follows from Corollary~\ref{crl_[v]_in_TS} that the section $s \in H^0(O, \mathcal N_O)^G$ given by $s(x_0) = [v]$ extends to $O \cup O_\varrho$. Taking into account Lemmas~\ref{lemma_v_alpha} and~\ref{lemma_value_is_one} along with Proposition~\ref{prop_section_extensions}(\ref{prop_section_extensions_c}), we get $\langle \varrho, \alpha \rangle = 1$.
\end{proof}

\begin{lemma} \label{lemma_not_more_than_two}
There are inequalities $1 \le |\mathcal K^1(\alpha)| \le 2$.
\end{lemma}

\begin{proof}
As $\mathcal K^1(\alpha)$ is nonempty, we have $|\mathcal K^1(\alpha)| \ge 1$. To prove the second inequality, assume that $\varrho_1, \varrho_2, \varrho_3 \in \mathcal K^1(\alpha)$ are three distinct elements. Since $\QQ^+\varrho_i$ is an extremal ray of~$\mathcal K$ for each $i = 1,2,3$, the elements $\varrho_1, \varrho_2, \varrho_3$ are linearly independent in~$\mathcal Q$. By Proposition~\ref{prop_canonical_form}(\ref{prop_canonical_form_a}), for each $i = 1,2,3$ there exist $v_i \in V^{(T_\ad)}_\alpha$ and $c_i \in \Bbbk$ such that $v_i = v - c_i e_{-\alpha} x_0$ and $\pr_\lambda(v_i) = 0$ for all $\lambda \in \mathrm E_{\varrho_i}$. In view of Lemma~\ref{lemma_v_alpha}(\ref{lemma_v_alpha_b}), for each $i = 1,2,3$ there exists $c'_i \in \Bbbk^\times$ such that $e_\alpha v_i = c'_i \sum \limits_{\mu \in \mathrm E} \langle \varrho_i, \mu \rangle v_\mu$. Obviously, the vectors $e_\alpha v_1$, $e_\alpha v_2$, and $e_\alpha v_3$ are linearly independent in~$V$, hence so are the vectors $v_1$, $v_2$, and $v_3$. The latter contradicts the fact that $v_1, v_2, v_3$ belong to the linear span of the two vectors $v$ and $e_{-\alpha} x_0$.
\end{proof}

\begin{proof}[Proof of~\textup(\ref{Phi8}\textup)]
According to Lemma~\ref{lemma_not_more_than_two}, we
consider two cases.

\textit{Case}~1: $\mathcal K^1(\alpha)$ contains a unique element~$\varrho_0$. Then $\langle \varrho_0, \alpha \rangle = 1$ by Lemma~\ref{lemma_value_is_one_bis} and $\langle \varrho, \alpha \rangle \le 0$ for all $\varrho \in \mathcal K^1 \setminus \lbrace \varrho_0 \rbrace$. Put $\varrho_1 = \varrho_0$ and $\varrho_2 = \iota(\alpha^\vee) - \varrho_0$. Proposition~\ref{prop_role_of_NE2} yields $\varrho_0 \in \mathcal P$, hence $\varrho_0$ is not proportional to~$\iota(\alpha^\vee)$ and so $\varrho_1 \ne \varrho_2$. Further, $\varrho_1, \varrho_2$ obviously satisfy conditions~(\ref{Phi8})(\ref{Phi8a}--\ref{Phi8c}). To complete the proof, it suffices to show that $\varrho_2 \in \mathcal K$. For that, take any $\mu \in \mathrm E \setminus \mathrm E_{\varrho_0}$. Clearly, there is a unique expression $\alpha = \tau + b \mu$ where $\tau \in \QQ\mathrm E_{\varrho_0}$ and $b \in \QQ$. Since $\langle \varrho_0, \alpha \rangle = 1$, one has $b = 1/ \langle \varrho_0, \mu \rangle$. Then $\tau = \alpha - \mu / \langle \varrho_0, \mu \rangle$. One easily checks that $\langle \varrho, \tau \rangle \le 0$ for all $\varrho \in \mathcal K^1 \setminus \lbrace \varrho_0 \rbrace$, hence there is an expression $\tau = - \sum \limits_{\lambda \in \mathrm E_{\varrho_0}} c_\lambda \lambda$ with $c_\lambda \in \QQ^+$ for all $\lambda \in \mathrm E_{\varrho_0}$. Consequently,
\begin{multline*}
\langle \varrho_2, \mu \rangle = \langle \alpha^\vee - \varrho_0, \mu \rangle = \langle \varrho_0, \mu \rangle \cdot \langle \alpha^\vee - \varrho_0, \alpha - \tau \rangle = \\
\langle \varrho_0, \mu \rangle \cdot (1 + \langle \alpha^\vee, \sum \limits_{\lambda \in \mathrm E_{\varrho_0}} c_\lambda\lambda \rangle) \ge \langle \varrho_0, \mu \rangle > 0
\end{multline*}
and $\langle \varrho_2, \lambda \rangle = \langle \alpha^\vee - \varrho_0, \lambda \rangle = \langle \alpha^\vee, \lambda \rangle \ge 0$ for all $\lambda \in \mathrm E_{\varrho_0}$. Thus $\varrho_2 \in \mathcal K$.

\textit{Case}~2: $\mathcal K^1(\alpha)$ consists of two distinct elements $\varrho_1$ and~$\varrho_2$. We claim that $\varrho_1, \varrho_2$ satisfy conditions~(\ref{Phi8})(\ref{Phi8a}--\ref{Phi8c}). By Lemma~\ref{lemma_value_is_one_bis} one has $\langle \varrho_1, \alpha \rangle = \langle \varrho_2, \alpha \rangle = 1$, hence~(\ref{Phi8})(\ref{Phi8a}) holds. Condition~(\ref{Phi8})(\ref{Phi8c}) holds automatically. It remains to prove~(\ref{Phi8})(\ref{Phi8b}).

\begin{lemma} \label{lemma_face}
The cone $\QQ^+\varrho_1 + \QQ^+\varrho_2 \subset \mathcal Q$ is a \textup(two-dimensional\textup) face of the cone~$\mathcal K$.
\end{lemma}

\begin{proof}
Since $\QQ^+\varrho_1$ is an extremal ray of~$\mathcal K$, there exists an element $\nu_1 \in \QQ^+\Gamma$ such that $\langle \varrho_1, \nu_1 \rangle = 0$, $\langle \varrho_2, \nu_1 \rangle = 1$, and $\langle \varrho, \nu_1 \rangle > 0$ for all $\varrho \in \mathcal K^1 \setminus \lbrace \varrho_1, \varrho_2 \rbrace$. Similarly, there exists an element $\nu_2 \in \QQ^+\Gamma$ such that $\langle \varrho_2, \nu_2 \rangle = 0$, $\langle \varrho_1, \nu_2 \rangle = 1$, and $\langle \varrho, \nu_2 \rangle > 0$ for all $\varrho \in \mathcal K^1 \setminus \lbrace \varrho_1, \varrho_2 \rbrace$. Put $\nu = \nu_1 + \nu_2 - \alpha$. Then $\langle \varrho_1, \nu \rangle = \langle \varrho_2, \nu \rangle = 0$ and $\langle \varrho, \nu \rangle > 0$ for all $\varrho \in \mathcal K^1 \setminus \lbrace \varrho_1, \varrho_2 \rbrace$, hence $\QQ^+\varrho_1 + \QQ^+\varrho_2$ is a face of~$\mathcal K$.
\end{proof}

It follows from Lemma~\ref{lemma_face} that the space $\QQ(\mathrm E_{\varrho_1} \cap \mathrm E_{\varrho_2})$ has codimension~$2$ in $\QQ \Gamma$.

By Proposition~\ref{prop_canonical_form}(\ref{prop_canonical_form_a}), there exist $v_1, v_2 \in V^{(T_\ad)}_\alpha$ such that $[v_1] = [v_2] = [v]$ and
\begin{equation} \label{eqn_pr=0}
\pr_\lambda(v_i) = 0 \quad \text{for all} \quad \lambda \in
\mathrm E_{\varrho_i}, \ i = 1,2.
\end{equation}
Lemma~\ref{lemma_v_alpha} yields $v_1 \ne v_2$, and so $v_1 - v_2 \in \mathfrak g x_0 \setminus \lbrace 0 \rbrace$, which by Lemma~\ref{lemma_intersection} implies $v_1 - v_2 = ce_{-\alpha}x_0$ for some $c \in\nobreak \Bbbk^\times$. It then follows from~(\ref{eqn_pr=0}) that $\langle \alpha^\vee, \lambda \rangle = 0$ for all $\lambda \in \mathrm E_{\varrho_1} \cap \mathrm E_{\varrho_2}$, therefore $\iota(\alpha^\vee) = a_1 \varrho_1 + a_2 \varrho_2$ for some $a_1, a_2 \in \QQ^+$. In view of Lemma~\ref{lemma_v_alpha}(\ref{lemma_v_alpha_a}) one has $\langle \alpha^\vee, \mu \rangle > 0$ for all $\mu \in \mathrm E \setminus (\mathrm E_{\varrho_1} \cap \mathrm E_{\varrho_2})$, whence $a_i \ne 0$ for $i = 1,2$, which proves~(\ref{Phi8})(\ref{Phi8b}).
\end{proof}

\subsubsection{Case $\sigma \in \Delta^+ \setminus \Pi$}

We need to prove properties (\ref{Phi2})--(\ref{Phi4}) and~(\ref{MF}). In what follows, we fix an arbitrary element $\varrho \in \mathcal K^1(\sigma)$.

\begin{lemma} \label{lemma_edeltav=0}
For every $\delta \in \Pi$ there exist $v' \in V^{(T_\ad)}_\sigma$ and $c \in \Bbbk$ such that $v' = v - c e_{-\sigma}x_0$ and $e_\delta v' = 0$.
\end{lemma}

\begin{proof}
Take any $\delta \in \Pi$ and assume that $e_\delta v \ne 0$. Then Corollary~\ref{crl_sigma-root} yields $\sigma - \delta \in \Delta^+$ and $e_\delta v = c e_{-(\sigma - \delta)} x_0$ for some $c \in \Bbbk^\times$. Then the vector $v' = v - cN_{\delta, -\sigma}^{-1} e_{-\sigma}x_0$ satisfies $e_\delta v' = 0$.
\end{proof}

\begin{lemma} \label{lemma_sigma_is_a_root}
The set $\lbrace \delta \in \Supp \sigma \mid \sigma - \delta \in \Delta^+ \rbrace$ contains at least two elements.
\end{lemma}

\begin{proof}
Thanks to Lemma~\ref{lemma_edeltav1}, there exists $\beta \in \Pi$ such that $e_\beta v \ne 0$. Then $\sigma - \beta \in \Delta^+$ by Corollary~\ref{crl_sigma-root}. Next, by Lemma~\ref{lemma_edeltav=0}
there exists $v' \in V^{(T_\ad)}_\sigma$ such that $[v'] = [v]$ and $e_\beta v' = 0$. Again, there exists $\gamma \in \Pi$ such that $e_\gamma v' \ne 0$, which implies $\sigma - \gamma \in \Delta^+$. Clearly, $\beta \ne \gamma$ and $\beta, \gamma \in \Supp \sigma$.
\end{proof}

\begin{lemma} \label{lemma_locally_dominant}
One of the following two alternatives holds.
\begin{enumerate}[label=\textup{(\arabic*)},ref=\textup{\arabic*}]
\item
$\langle \delta^\vee, \sigma \rangle \ge 0$ for all $\delta \in \Supp \sigma$ \textup(that is, $\sigma$ is a dominant root of~$\Delta_\sigma$\textup).

\item
$\Supp \sigma$ is of type~$\mathsf G_2$ and $\sigma = \alpha_1 + \alpha_2$.
\end{enumerate}
\end{lemma}

\begin{proof}
In view of Proposition~\ref{prop_canonical_form}(\ref{prop_canonical_form_a}), we may assume that
\begin{equation} \label{eqn_canonical_form}
\pr_\lambda(v) = 0 \; \text{ for all } \; \lambda \in \mathrm E_\varrho.
\end{equation}
Thanks to Lemma~\ref{lemma_edeltav1}, there exists $\alpha \in \Pi$ such that $e_\alpha v \ne 0$. Then Corollary~\ref{crl_sigma-root} yields $\sigma - \alpha \in \Delta^+$ and $e_\alpha v = ce_{-(\sigma - \alpha)}x_0$ for some $c \in \Bbbk^\times$. By~(\ref{eqn_canonical_form}) the latter implies $\langle (\sigma - \alpha)^\vee, \lambda \rangle = 0$ for all $\lambda \in \mathrm E_\varrho$. Thus for every $\delta \in \Supp (\sigma - \alpha)$ one has $\delta \in \mathrm E_\varrho^\perp$, whence $\iota(\delta^\vee)$ is a non-negative multiple of~$\varrho$. As $\langle \varrho, \sigma \rangle > 0$, it follows that $\langle \delta^\vee, \sigma \rangle \ge 0$ for all $\delta \in \Supp (\sigma - \alpha)$. Now assume that $\langle \alpha^\vee, \sigma \rangle < 0$. Then $\langle \alpha^\vee, \sigma \rangle \le -1$ and $\langle \alpha^\vee, \sigma - \alpha \rangle \le -3$. The latter implies that $\Supp \sigma$ has type~$\mathsf G_2$ with $\alpha = \alpha_1$ and $\sigma = \alpha_1 + \alpha_2$.
\end{proof}

\begin{proof}[Proof of~\textup(\ref{Phi2}\textup)]
Applying Lemma~\ref{lemma_sigma_is_a_root} we find that $\sigma$ cannot be the highest root of $\Delta_\sigma$ unless the support of~$\sigma$ has type $\mathsf A_r$. (The latter can be seen, for instance, by inspecting the extended Dynkin diagrams.) By the same reason $\sigma$ cannot be the short dominant root in type~$\mathsf G_2$. All the other possibilities given by Lemma~\ref{lemma_locally_dominant} are already contained in~$\overline \Sigma(G)$.
\end{proof}

\begin{proof}[Proof of~\textup(\ref{Phi3}\textup)]
Reasoning as in the proof of Lemma~\ref{lemma_locally_dominant}, we find a simple root $\alpha \in \Pi$ such that $\sigma - \alpha \in \Delta^+$ and $\iota(\delta^\vee)$ is a non-negative multiple of $\varrho$ for all $\delta \in \Supp (\sigma - \alpha)$.
We claim that $\Pi_\sigma \subset \Supp(\sigma - \alpha)$. Indeed, the inclusion $(\Pi_\sigma \setminus \lbrace \alpha \rbrace) \subset \Supp(\sigma - \alpha)$ is obvious and the condition $\alpha \in \Pi_\sigma$ implies $\alpha \in \Supp(\sigma - \alpha)$ in view of~(\ref{eqn_Pi_sigma}). Thus, given any $\beta \in \Pi_\sigma$, the element $\iota(\beta^\vee)$ is a non-negative multiple of~$\varrho$. As $\beta \in \sigma^\perp$ and $\langle \varrho, \sigma \rangle > 0$, it follows that $\iota(\beta^\vee) = 0$, whence $\beta \in \Gamma^\perp$.
\end{proof}

\begin{proof}[Proof of~\textup(\ref{Phi4}\textup)] Suppose that $\sigma = \alpha_1 + \ldots + \alpha_r$ with $\Supp \sigma$ of type $\mathsf B_r$ ($r \ge 2$). Taking into account Lemma~\ref{lemma_edeltav=0} we may assume that $e_{\alpha_r} v = 0$. For $2 \le i \le r - 1$, we have $\sigma - \alpha_i \notin \Delta^+$, which implies $e_{\alpha_i} v = 0$ by Corollary~\ref{crl_sigma-root}(\ref{crl_sigma-root_a}). Now assume that $\alpha_r \in \Gamma^\perp$. Since $\alpha_i \in \Gamma^\perp$ for $2 \le i \le r - 1$ by~(\ref{Phi3}), it follows that $\langle (\sigma - \alpha_1)^\vee, \lambda \rangle = 0$ for all $\lambda \in \mathrm E$, whence $e_{-(\sigma - \alpha_1)} x_0 = 0$. In view of Lemmas~\ref{lemma_edeltav2} and~\ref{lemma_intersection} the latter implies $e_{\alpha_1}v = 0$. We have obtained that $e_{\delta} v = 0$ for all $\delta \in \Supp \sigma$ and hence for all $\delta \in \Pi$, which contradicts Lemma~\ref{lemma_edeltav1} as $\sigma \ne 0$.
\end{proof}

\begin{proof}[Proof of~\textup{(\ref{MF})}]
Here we use a short argument from the proof of~\cite[Proposition~3.16]{BvS}. A case-by-case check of all relevant entries in Table~\ref{table_spherical_roots} shows that the set
\[
\lbrace \delta \in \Pi \mid \sigma - \delta \in \Delta^+ \rbrace
\]
contains exactly two elements, which will be denoted by $\beta$ and~$\gamma$. Let $v' \in V^{(T_\ad)}_\sigma$ be another vector such that $[v'] \in \EuScript{TS} \setminus \lbrace 0 \rbrace$. By Lemma~\ref{lemma_edeltav=0} we may assume that $e_\beta v = e_\beta v' = 0$. It follows from Corollary~\ref{crl_sigma-root} that $e_\delta v = e_\delta v' = 0$ for all $\delta \in \Pi \setminus \lbrace \gamma \rbrace$. Consequently, $e_\gamma v \ne 0$ and $e_\gamma v' \ne 0$ in view of Lemma~\ref{lemma_edeltav1}. Then Corollary~\ref{crl_sigma-root}(\ref{crl_sigma-root_b}) yields $e_\gamma v = c e_{-(\sigma - \gamma)} x_0$ and $e_\gamma v' = c' e_{-(\sigma - \gamma)} x_0$ for some $c, c' \in \Bbbk^\times$. It follows that the vector $c'v - c v'$ is annihilated by $e_\gamma$ and hence by all $e_\delta$ with $\delta \in \Pi$. As $\sigma \ne 0$, Lemma~\ref{lemma_edeltav1} yields $c'v - cv' = 0$.
\end{proof}

\subsubsection{Case $\sigma = \alpha + \beta$ with $\alpha, \beta \in \Pi$ and $\alpha \perp \beta$}

Properties (\ref{Phi2}) and (\ref{Phi3}) hold automatically. Properties (\ref{Phi5}) and~(\ref{MF}) follow from the lemma below.

\begin{lemma} \label{lemma_alpha+beta}
The following assertions hold:
\begin{enumerate}[label=\textup{(\alph*)},ref=\textup{\alph*}]
\item \label{lemma_alpha+beta_a}
$\iota(\alpha^\vee) = \iota(\beta^\vee)$;

\item \label{lemma_alpha+beta_b}
the vector $v$ is given by
\[
v = c \sum \limits_{\mu \in \mathrm E : \langle \alpha^\vee, \mu \rangle > 0} \frac{1}{\langle \alpha^\vee, \mu \rangle} e_{-\alpha} e_{-\beta} v_\mu
\]
for some $c \in \Bbbk^\times$.
\end{enumerate}
\end{lemma}

\begin{proof}
Take any $\varrho \in \mathcal K^1(\sigma)$. Proposition~\ref{prop_role_of_NE2} yields $\varrho \notin \mathcal P$. Then by Lemma~\ref{lemma_positive_multiple} there exists $\delta \in \Pi$ such that $\iota(\delta^\vee)$ is a positive multiple of~$\varrho$. As $\langle \delta^\vee, \sigma \rangle > 0$, it follows that $\delta \in \lbrace \alpha, \beta \rbrace$. Assume without loss of generality that $\delta = \alpha$. Proposition~\ref{prop_canonical_form}(\ref{prop_canonical_form_b}) yields $\pr_\lambda(v) = 0$ for all $\lambda \in \mathrm E_\varrho$. Next, for every $\mu \in \mathrm E \setminus \mathrm E_\varrho$ one has
\begin{equation} \label{eqn_p(v)}
\pr_\mu(v) = c_\mu e_{-\alpha}e_{-\beta}v_\mu = c_\mu e_{-\beta}e_{-\alpha}v_\mu
\end{equation}
with $c_\mu \in \Bbbk$. It follows that
\begin{equation} \label{eqn_three}
e_\alpha v = \sum \limits_{\mu \in \mathrm E \setminus \mathrm E_\varrho} c_\mu e_\alpha e_{-\alpha} e_{-\beta} v_\mu = \sum \limits_{\mu \in \mathrm E \setminus \mathrm E_\varrho} c_\mu \langle \alpha^\vee, \mu \rangle e_{-\beta} v_\mu.
\end{equation}
On the other hand, since $v \ne 0$, there exists $\nu \in \mathrm E \setminus \mathrm E_\varrho$ such that $\pr_\nu(v) \ne 0$, which implies $c_\nu \ne 0$ and $e_{-\beta} v_\nu \ne 0$ in view of~(\ref{eqn_p(v)}). It follows that $e_\alpha v \ne 0$, and Corollary~\ref{crl_sigma-root}(\ref{crl_sigma-root_b}) yields
\begin{equation} \label{eqn_four}
e_\alpha v = ce_{-\beta}x_0 = c\sum \limits_{\lambda \in \mathrm E} e_{-\beta} v_\lambda
\end{equation}
for some $c \in \Bbbk^\times$. Then $\langle \beta^\vee, \lambda \rangle = 0$ for all $\lambda \in \mathrm E_\varrho$, whence $\iota(\beta^\vee)$ and $\iota(\alpha^\vee)$ are proportional. The equalities $\langle \alpha^\vee, \sigma \rangle = 2 = \langle \beta^\vee, \sigma \rangle$ imply $\iota(\alpha^\vee) = \iota(\beta^\vee)$, which proves~(\ref{lemma_alpha+beta_a}). Now comparing~(\ref{eqn_three}) with~(\ref{eqn_four}) yields~(\ref{lemma_alpha+beta_b}).
\end{proof}

\subsubsection{Case $\sigma \notin \Delta^+$ and $\sigma$ is not the sum of two orthogonal simple roots}

We need to prove properties (\ref{Phi2}), (\ref{Phi3}), (\ref{Phi6}), and~(\ref{MF}). In what follows, we fix an arbitrary element $\varrho \in \mathcal K^1(\sigma)$.

The following lemma is similar to~\cite[Proposition~3.9]{BvS}.

\begin{lemma} \label{lemma_nonroot-alpha}
There exists a unique $\beta \in \Pi$ such that $\sigma - \beta \in \Delta^+$.
\end{lemma}

\begin{proof}
By Lemma~\ref{lemma_edeltav1} there exists $\beta \in \Pi$ such that $e_\beta v \ne 0$. Then Corollary~\ref{crl_sigma-root}(\ref{crl_sigma-root_a}) yields $\sigma - \beta \in \Delta^+$. The condition $\sigma \notin \Delta^+$ implies $\langle \beta^\vee, \sigma - \beta \rangle \ge 0$, and so $\langle \beta^\vee, \sigma \rangle \ge 2$. Now take any $\gamma \in \Pi \setminus \lbrace \beta \rbrace$ and assume that $\sigma - \gamma \in \Delta^+$. Clearly,
\[
\langle \beta^\vee, \sigma - \gamma \rangle = \langle \beta^\vee, \sigma \rangle - \langle \beta^\vee, \gamma \rangle \ge 2 > 0,
\]
whence $\sigma - \beta - \gamma \in \Delta^+$ (note that $\sigma \ne \beta + \gamma$ by our assumptions). Let $\mathfrak h \subset \mathfrak g$ be the standard Levi subalgebra with set of simple roots $\lbrace \beta, \gamma \rbrace$ and regard $\mathfrak g$ as an $\mathfrak h$-module. As $\langle \beta^\vee, \sigma - \gamma + \beta \rangle \ge 4$, one has $\sigma - \gamma + \beta \notin \Delta^+$. Consequently, $e_{\sigma - \gamma}$ is a highest weight vector for~$\mathfrak h$. Similarly, $e_{\sigma - \beta}$ is another highest weight vector for~$\mathfrak h$. Since $e_{\sigma - \beta - \gamma}$ is proportional to both $[e_{-\beta}, e_{\sigma - \gamma}]$ and $[e_{-\gamma}, e_{\sigma - \beta}]$, it follows that $e_{\sigma - \beta - \gamma}$ is contained in two different simple $\mathfrak h$-submodules of~$\mathfrak g$, a contradiction.
\end{proof}

Until the end of this case, $\beta$ stands for the unique simple root such that $\sigma - \beta \in \Delta^+$.

\begin{lemma} \label{lemma_nonroot_properties}
The following assertions hold:
\begin{enumerate}[label=\textup{(\alph*)},ref=\textup{\alph*}]
\item \label{lemma_nonroot_properties_a}
$e_\beta v \ne 0$;

\item \label{lemma_nonroot_properties_b}
$e_\beta v = c e_{-(\sigma - \beta)} x_0$ for some $c \in \Bbbk^\times$;

\item \label{lemma_nonroot_properties_c}
$\iota((\sigma - \beta)^\vee)$ is a positive multiple of $\varrho$.
\end{enumerate}
\end{lemma}

\begin{proof}
By Lemma~\ref{lemma_edeltav1} there exists $\delta \in \Pi$ such that $e_\delta v \ne 0$. Corollary~\ref{crl_sigma-root}(\ref{crl_sigma-root_a}) and Lemma~\ref{lemma_nonroot-alpha} then yield $\delta = \beta$, whence part~(\ref{lemma_nonroot_properties_a}). Part~(\ref{lemma_nonroot_properties_b}) is implied by Corollary~\ref{crl_sigma-root}(\ref{crl_sigma-root_b}). For every $\lambda \in \mathrm E_\varrho$, one has $\pr_\lambda(v) = 0$ by Proposition~\ref{prop_canonical_form}(\ref{prop_canonical_form_b}), hence $\langle (\sigma - \beta)^\vee, \lambda \rangle = 0$. Thus $\iota((\sigma - \beta)^\vee)$ is a non-negative multiple of~$\varrho$. As $e_\beta v \ne 0$, it follows that $\iota((\sigma - \beta)^\vee) \ne 0$, whence part~(\ref{lemma_nonroot_properties_c}).
\end{proof}

\begin{lemma} \label{lemma_non-orthogonal}
Suppose that $\gamma \in \Pi \setminus \lbrace \beta \rbrace$ is such that $\sigma - \beta - \gamma \in \Delta^+$. Then $\langle \beta^\vee, \gamma \rangle < 0$ \textup(that is, $\beta + \gamma \in \Delta^+$\textup).
\end{lemma}

\begin{proof}
By Lemma~\ref{lemma_nonroot_properties}(\ref{lemma_nonroot_properties_b}) one has $e_\beta v = c e_{-(\sigma - \beta)} x_0$ for some $c \in \Bbbk^\times$. Assume that $\beta$ and $\gamma$ are orthogonal. Then $\beta + \gamma \notin \Delta^+$, and so $e_{\sigma - \beta - \gamma}v = 0$ by Corollary~\ref{crl_sigma-root}(\ref{crl_sigma-root_a}). As $\beta \ne \gamma$, it follows from Lemma~\ref{lemma_nonroot-alpha} and Corollary~\ref{crl_sigma-root}(\ref{crl_sigma-root_a}) that $e_\gamma v = 0$. Then $e_{\sigma - \beta} v = N_{\gamma, \sigma - \beta - \gamma}^{-1} e_\gamma e_{\sigma - \beta - \gamma}v = 0$. Hence in view of Lemma~\ref{lemma_nonroot_properties}(\ref{lemma_nonroot_properties_c}) one has
\[
0 = e_\beta e_{\sigma - \beta} v = e_{\sigma - \beta} e_\beta v = e_{\sigma - \beta} (c e_{-(\sigma - \beta)}x_0) = c h_{\sigma - \beta} x_0 \ne 0,
\]
a contradiction.
\end{proof}

\begin{lemma} \label{lemma_nonroot_multiple}
For every $\delta \in \Supp \sigma$, the element $\iota(\delta^\vee)$ is a non-negative multiple of~$\varrho$.
\end{lemma}

\begin{proof}
By Lemma~\ref{lemma_nonroot_properties}(\ref{lemma_nonroot_properties_c}), the element $\iota((\sigma - \beta)^\vee)$ is a positive multiple of~$\varrho$. Since $(\sigma - \beta)^\vee = \sum \limits_{\delta \in \Supp(\sigma - \beta)} c_\delta \delta^\vee$ with all the coefficients $c_\delta$ being positive, it follows that $\langle \delta^\vee, \lambda \rangle = 0$ for all $\delta \in \Supp (\sigma - \beta)$ and $\lambda \in \mathrm E_\varrho$. Consequently, $\iota(\delta^\vee)$ is a non-negative multiple of~$\varrho$ for all $\delta \in \Supp (\sigma -\nobreak \beta)$. It remains to show that $\Supp (\sigma - \beta) = \Supp \sigma$ or, equivalently, $\beta \in \Supp (\sigma - \beta)$. Assume the converse and choose $\gamma \in \Supp (\sigma - \beta)$ such that $\sigma - \beta - \gamma \in \Delta^+$. (The latter is possible because $\sigma - \beta \notin \Pi$.) Then $\langle \beta^\vee, \gamma \rangle < 0$ by Lemma~\ref{lemma_non-orthogonal}, hence $\langle \beta^\vee, \sigma - \beta \rangle < 0$. The latter yields $\sigma = \beta + (\sigma - \beta) \in \Delta^+$, a~contradiction.
\end{proof}

\begin{proof}[Proof of~\textup(\ref{Phi2}\textup)] The key idea of our proof is to reduce the consideration to the case where $V$ is a simple $G$-module, which has already been investigated in~\cite{Jan}.

Replacing $G$ with a suitable finite cover, we may assume that $G = G_0 \times C$ where $G_0$ is a simply connected semisimple group and $C$ is a torus. Let $L$ denote the standard Levi subgroup of~$G$ with set of simple roots $\Supp \sigma$ and let $L'$ be the derived subgroup of~$L$. Put $T' = L' \cap T$, so that $T'$ is a maximal torus of~$L'$, and consider the natural restriction map $\pi \colon \mathfrak X(T) \to \mathfrak X(T')$.

Lemma~\ref{lemma_nonroot_properties}(\ref{lemma_nonroot_properties_b}) says that $e_\beta v = c e_{-(\sigma - \beta)} x_0$ for some $c \in \Bbbk^\times$. By Lemma~\ref{lemma_value_is_one} there exists $\nu \in \mathrm E \setminus \mathrm E_\varrho$ such that $\langle \varrho, \nu \rangle = 1$. It then follows from Lemma~\ref{lemma_nonroot_properties}(\ref{lemma_nonroot_properties_c}) that $\langle (\sigma - \beta)^\vee, \nu \rangle > 0$, whence $e_{-(\sigma - \beta)} v_\nu \ne 0$. Consequently, $e_\beta \pr_\nu(v) \ne 0$ and $\pr_\nu(v) \ne 0$.

Let $W \subset V(\nu)$ be the $L$-submodule generated by~$v_{\nu}$. Note that $W \cap \mathfrak gv_\nu = \mathfrak l'v_\nu$. Since $\sigma \notin \Delta^+ \cup \lbrace 0 \rbrace$, one has $\pr_\nu(v) \notin \mathfrak l'v_\nu$, therefore the image of $\pr_\nu(v)$ in $W/ \mathfrak l'v_\nu$ is nonzero. We now show that this image is $L'_{v_\nu}$-invariant.

First of all, we prove that
\begin{equation} \label{eqn_L'_v_nu1}
e_\delta \pr_\nu (v) \in \mathfrak l' v_\nu \; \text{ for all } \delta \in \Delta_\sigma \text{ with } e_\delta \in \mathfrak l'_{v_\nu}.
\end{equation}
Take any such $\delta$. It suffices to show that
\begin{equation} \label{eqn_g_v0}
e_\delta \in \mathfrak g_{x_0},
\end{equation}
because the latter implies $e_\delta v \in \mathfrak gx_0$ and $e_\delta \pr_\nu(v) \in \mathfrak g v_\nu \cap W = \mathfrak l' v_\nu$. If $\delta \in \Delta^+$ then (\ref{eqn_g_v0}) holds automatically. Now assume that $\delta \in \Delta^-$. Then $\Supp \delta \subset \Supp \sigma$, hence $\iota(\delta^\vee)$ is a multiple of~$\varrho$ by Lemma~\ref{lemma_nonroot_multiple}. Since $e_\delta v_\nu = 0$, it follows that $\langle \delta^\vee, \nu \rangle = 0$. The latter implies $\iota(\delta^\vee) = 0$, whence~(\ref{eqn_g_v0}).

Next, we prove that
\begin{equation} \label{eqn_L'_v_nu2}
t \cdot \pr_\nu (v) = \pr_\nu (v) \; \text{ for all } t \in T'_{v_\nu}.
\end{equation}
The latter claim will follow as soon as we prove that $T'_{v_\nu} \subset T_{x_0}$. Since
\[
T'_{v_\nu} = \lbrace t \in T \mid t^\lambda = 1 \text{ for all } \lambda \in \Ker \pi + \ZZ \nu \rbrace,
\]
it suffices to show that $\mathrm E \subset \Ker \pi + \ZZ \nu$. Observe that the lattice $\Ker \pi$ is generated by all elements of $\mathfrak X(C)$ and all fundamental weights of $G_0$ corresponding to simple roots in the set $\Pi \setminus \Supp \sigma$. Since $\delta \in \mathrm E_\varrho^\perp$ for all $\delta \in \Supp \sigma$ (see Lemma~\ref{lemma_nonroot_multiple}), we have $\mathrm E_\varrho \subset \Ker \pi \subset \Ker \pi + \ZZ \nu$. Next, for every $\mu \in \mathrm E \setminus \mathrm E_\varrho$, there is a unique expression $\mu = \tau + b\nu$ with $\tau \in \QQ \mathrm E_\varrho$ and $b \in \QQ$. Since $\langle \varrho, \nu \rangle = 1$, we have $b = \langle \varrho, \mu \rangle \in \ZZ$ and $\tau = \mu - \langle \varrho, \mu \rangle \nu \in \ZZ \Gamma \cap \QQ \mathrm E_\varrho = \ZZ \mathrm E_\varrho$, which implies $\mu \in \Ker \pi + \ZZ \nu$.

It follows from~(\ref{eqn_L'_v_nu1}) and~(\ref{eqn_L'_v_nu2}) that the image of $\pr_\nu(v)$ in $W / \mathfrak l' v_\nu$ is a nonzero element of~$(W / \mathfrak l' v_\nu)^{L'_{v_\nu}}$ (compare with Lemma~\ref{lemma_G_x-fixed}). Now results of~\cite[\S\,1.3]{Jan} (see Proposition~1.6 in loc. cit. and its proof) imply that $\sigma \in\nobreak \overline \Sigma(G)$.
\end{proof}

\begin{proof}[Proof of~\textup(\ref{Phi3}\textup)]
Let $\delta \in \Pi_\sigma$. Then $\langle \delta^\vee, \sigma \rangle = 0$ by~(\ref{eqn_Pi_sigma}). Lemma~\ref{lemma_nonroot_multiple} implies that $\iota(\delta^\vee)$ is a non-negative multiple of~$\varrho$. As $\langle \varrho, \sigma \rangle > 0$, one has $\iota(\delta^\vee) = 0$ and so $\delta \in \Gamma^\perp$.
\end{proof}

\begin{proof}[Proof of~\textup(\ref{Phi6}\textup)]
Suppose that $\sigma = 2\alpha$ with $\alpha \in \Pi$. Proposition~\ref{prop_role_of_NE2} yields $\varrho \notin \mathcal P$. Hence by Lemma~\ref{lemma_positive_multiple} there exist $\delta \in \Pi$ and a positive integer $n$ such that $\iota(\delta^\vee) = n \varrho$. In particular, we obtain $\langle \delta^\vee, \sigma \rangle > 0$, whence $\delta = \alpha$. Note that for every $\mu \in \mathrm E \setminus \mathrm E_\varrho$ one has $\pr_\mu(v) = c_\mu e_{-\alpha}e_{-\alpha}v_\mu$ with $c_\mu \in \Bbbk$. Applying Lemma~\ref{lemma_edeltav1} we obtain $e_\alpha v \ne 0$. Then Corollary~\ref{crl_sigma-root}(\ref{crl_sigma-root_b}) yields $e_\alpha v = ce_{-\alpha}x_0$ for some $c \in \Bbbk^\times$. Consequently, for every $\mu \in \mathrm E \setminus \mathrm E_\varrho$ one has $e_\alpha \pr_\mu(v) \ne 0$ and hence $\pr_\mu(v) \ne 0$, which implies $e_{-\alpha} e_{-\alpha} v_\mu \ne 0$ and therefore $\langle \alpha^\vee, \mu \rangle \ge 2$. In view of Lemma~\ref{lemma_value_is_one}, the latter yields $n \ge 2$. Since $4/n = \langle \alpha^\vee / n, 2 \alpha \rangle = \langle \varrho, \sigma \rangle \in \ZZ$, it follows that $n \in \lbrace 2, 4 \rbrace$. In both cases we obtain $\langle \alpha^\vee, \ZZ \Gamma \rangle \subset 2 \ZZ$ as required.
\end{proof}

\begin{proof}[Proof of~\textup{(\ref{MF})}]
Here we again use a short argument from the proof of~\cite[Proposition~3.16]{BvS}. Let $v' \in V^{(T_\ad)}_\sigma$ be another vector such that $[v'] \in \EuScript{TS} \setminus \lbrace 0 \rbrace$. By Corollary~\ref{crl_sigma-root}(\ref{crl_sigma-root_a}) and Lemma~\ref{lemma_nonroot-alpha} one has $e_\delta v = e_\delta v' = 0$ for all $\delta \in \Pi \setminus \lbrace \beta \rbrace$. Next, by Lemma~\ref{lemma_nonroot_properties}(\ref{lemma_nonroot_properties_b}) we have $e_\beta v = c e_{-(\sigma - \beta)} x_0$ and $e_\beta v' = c' e_{-(\sigma - \beta)} x_0$ for some $c,c' \in \Bbbk^\times$. Then the vector $c'v - cv'$ is annihilated by $e_\beta$ and hence by all $e_\delta$ with $\delta \in \Pi$. As $\sigma \ne 0$, Lemma~\ref{lemma_edeltav1} yields $c'v - cv' = 0$.
\end{proof}

\subsection{Proof of Theorem~\ref{thm_tangent_space}: Step~2}
\label{subsec_step_2}

Our goal in this subsection is to prove the following

\begin{proposition} \label{prop_last_step}
Suppose that a weight $\sigma \in \mathfrak X(T_\ad)$ satisfies
conditions~\textup(\ref{Phi1}\textup)--\textup(\ref{Phi8}\textup). Then $\sigma \in \Phi(\Gamma)$.
\end{proposition}

In the proof of this proposition we shall need the following lemma.

\begin{lemma} \label{lemma_v_is_there}
Suppose that $\sigma \in \ZZ \Gamma$, $v \in V^{(T_\ad)}_\sigma$, $e_{\delta} v = 0$ for all $\delta \in \Gamma^\perp$, and $e_{\delta} v \in \mathfrak g x_0$ for all $\delta \in \Pi \setminus \Gamma^\perp$. Then $[v] \in (V / \mathfrak g x_0)^{G_{x_0}}$.
\end{lemma}

\begin{proof}
The claim will follow as soon as we check conditions (\ref{lemma_G_x-fixed_1}) and~(\ref{lemma_G_x-fixed_3}) of Lemma~\ref{lemma_G_x-fixed}. As $\sigma \in \ZZ \Gamma$, the vector $v$ is $T_{x_0}$-invariant, hence so is~$[v]$. In view of the hypothesis it now suffices to prove that $e_{-\delta} v = 0$ for all $\delta \in \Gamma^\perp = \mathrm E^\perp$. But the latter holds because $e_\delta v = 0$ and $\langle \delta^\vee, \lambda - \sigma \rangle = 0$ for all $\delta \in \mathrm E^\perp$ and $\lambda \in \mathrm E$.
\end{proof}

\begin{proof}[Proof of Proposition~\textup{\ref{prop_last_step}} for $\sigma \in \Pi$]

Suppose that $\sigma = \alpha \in \Pi$. Then by~(\ref{Phi8}) there exist two distinct elements $\varrho_1, \varrho_2 \in \mathcal L$ satisfying conditions (\ref{Phi8})(\ref{Phi8a}--\ref{Phi8c}). It follows from (\ref{Phi8})(\ref{Phi8a},\,\ref{Phi8c}) that $\varnothing \ne \mathcal K^1(\alpha) \subset \lbrace \varrho_1, \varrho_2 \rbrace$. Further we consider two cases.

\textit{Case}~1: $|\mathcal K^1(\alpha)| = 1$. Assume without loss of generality that $\mathcal K^1(\alpha) = \lbrace \varrho_1 \rbrace$. Then condition (\ref{Phi8})(\ref{Phi8b}) implies $\langle \alpha^\vee, \mu \rangle > 0$ for all $\mu \in \mathrm E \setminus \mathrm E_{\varrho_1}$, and we put
\[
v = \sum \limits_{\mu \in \mathrm E \setminus \mathrm E_{\varrho_1}} \frac{\langle \varrho_1, \mu \rangle}{\langle \alpha^\vee, \mu \rangle} e_{-\alpha}v_\mu.
\]
Clearly, $v \in V^{(T_\ad)}_\alpha$ and $e_\beta v = 0$ for all $\beta \in \Pi \setminus \lbrace \alpha \rbrace$.  As $e_\alpha v = \sum \limits_{\mu \in \mathrm E \setminus \mathrm E_{\varrho_1}} \langle \varrho_1, \mu \rangle v_\mu \in \mathfrak t x_0 \subset \mathfrak g x_0$, we obtain $[v] \in (V / \mathfrak g x_0)^{G_{x_0}}$ by Lemma~\ref{lemma_v_is_there}. Since $\iota(\alpha^\vee)$ is not proportional to~$\varrho_1$, one has $[v] \ne 0$ (see Lemma~\ref{lemma_intersection}). Applying Proposition~\ref{prop_section_extensions}(\ref{prop_section_extensions_b},\,\ref{prop_section_extensions_d}) and Corollary~\ref{crl_[v]_in_TS} we find that $[v] \in \EuScript{TS}$, whence $\alpha \in \Phi(\Gamma)$.

\textit{Case}~2: $|\mathcal K^1(\alpha)| = 2$, so that $\mathcal K^1(\alpha) = \lbrace \varrho_1, \varrho_2 \rbrace$.  According to~(\ref{Phi8})(\ref{Phi8b}), let $b_1, b_2 \in \QQ^+ \setminus \lbrace 0 \rbrace$ be such that $\iota(\alpha^\vee) = b_1\varrho_1 + b_2\varrho_2$. Put
\[
v_1 = b_1 \sum \limits_{\mu \in \mathrm E \setminus \mathrm E_{\varrho_1}} \frac{\langle \varrho_1, \mu \rangle}{\langle \alpha^\vee, \mu \rangle} e_{-\alpha}v_\mu \; \text{ and } \; v_2 = -b_2 \sum \limits_{\mu \in \mathrm E \setminus \mathrm E_{\varrho_2}} \frac{\langle \varrho_2, \mu \rangle}{\langle \alpha^\vee, \mu \rangle} e_{-\alpha}v_\mu.
\]
Clearly, $v_1, v_2 \in V^{(T_\ad)}_\alpha$. As in Case~1, we see that $[v_1], [v_2] \in (V / \mathfrak gx_0)^{G_{x_0}} \setminus \lbrace 0 \rbrace$. An easy check shows that $[v_1] = [v_2]$. Applying Proposition~\ref{prop_section_extensions}(\ref{prop_section_extensions_b},\,\ref{prop_section_extensions_d}) and Corollary~\ref{crl_[v]_in_TS}
we find that $[v_1] = [v_2] \in \EuScript{TS}$ and so $\alpha \in \Phi(\Gamma)$.
\end{proof}

\begin{proof}[Proof of Proposition~\textup{\ref{prop_last_step}} for
$\sigma \notin \Pi$]
Let $\sigma \in \mathfrak X(T_\ad) \setminus \Pi$ and suppose that $\sigma$ satisfies conditions (\ref{Phi1})--(\ref{Phi8}).

\begin{lemma} \label{lemma_not_in_Pi_extends}
Suppose that $v \in V^{(T_\ad)}_\sigma$ and $[v] \in (V / \mathfrak g x_0)^{G_{x_0}}$. Then $[v] \in \EuScript{TS}$.
\end{lemma}

\begin{proof}
Let $s \in H^0(O, \mathcal N_{O})^G$ be the section defined by $s(x_0) = [v]$. Recall from Corollary~\ref{crl_[v]_in_TS} that $[v] \in \EuScript{TS}$ if and only if $s$ extends to $O \cup O_\varrho$ for each $\varrho \in \mathcal P$. We now fix any $\varrho \in \mathcal P$ and show that $s$ extends to $O \cup O_\varrho$. Assume that $\langle \varrho, \sigma \rangle > 0$. Then by~(\ref{Phi7}) there exists $\delta \in \Pi$ such that $\iota(\delta^\vee)$ is a positive multiple of~$\varrho$. It follows that $\delta \in \mathrm E_\varrho^\perp \backslash \mathrm E^\perp$, which contradicts the condition $\mathrm E^\perp = \mathrm E_\varrho^\perp$. Consequently, $\langle \varrho, \sigma \rangle \le 0$, which implies $[v] \in \EuScript{TS}$ by Proposition~\ref{prop_section_extensions}(\ref{prop_section_extensions_b}).
\end{proof}

To complete the proof, by Lemmas~\ref{lemma_not_in_Pi_extends} and~\ref{lemma_v_is_there} it suffices to find a vector $v \in V^{(T_\ad)}_\sigma$ with the following properties:
\begin{enumerate}[label=\textup{(V\arabic*)},ref=\textup{V\arabic*}]
\item \label{V1}
$v \notin \mathfrak gx_0$;

\item \label{V2}
$e_\delta v = 0$ for all $\delta \in \Gamma^\perp$;

\item \label{V3}
$e_\delta v \in \mathfrak g x_0$ for all $\delta \in \Pi \setminus \Gamma^\perp$.
\end{enumerate}
In view of condition~(\ref{Phi2}), it is enough to present such a vector $v$ for each of the cases in Table~\ref{table_spherical_roots}. This is done in the remaining part of the proof. In each case, the explicit formula for $v$ depends on the signs of the structure constants $N_{\alpha, \beta}$ of the Lie algebra $[\mathfrak l, \mathfrak l]$, where $\mathfrak l \subset \mathfrak g$ is the standard Levi subalgebra with set of simple roots $\Supp \sigma$; we use the choice of these signs specified in Appendix~\ref{app_structure_constants}.

\textit{Case}~1: $\sigma \in \Delta^+$. By Lemma~\ref{lemma_intersection} one has $V^{(T_\ad)}_\sigma \cap \mathfrak gx_0 = \Bbbk e_{-\sigma}x_0$.

\textit{Subcase}~1.1: $\sigma = \alpha_1 + \ldots + \alpha_r$ with $\Supp \sigma$ of type $\mathsf A_r$ ($r \ge 2$). Then $\Supp \sigma \cap \sigma^\perp = \Pi_\sigma$. Conditions (\ref{Phi1}) and (\ref{Phi3}) yield $\Supp \sigma \setminus \Gamma^\perp = \lbrace \alpha_1, \alpha_r \rbrace$, hence there are $\mu_1, \mu_2 \in \mathrm E$ such that $e_{-(\sigma - \alpha_r)} v_{\mu_1} \ne 0$ and $e_{-(\sigma - \alpha_1)} v_{\mu_2} \ne 0$. Consider the element
\[
f = e_{-(\alpha_2 + \ldots + \alpha_r)} e_{-\alpha_1} + e_{-(\alpha_3 + \ldots + \alpha_r)} e_{-(\alpha_1 + \alpha_2)} + \ldots + e_{-\alpha_r} e_{-(\alpha_1 + \ldots + \alpha_{r-1})} \in U(\mathfrak g).
\]
Direct computations taking into account (\ref{Phi3}) show that
\begin{align*}
e_{\alpha_1} f v_\lambda &= \langle \alpha_1^\vee, \lambda \rangle e_{-(\sigma - \alpha_1)} v_\lambda, & e_{\alpha_1} e_{-\sigma} v_\lambda &= - e_{-(\sigma - \alpha_1)} v_\lambda, \\
e_{\alpha_r} f v_\lambda &= (\langle \alpha_r^\vee, \lambda \rangle + r - 1) e_{-(\sigma - \alpha_r)} v_\lambda, & e_{\alpha_r} e_{-\sigma} v_\lambda &= e_{-(\sigma - \alpha_r)} v_\lambda; \\
e_{\delta} f v_\lambda &= 0, & e_{\delta} e_{-\sigma}v_\lambda &= 0
\end{align*}
for every $\lambda \in \mathrm E$ and $\delta \in \Supp \sigma \setminus \lbrace \alpha_1, \alpha_r \rbrace$. We now put
\[
v = \sum \limits_{\lambda \in \mathrm E} \frac{f + \langle \alpha_1^\vee, \lambda \rangle e_{-\sigma}}{\langle \alpha_1^\vee, \lambda \rangle + \langle \alpha_r^\vee, \lambda \rangle + r - 1} v_\lambda.
\]
Then $e_{\alpha_1} v = 0$, $e_{\alpha_r} v = e_{-(\sigma - \alpha_r)} x_0$, and $e_\delta v = 0$ for all $\delta \in \Supp \sigma \setminus \lbrace \alpha_1, \alpha_r \rbrace$. Clearly, $e_\delta v = 0$ for all $\delta \in \Pi \setminus \Supp \sigma$,
and we have proved (\ref{V2}) and~(\ref{V3}). Since $\pr_{\mu_1} (e_{\alpha_r} v) = e_{-(\sigma - \alpha_r)} v_{\mu_1} \ne 0$, we have $v \ne 0$. As $\pr_{\mu_2} (e_{\alpha_1} e_{-\sigma} x_0) = - e_{-(\sigma - \alpha_1)} v_{\mu_2} \ne 0$, the vector $v$ is not proportional to $e_{-\sigma} x_0$, hence~(\ref{V1}).

\textit{Subcase}~1.2: $\sigma = \alpha_1 + \ldots + \alpha_r$ with $\Supp \sigma$ of type $\mathsf B_r$ ($r \ge 2$). Then $\Supp \sigma \cap \sigma^\perp = \Pi_\sigma \cup \lbrace \alpha_r \rbrace$. Conditions (\ref{Phi1}), (\ref{Phi3}), and~(\ref{Phi4}) yield $\Supp \sigma \setminus \Gamma^\perp = \lbrace \alpha_1, \alpha_r \rbrace$, hence there are $\mu_1, \mu_2 \in \mathrm E$ such that $e_{-(\sigma - \alpha_r)} v_{\mu_1} \ne 0$ and $e_{-(\sigma - \alpha_1)} v_{\mu_2} \ne 0$. Consider the element
\[
f = e_{-(\alpha_2 + \ldots + \alpha_r)} e_{-\alpha_1} + e_{-(\alpha_3 + \ldots + \alpha_r)} e_{-(\alpha_1 + \alpha_2)} + \ldots + e_{-\alpha_r} e_{-(\alpha_1 + \ldots + \alpha_{r-1})} \in U(\mathfrak g).
\]
Direct computations taking into account (\ref{Phi3}) show that
\begin{align*}
e_{\alpha_1} f v_\lambda &= \langle \alpha_1^\vee, \lambda \rangle e_{-(\sigma - \alpha_1)} v_\lambda, & e_{\alpha_1} e_{-\sigma} v_\lambda &= - e_{-(\sigma - \alpha_1)} v_\lambda, \\
e_{\alpha_r} f v_\lambda &= (\langle \alpha_r^\vee, \lambda \rangle + 2r - 2) e_{-(\sigma - \alpha_r)} v_\lambda, & e_{\alpha_r} e_{-\sigma} v_\lambda &= 2e_{-(\sigma - \alpha_r)} v_\lambda, \\
e_\delta f v_\lambda &= 0, & e_\delta e_{-\sigma} v_\lambda &= 0
\end{align*}
for every $\lambda \in \mathrm E$ and $\delta \in \Supp \sigma \setminus \lbrace \alpha_1, \alpha_r \rbrace$. We now put
\[
v = \sum \limits_{\lambda \in \mathrm E} \frac{f + \langle \alpha_1^\vee, \lambda \rangle e_{-\sigma}}{2\langle \alpha_1^\vee, \lambda \rangle + \langle \alpha_r^\vee, \lambda \rangle + 2r - 2} v_\lambda.
\]
Then $e_{\alpha_1} v = 0$, $e_{\alpha_r} v = e_{-(\sigma - \alpha_r)} x_0$, and $e_\delta v = 0$ for all $\delta \in \Supp \sigma \setminus \lbrace \alpha_1, \alpha_r \rbrace$. Clearly, $e_\delta v= 0$ for all $\delta \in \Pi \setminus \Supp \sigma$, and we have proved (\ref{V2}) and~(\ref{V3}). Since $\pr_{\mu_1} (e_{\alpha_r} v) = e_{-(\sigma - \alpha_r)} v_{\mu_1} \ne 0$, we have $v \ne 0$. As $\pr_{\mu_2} (e_{\alpha_1} e_{-\sigma} x_0) = - e_{-(\sigma - \alpha_1)} v_{\mu_2} \ne 0$, the vector $v$ is not proportional to $e_{-\sigma} x_0$, hence~(\ref{V1}).

\textit{Subcase}~1.3: $\sigma = \alpha_1 + 2\alpha_2 + \ldots + 2\alpha_{r-1} + \alpha_r$ with $\Supp \sigma$ of type $\mathsf C_r$ ($r \ge 3$). Then $\Supp \sigma \cap \sigma^\perp = \Pi_\sigma \cup \lbrace \alpha_1 \rbrace$. Conditions (\ref{Phi1}) and (\ref{Phi3}) yield $\lbrace \alpha_2 \rbrace \subset \Supp \sigma \setminus \Gamma^\perp \subset \lbrace \alpha_1, \alpha_2 \rbrace$. In any case there is $\mu \in \mathrm E$ such that $e_{-(\sigma - \alpha_2)} v_\mu \ne 0$ and $e_{-(\sigma - \alpha_1)} v_\mu \ne 0$. Consider the element
\begin{multline*}
f = e_{-(\sigma - \alpha_1)} e_{-\alpha_1} + e_{-(\sigma - \alpha_1 - \alpha_2)} e_{-(\alpha_1 + \alpha_2)} + \ldots + e_{-(\sigma - \alpha_1 - \ldots - \alpha_{r-1})} e_{-(\alpha_1 + \ldots + \alpha_{r-1})} -\\
e_{-(\sigma - \alpha_2)} e_{-\alpha_2} - e_{-(\sigma - \alpha_2 - \alpha_3)} e_{-(\alpha_2 + \alpha_3)} - \ldots - e_{-(\sigma - \alpha_2 - \ldots - \alpha_{r-1})} e_{-(\alpha_2 + \ldots + \alpha_{r-1})} \in U(\mathfrak g).
\end{multline*}
Direct computations taking into account (\ref{Phi3}) show that
\begin{align*}
e_{\alpha_1} f v_\lambda &= \langle \alpha_1^\vee, \lambda \rangle e_{-(\sigma - \alpha_1)} v_\lambda, & e_{\alpha_1} e_{-\sigma} v_\lambda &= -2 e_{-(\sigma - \alpha_1)} v_\lambda, \\
e_{\alpha_2} f v_\lambda &= -(\langle \alpha_2^\vee, \lambda \rangle + r - 1) e_{-(\sigma - \alpha_2)} v_\lambda, & e_{\alpha_2} e_{-\sigma} v_\lambda &= -e_{-(\sigma - \alpha_2)} v_\lambda, \\
e_{\delta} f v_\lambda &= 0, & e_{\delta} e_{-\sigma} v_\lambda &= 0
\end{align*}
for every $\lambda \in \mathrm E$ and $\delta \in \Supp \sigma \setminus \lbrace \alpha_1, \alpha_2 \rbrace$. We now put
\[
v = \sum \limits_{\lambda \in \mathrm E} \frac{2f + \langle \alpha_1^\vee, \lambda \rangle e_{-\sigma}}{\langle \alpha_1^\vee, \lambda \rangle + 2\langle \alpha_2^\vee, \lambda \rangle + 2r - 2} v_\lambda.
\]
Then $e_{\alpha_1} v = 0$, $e_{\alpha_2} v = - e_{-(\sigma - \alpha_2)} x_0$, and $e_\delta v = 0$ for all $\delta \in \Supp \sigma \setminus \lbrace \alpha_1, \alpha_2 \rbrace$. Clearly, $e_\delta v = 0$ for all $\delta \in \Pi \setminus \Supp \sigma$, and we have proved~(\ref{V2}) and~(\ref{V3}) regardless of whether $\alpha_1$ belongs to $\Gamma^\perp$ or not. Since $\pr_\mu (e_{\alpha_2} v) = e_{-(\sigma - \alpha_2)} v_\mu \ne 0$, we have $v \ne 0$. As $\pr_\mu (e_{\alpha_1} e_{-\sigma} x_0) = - 2e_{-(\sigma - \alpha_1)} v_\mu \ne 0$, the vector $v$ is not proportional to $e_{-\sigma} x_0$, hence~(\ref{V1}).

\textit{Subcase}~1.4: $\sigma = \alpha_1 + 2\alpha_2 + 3\alpha_3 + 2\alpha_4$ with $\Supp \sigma$ of type $\mathsf F_4$. Then $\Supp \sigma \cap \sigma^\perp = \Pi_\sigma$. Conditions (\ref{Phi1}) and (\ref{Phi3}) yield $\Supp \sigma \setminus \Gamma^\perp = \lbrace \alpha_4 \rbrace$, hence there is $\mu \in \mathrm E$ such that $e_{-(\sigma - \alpha_4)} v_\mu \ne 0$ and $e_{-(\sigma - \alpha_3)} v_\mu \ne 0$. Consider the element
\begin{multline*}
f = e_{-(\sigma - \alpha_4)} e_{-\alpha_4} - e_{-(\sigma - \alpha_3 - \alpha_4)} e_{-(\alpha_3 + \alpha_4)} + e_{-(\sigma - \alpha_2 - \alpha_3 - \alpha_4)} e_{-(\alpha_2 + \alpha_3 + \alpha_4)} -\\
e_{-(\sigma - \alpha_2 - 2\alpha_3 - \alpha_4)} e_{-(\alpha_2 + 2\alpha_3 + \alpha_4)} \in U(\mathfrak g).
\end{multline*}
Direct computations taking into account (\ref{Phi3}) show that
\begin{align*}
e_{\alpha_3} f v_\lambda &= 0, & e_{\alpha_3} e_{-\sigma} v_\lambda &= -2e_{-(\sigma - \alpha_3)} v_\lambda, \\
e_{\alpha_4} f v_\lambda &= (\langle \alpha_4^\vee, \lambda \rangle + 5) e_{-(\sigma - \alpha_4)} v_\lambda, & e_{\alpha_4} e_{-\sigma} v_\lambda &= -e_{-(\sigma - \alpha_4)}v_\lambda, \\
e_{\alpha_1} f v_\lambda &= e_{\alpha_2} f v_\lambda = 0, & e_{\alpha_1} e_{-\sigma} v_\lambda &= e_{\alpha_2} e_{-\sigma} v_\lambda = 0
\end{align*}
for every $\lambda \in \mathrm E$. We now put
\[
v = \sum \limits_{\lambda \in \mathrm E} \frac{f}{\langle \alpha_4^\vee, \lambda \rangle + 5} v_\lambda.
\]
Then $e_{\alpha_4} v = e_{-(\sigma - \alpha_4)} x_0$ and $e_{\alpha_1} v = e_{\alpha_2} v = e_{\alpha_3} v = 0$. Clearly, $e_\delta v = 0$ for all $\delta \in \Pi \setminus \Supp \sigma$, and we have proved~(\ref{V2}) and~(\ref{V3}). Since $\pr_\mu(e_{\alpha_4} v) = e_{-(\sigma - \alpha_4)} v_\mu \ne 0$, we have $v \ne 0$. As $\pr_\mu(e_{\alpha_3} e_{-\sigma} x_0) = -2 e_{-(\sigma - \alpha_3)} v_\mu \ne 0$, the vector $v$ is not proportional to $e_{-\sigma} x_0$, hence~(\ref{V1}).

\textit{Subcase}~1.5: $\sigma = \alpha_1 + \alpha_2$ with $\Supp \sigma$ of type $\mathsf G_2$. Then $\Supp \sigma \cap \sigma^\perp = \varnothing$. Condition (\ref{Phi1}) yields $\Supp \sigma \setminus \Gamma^\perp = \lbrace \alpha_1, \alpha_2 \rbrace$, hence there are $\mu_1, \mu_2 \in \mathrm E$ such that $e_{-(\sigma - \alpha_2)} v_{\mu_1} \ne 0$ and $e_{-(\sigma - \alpha_1)} v_{\mu_2} \ne 0$. Consider the element
\[
f = e_{-\alpha_2} e_{-\alpha_1} \in U(\mathfrak g).
\]
Direct computations show that
\begin{align*}
e_{\alpha_1} f v_\lambda &= \langle \alpha_1^\vee, \lambda \rangle e_{-(\sigma - \alpha_1)} v_\lambda, & e_{\alpha_1} e_{-\sigma} v_\lambda &= -3 e_{-(\sigma - \alpha_1)} v_\lambda, \\
e_{\alpha_2} f v_\lambda &= (\langle \alpha_2^\vee, \lambda \rangle + 1) e_{-(\sigma - \alpha_2)} v_\lambda, & e_{\alpha_2} e_{-\sigma} v_\lambda &= e_{-(\sigma - \alpha_2)} v_\lambda
\end{align*}
for every $\lambda \in \mathrm E$. We now put
\[
v = \sum \limits_{\lambda \in \mathrm E} \frac{3f + \langle \alpha_1^\vee, \lambda \rangle e_{-\sigma}}{\langle \alpha_1^\vee, \lambda \rangle + 3\langle \alpha_2^\vee, \lambda \rangle + 3} v_\lambda.
\]
Then $e_{\alpha_1} v = 0$ and $e_{\alpha_2} v = e_{-(\sigma - \alpha_2)} x_0$. Clearly, $e_{\delta} v = 0$ for all $\delta \in \Pi \setminus \Supp \sigma$, and we have proved~(\ref{V2}) and~(\ref{V3}). Since $\pr_{\mu_1}(e_{\alpha_2} v) = e_{-(\sigma - \alpha_2)} v_{\mu_1} \ne 0$, we have $v \ne 0$. As $\pr_{\mu_2}(e_{\alpha_1} e_{-\sigma} x_0) = -3e_{-(\sigma - \alpha_1)} v_{\mu_2} \ne 0$, the vector $v$ is not proportional to~$e_{-\sigma} x_0$, hence~(\ref{V1}).

\textit{Case}~2: $\sigma \notin \Delta^+$. It follows from Lemma~\ref{lemma_intersection} that $V^{(T_\ad)}_\sigma \cap \mathfrak gx_0 = \lbrace 0 \rbrace$, hence in this case condition~(\ref{V1}) is equivalent to $v \ne 0$.

\textit{Subcase}~2.1: $\sigma = 2\alpha$ with $\alpha \in \Pi$. Condition~(\ref{Phi6}) yields $\langle \alpha^\vee, \lambda \rangle \in 2\ZZ$ for all $\lambda \in \mathrm E$. Next, in view of~(\ref{Phi1}) there exists $\mu \in \mathrm E$ such that $\langle \alpha^\vee, \mu \rangle > 0$. Then the vector
\[
v = \sum \limits_{\lambda \in \mathrm E : \langle \alpha^\vee, \lambda \rangle > 0} \frac{1}{\langle \alpha^\vee, \lambda \rangle - 1} e_{-\alpha} e_{-\alpha} v_\lambda
\]
evidently has properties (\ref{V1})--(\ref{V3}).

\textit{Subcase}~2.2: $\sigma = \alpha + \beta$ for $\alpha, \beta \in \Pi$ with $\alpha \perp \beta$. Condition~(\ref{Phi5}) yields $\langle \alpha^\vee, \lambda \rangle = \langle \beta^\vee, \lambda \rangle$ for all $\lambda \in \mathrm E$. Next, in view of~(\ref{Phi1}) there exists $\mu \in \mathrm E$ such that $\langle \alpha^\vee, \mu \rangle > 0$. Then the vector
\[
v = \sum \limits_{\lambda \in \mathrm E : \langle \alpha^\vee, \lambda \rangle > 0} \frac{1}{\langle \alpha^\vee, \lambda \rangle} e_{-\alpha} e_{-\beta} v_\lambda
\]
evidently has properties (\ref{V1})--(\ref{V3}).

\textit{Subcase}~2.3: $\sigma = \alpha_1 + 2\alpha_2 + \alpha_3$ with $\Supp \sigma$ of type $\mathsf A_3$. Then $\Supp \sigma \cap \sigma^\perp = \Pi_\sigma$. Condition (\ref{Phi1}) yields $\Supp \sigma \setminus \Gamma^\perp = \lbrace \alpha_2 \rbrace$, hence there is $\mu \in \mathrm E$ such that $e_{-(\sigma - \alpha_2)} v_\mu \ne 0$. Consider the element
\[
f = e_{-(\alpha_1 + \alpha_2 + \alpha_3)} e_{-\alpha_2} - e_{-(\alpha_1 + \alpha_2)} e_{-(\alpha_2 + \alpha_3)} \in U(\mathfrak g).
\]
Direct computations taking into account~(\ref{Phi3}) show that
\[
e_{\alpha_2} f v_\lambda = (\langle \alpha_2^\vee, \lambda \rangle + 1) e_{-(\sigma - \alpha_2)} v_\lambda \quad \text{and} \quad e_{\alpha_1} f v_\lambda = e_{\alpha_3} f v_\lambda = 0
\]
for all $\lambda \in \mathrm E$. We now put
\[
v = \sum \limits_{\lambda \in \mathrm E} \frac{f}{\langle \alpha_2^\vee, \lambda \rangle + 1}v_\lambda.
\]
Then $e_{\alpha_1} v = e_{\alpha_3} v = 0$ and $e_{\alpha_2} v = e_{-(\sigma - \alpha_2)} x_0$. Clearly, $e_\delta v = 0$ for all $\delta \in \Pi \setminus \Supp \sigma$, and we have proved~(\ref{V2}) and~(\ref{V3}). Since $\pr_\mu(e_{\alpha_2}v) = e_{-(\sigma - \alpha_2)} v_\mu \ne 0$, we have $v \ne 0$, hence~(\ref{V1}).

\textit{Subcase}~2.4: $\sigma = \alpha_1 + 2\alpha_2 + 3\alpha_3$ with $\Supp \sigma$ of type $\mathsf B_3$. Then $\Supp \sigma \cap \sigma^\perp = \Pi_\sigma$. Condition (\ref{Phi1}) yields $\Supp \sigma \setminus \Gamma^\perp = \lbrace \alpha_3 \rbrace$, hence there is $\mu \in \mathrm E$ such that $e_{-(\sigma - \alpha_3)} v_\mu \ne 0$. Consider the element
\[
f = e_{-(\alpha_1 +2\alpha_2 + 2\alpha_3)} e_{-\alpha_3} - e_{-(\alpha_1 + \alpha_2 + 2\alpha_3)} e_{-(\alpha_2 + \alpha_3)} + e_{-(\alpha_1 + \alpha_2 + \alpha_3)} e_{-(\alpha_2 + 2\alpha_3)} \in U(\mathfrak g).
\]
Direct computations taking into account~(\ref{Phi3}) show that
\[
e_{\alpha_3} f v_\lambda = (\langle \alpha_3^\vee, \lambda \rangle + 2) e_{-(\sigma - \alpha_3)} v_\lambda \quad \text{and} \quad e_{\alpha_1} f v_\lambda = e_{\alpha_2} f v_\lambda = 0
\]
for all $\lambda \in \mathrm E$. We now put
\[
v = \sum \limits_{\lambda \in \mathrm E} \frac{f}{\langle \alpha_3^\vee, \lambda \rangle + 2}v_\lambda.
\]
Then $e_{\alpha_3} v = e_{-(\sigma - \alpha_3)} x_0$ and $e_{\alpha_1} v = e_{\alpha_2} v = 0$. Clearly, $e_\delta v = 0$ for all $\delta \in \Pi \setminus \Supp \sigma$, and we have proved~(\ref{V2}) and~(\ref{V3}). Since $\pr_\mu(e_{\alpha_3}v) = e_{-(\sigma - \alpha_3)} v_\mu \ne 0$, we have $v \ne 0$,
hence~(\ref{V1}).

\textit{Subcase}~2.5: $\sigma = 2\alpha_1 + \ldots + 2\alpha_r$ with $\Supp \sigma$ of type $\mathsf B_r$ ($r \ge 2$). Then $\Supp \sigma \cap \sigma^\perp = \Pi_\sigma$. Condition (\ref{Phi1}) yields $\Supp \sigma \setminus \Gamma^\perp = \lbrace \alpha_1 \rbrace$, hence there is $\mu \in \mathrm E$ such that $e_{-(\sigma - \alpha_1)} v_\mu \ne 0$. Consider the element
\begin{multline*}
f = 4e_{-\sigma + \alpha_1} e_{-\alpha_1} + 4 e_{-\sigma + \alpha_1 + \alpha_2} e_{-(\alpha_1 + \alpha_2)} + \ldots + 4e_{-\sigma + \alpha_1 + \ldots + \alpha_{r-1}} e_{-(\alpha_1 + \ldots + \alpha_{r-1})} + \\ e_{-(\alpha_1 + \ldots + \alpha_r)} e_{-(\alpha_1 + \ldots + \alpha_r)} \in U(\mathfrak g).
\end{multline*}
Direct computations taking into account~(\ref{Phi3}) show that
\[
e_{\alpha_1} f v_\lambda = (4\langle \alpha_1^\vee, \lambda \rangle + 4r - 6) v_\lambda \quad \text{and} \quad e_\delta f v_\lambda = 0
\]
for all $\lambda \in \mathrm E$ and $\delta \in \Supp \sigma \setminus \lbrace \alpha_1 \rbrace$. We now put
\[
v = \sum \limits_{\lambda \in \mathrm E} \frac{f}{4\langle \alpha_1^\vee, \lambda \rangle + 4r - 6}v_\lambda.
\]
Then $e_{\alpha_1} v = e_{-(\sigma - \alpha_1)} x_0$ and $e_\delta v = 0$ for all $\delta \in \Supp \sigma \setminus \lbrace \alpha_1 \rbrace$. Clearly, $e_\delta v = 0$ for all $\delta \in \Pi \setminus \Supp \sigma$, and we have proved~(\ref{V2}) and~(\ref{V3}). Since $\pr_\mu(e_{\alpha_1} v) = e_{-(\sigma - \alpha_1)} v_\mu \ne 0$, we have $v \ne 0$, hence~(\ref{V1}).

\textit{Subcase}~2.6:  $\sigma = 2\alpha_1 + \ldots + 2\alpha_{r-2} + \alpha_{r-1} + \alpha_r$ with $\Supp \sigma$ of type $\mathsf D_r$ ($r \ge 4$). Then $\Supp \sigma \cap \sigma^\perp = \Pi_\sigma$. Condition (\ref{Phi1}) yields $\Supp \sigma \setminus \Gamma^\perp = \lbrace \alpha_1 \rbrace$, hence there is $\mu \in \mathrm E$ such that $e_{-(\sigma - \alpha_1)} v_\mu \ne 0$. Consider the element
\[
f = e_{-\sigma + \alpha_1} e_{-\alpha_1} + e_{-\sigma + \alpha_1 + \alpha_2} e_{-(\alpha_1 + \alpha_2)} + \ldots + e_{-\sigma + \alpha_1 + \ldots + \alpha_{r-1}} e_{-(\alpha_1 + \ldots + \alpha_{r-1})} \in U(\mathfrak g).
\]
Direct computations taking into account~(\ref{Phi3}) show that
\[
e_{\alpha_1} f v_\lambda = (\langle \alpha_1^\vee, \lambda \rangle + r-2)e_{-\sigma + \alpha_1} v_\lambda \quad \text{and} \quad e_\delta f v_\lambda = 0
\]
for all $\lambda \in \mathrm E$ and $\delta \in \Supp \sigma \setminus \lbrace \alpha_1 \rbrace$. We now put
\[
v = \sum \limits_{\lambda \in \mathrm E} \frac{f}{\langle \alpha_1^\vee, \lambda \rangle + r-2}v_\lambda.
\]
Then $e_{\alpha_1} v = e_{-(\sigma - \alpha_1)} x_0$ and $e_\delta v = 0$ for all $\delta \in \Supp \sigma \setminus \lbrace \alpha_1 \rbrace$. Clearly, $e_\delta v = 0$ for all $\delta \in \Pi \setminus \Supp \sigma$, and we have proved~(\ref{V2}) and~(\ref{V3}). Since $\pr_\mu(e_{\alpha_1}v) = e_{-(\sigma - \alpha_1)} v_\mu \ne 0$, we have $v \ne 0$, hence~(\ref{V1}).

\textit{Subcase}~2.7: $\sigma = 4\alpha_1 + 2\alpha_2$ with $\Supp \sigma$ of type~$\mathsf G_2$. Then $\Supp \sigma \cap \sigma^\perp = \Pi_\sigma$. Condition (\ref{Phi1}) yields $\Supp \sigma \setminus \Gamma^\perp = \lbrace \alpha_1 \rbrace$, hence there is $\mu \in \mathrm E$ such that $e_{-(\sigma - \alpha_1)} v_\mu \ne 0$. Consider the element
\[
f = 4e_{-\sigma + \alpha_1} e_{-\alpha_1} + 4 e_{-\sigma + \alpha_1 + \alpha_2} e_{-(\alpha_1 + \alpha_2)} - 3 e_{-(2\alpha_1 + \alpha_2)} e_{-(2\alpha_1 + \alpha_2)} \in U(\mathfrak g).
\]
Direct computations taking into account~(\ref{Phi3}) show that
\[
e_{\alpha_1} f v_\lambda = (4\langle \alpha_1^\vee, \lambda \rangle + 18) e_{-\sigma + \alpha_1} \quad \text{and} \quad e_{\alpha_2} f v_\lambda = 0
\]
for all $\lambda \in \mathrm E$. We now put
\[
v = \sum \limits_{\lambda \in \mathrm E} \frac{f}{4\langle \alpha_1^\vee, \lambda \rangle + 18}v_\lambda.
\]
Then $e_{\alpha_1} v = e_{-(\sigma - \alpha_1)} x_0$ and $e_{\alpha_2} v = 0$. Clearly, $e_\delta v = 0$ for all $\delta \in \Pi \setminus \Supp \sigma$, and we have proved~(\ref{V2}) and~(\ref{V3}). Since $\pr_\mu(e_{\alpha_1}v) = e_{-(\sigma - \alpha_1)} v_\mu \ne 0$, we have $v \ne 0$, hence~(\ref{V1}).
\end{proof}

The proof of Proposition~\ref{prop_last_step} is completed.

\section{Applications}
\label{sect_applications}

Given a finitely generated and saturated monoid $\Gamma \subset \Lambda^+$, recall the set $\Phi(\Gamma)$ and Theorem~\ref{thm_tangent_space} from \S~\ref{Statement of the main result}. All results obtained in this section depend only on the following parts of Theorem~\ref{thm_tangent_space}:
\begin{itemize}
\item
the $T_\ad$-module $T_{X_0} \mathrm M_\Gamma$ is multiplicity-free;

\item
every element of $\Phi(\Gamma)$ satisfies conditions (\ref{Phi1})--(\ref{Phi8}).
\end{itemize}
We point out that the existence part of Theorem~\ref{thm_tangent_space} (see Proposition~\ref{prop_last_step}) is not used in this section.

\subsection{Auxiliary results on $\Phi(\Gamma)$}

Throughout this subsection, $\Gamma \subset \Lambda^+$ is an arbitrary finitely generated and saturated monoid.

\begin{lemma} \label{lemma_non-proportional}
The set $\Phi(\Gamma)$ contains no proportional elements.
\end{lemma}

\begin{proof}
Assume that $\sigma, \sigma' \in \Phi(\Gamma)$ are two distinct proportional elements. Thanks to~(\ref{Phi2}), one has $\sigma, \sigma' \in \overline \Sigma(G)$. By inspecting Table~\ref{table_spherical_roots}, we get the only two following possibilities (up to interchanging $\sigma$ and~$\sigma'$).

\textit{Case}~1: $\sigma = \alpha$ for some $\alpha \in \Pi$ and $\sigma' = 2\alpha$. Then condition~(\ref{Phi8}) for $\sigma$ contradicts condition~(\ref{Phi7}) for~$\sigma'$.

\textit{Case}~2: $\sigma = \alpha_1 + \ldots + \alpha_r$ with $\Supp \sigma$ of type~$\mathsf B_r$ ($r \ge 2$) and $\sigma' = 2\sigma$. Then condition~(\ref{Phi4}) for~$\sigma$ contradicts condition~(\ref{Phi3}) for~$\sigma'$.
\end{proof}

\begin{lemma} \label{lemma_alpha+beta_in_Phi}
If $\alpha + \beta \in \Phi(\Gamma)$ for some $\alpha, \beta \in \Pi$ with $\alpha \perp \beta$, then $\Phi(\Gamma) \cap \lbrace \alpha, \beta \rbrace = \varnothing$.
\end{lemma}

\begin{proof}
Assume without loss of generality that $\alpha \in \Phi(\Gamma)$. Then $\alpha \in \ZZ\Gamma$ by~(\ref{Phi1}). On the other hand, $2 = \langle \alpha^\vee, \alpha \rangle \ne \langle \beta^\vee, \alpha \rangle = 0$, which contradicts condition~(\ref{Phi5}) for $\alpha + \beta$.
\end{proof}

\begin{lemma} \label{lemma_alpha_not_in_Phi}
Suppose that $\sigma \in \Phi(\Gamma) \setminus \Pi$, $\alpha \in \Supp \sigma$, and $\langle \alpha^\vee, \sigma \rangle > 0$. Then $\alpha \notin \Phi(\Gamma)$.
\end{lemma}

\begin{proof}
As $\langle \alpha^\vee, \sigma \rangle > 0$, by condition~(\ref{Phi7}) for~$\sigma$ there exists $\varrho \in \mathcal K^1$ such that $\iota(\alpha^\vee)$ is a positive multiple of~$\varrho$. Assume $\alpha \in \Phi(\Gamma)$. Then conditions~(\ref{Phi8})(\ref{Phi8b},\,\ref{Phi8c}) for~$\alpha$ yield $\varrho \notin \mathcal K^1$, a contradiction.
\end{proof}

\begin{corollary} \label{crl_Supp_is_not_in_Phi}
Suppose that $\sigma \in \Phi(\Gamma) \setminus \Pi$. Then there exists $\alpha \in \Supp \sigma$ such that $\alpha \notin
\Phi(\Gamma)$.
\end{corollary}

\begin{proof}
This follows from Lemmas~\ref{lemma_alpha_not_in_Phi} and~\ref{lemma_non-acute}.
\end{proof}

\begin{proposition} \label{prop_aux_free}
Every element of $\Phi(\Gamma)$ is primitive in the lattice $\ZZ \Phi(\Gamma)$.
\end{proposition}

\begin{proof}
Assume there exists an element $\sigma \in \ZZ \Phi(\Gamma)$ such that $n\sigma \in \Phi(\Gamma)$ for some $n \ge 2$. Since $\Phi(\Gamma) \subset \overline \Sigma(G) \subset \ZZ^+\Pi$, it follows that $\sigma \in \ZZ^+ \Pi$. An inspection of Table~\ref{table_spherical_roots} shows that $n = 2$ and one of the three cases below occurs.

\textit{Case}~1: $\sigma = \alpha \in \Pi$. Then $2\alpha \in \Phi(\Gamma)$ and condition~(\ref{Phi6}) yields
\begin{equation} \label{eqn_even_values}
\langle \alpha^\vee, \tau \rangle \in 2\ZZ \quad \text{for all} \quad \tau \in \Phi(\Gamma).
\end{equation}
Since $\sigma \in \ZZ \Phi(\Gamma)$, there exists $\sigma_1 \in \Phi(\Gamma) \setminus \lbrace 2\alpha \rbrace$ such that $\alpha \in \Supp \sigma_1$. Then $\Pi_{\sigma_1} \subset \alpha^\perp$ by~(\ref{Phi3}). A case-by-case check of all entries in Table~\ref{table_spherical_roots} together with~(\ref{eqn_even_values}) and Lemma~\ref{lemma_alpha+beta_in_Phi} yields only the following two possibilities for~$\sigma_1$ and~$\alpha$:
\begin{itemize}
\item
$\sigma_1 = \alpha_1 + \alpha_2$ with $\Supp \sigma_1$ of type~$\mathsf B_2$ and $\alpha = \alpha_2$;

\item
$\sigma_1 = \alpha_1 + 2\alpha_2 + 2\alpha_3 + \ldots + 2\alpha_{r-1} + \alpha_r$ with $\Supp \sigma_1$ of type~$\mathsf C_r$ ($r \ge 3$) and $\alpha = \alpha_1$.
\end{itemize}
It is easy to see that $\sigma_1$ is the unique element in $\Phi(\Gamma) \setminus \lbrace 2\alpha \rbrace$ with $\alpha \in \Supp \sigma_1$. The subsequent consideration is divided into three subcases.

\textit{Subcase}~1.1: $\sigma_1 = \alpha_1 + \alpha_2$ with $\Supp \sigma_1$ of type~$\mathsf B_2$ and $\alpha = \alpha_2$. Then there exists $\sigma_2 \in \Phi(\Gamma) \setminus \lbrace 2\alpha, \sigma_1 \rbrace$ such that $\alpha_1 \in \Supp \sigma_2$. Recall that $\alpha_2 \notin \Supp \sigma_2$. As $\langle \alpha_1^\vee, \sigma_1 \rangle = 1$, one has $\sigma_2 \ne \alpha_1$ by Lemma~\ref{lemma_alpha_not_in_Phi} and $\sigma_2 \ne 2\alpha_1$ by condition~(\ref{Phi6}). Further, $\sigma_2 \ne \alpha_1 + \beta$ for all $\beta \in \Pi$ with $\alpha_1 \perp \beta$; otherwise one would have $\langle \alpha_1^\vee, \sigma_1 \rangle > 0$ and $\langle \beta^\vee, \sigma_1 \rangle \le 0$, contradicting condition~(\ref{Phi5}). It follows that $\Supp \sigma_2$ is of type~$\mathsf A_s$ for some $s \ge 2$. Condition~(\ref{Phi3}) for~$\sigma_2$ yields $\Pi_{\sigma_2} \subset \sigma_1^\perp$, whence $s = 2$ and $\sigma_2 = \alpha_1 + \beta_1$ for some $\beta_1 \in \Pi \setminus \lbrace \alpha_1, \alpha_2 \rbrace$. Note that $\sigma_2$ is the unique element in $\Phi(\Gamma) \setminus \lbrace 2\alpha, \sigma_1 \rbrace$ with $\alpha_1 \in \Supp \sigma_2$. Iterating the above argument leads to an infinite chain $\sigma_3, \sigma_4, \ldots \subset \Phi(\Gamma)$ such that for every $i \ge 3$ the following properties hold:
\begin{itemize}
\item
$\Supp \sigma_i$ is of type~$\mathsf A_2$;

\item
$\sigma_i = \beta_{i-2} + \beta_{i-1}$ for some $\beta_{i-1} \in \Pi \setminus \lbrace \alpha_1, \alpha_2, \beta_1, \ldots, \beta_{i-2} \rbrace$;

\item
$\sigma_i$ is the unique element in $\Phi(\Gamma) \setminus \lbrace 2\alpha, \sigma_1, \ldots, \sigma_{i-1} \rbrace$ with $\beta_{i-2} \in \Supp \sigma_i$.
\end{itemize}
As $\Phi(\Gamma)$ is finite, we have got a contradiction.

\textit{Subcase}~1.2: $\sigma_1 = \alpha_1 + 2\alpha_2 + \alpha_3$ with $\Supp \sigma_1$ of type~$\mathsf C_3$ and $\alpha = \alpha_1$. Then there exists $\sigma_2 \in \Phi(\Gamma) \setminus \lbrace 2\alpha, \sigma_1 \rbrace$ such that $\alpha_2 \in \Supp \sigma_2$. Recall that $\alpha_1 \notin \Supp \sigma_2$. As $\sigma_2 \in \ZZ \Gamma$ by~(\ref{Phi1}) and $\Pi_{\sigma_1} \subset \Gamma^\perp$ by~(\ref{Phi3}), it follows that $\alpha_3 \in \sigma_2^\perp$, which implies $\alpha_3 \in \Supp \sigma_2$. If $\Supp \sigma_2 = \lbrace \alpha_2, \alpha_3 \rbrace$ then $\sigma_2 = k(\alpha_2 + \alpha_3)$ with $k \in \lbrace 1,2 \rbrace$, hence $\langle \alpha_3^\vee, \sigma_2 \rangle \ne 0$, which contradicts condition~(\ref{Phi3}) for~$\sigma_1$. Consequently, $\Supp \sigma_2 \ne \lbrace \alpha_2, \alpha_3 \rbrace$ and there exists $\alpha_4 \in \Pi \setminus \lbrace \alpha_1, \alpha_2, \alpha_3 \rbrace$ such that $\Pi \setminus \alpha_4^\perp = \lbrace \alpha_3, \alpha_4 \rbrace$, the set $\lbrace \alpha_1, \alpha_2, \alpha_3, \alpha_4 \rbrace$ is of type~$\mathsf F_4$, and $\Supp \sigma_2 = \lbrace \alpha_2, \alpha_3, \alpha_4 \rbrace$. It follows that
\[
\sigma_2 \in \lbrace \alpha_2 + \alpha_3 + \alpha_4, 2\alpha_2 + 2\alpha_3 + 2\alpha_4, 3\alpha_2 + 2\alpha_3 + \alpha_4 \rbrace.
\]
Condition~(\ref{eqn_even_values}) leaves the only possibility $\sigma_2 = 2\alpha_2 + 2\alpha_3 + 2\alpha_4$. Condition~(\ref{Phi3}) for~$\sigma_2$ then implies $\alpha_2 \in \Gamma^\perp$, which is false because $\langle \alpha_2^\vee, \sigma_1 \rangle \ne 0$.

\textit{Subcase}~1.3: $\sigma_1 = \alpha_1 + 2\alpha_2 + 2\alpha_3 + \ldots + 2\alpha_{r-1} + \alpha_r$ with $\Supp \sigma_1$ of type~$\mathsf C_r$ ($r \ge 4$); $\alpha = \alpha_1$. Then there exists $\sigma_2 \in \Phi(\Gamma) \setminus \lbrace 2\alpha, \sigma_1 \rbrace$ such that $\alpha_2 \in \Supp \sigma_2$. Recall that $\alpha_1 \notin \Supp \sigma_2$. As $\sigma_2 \in \ZZ \Gamma$ by~(\ref{Phi1}) and $\Pi_{\sigma_1} \subset \Gamma^\perp$ by~(\ref{Phi3}), it follows that $\alpha_3 \in \sigma_2^\perp$, which implies $\alpha_3 \in \Supp \sigma_2$. Iterating this argument yields $\alpha_4, \ldots, \alpha_r \in \Supp \sigma_2$. It follows that $\Supp \sigma_2 = \lbrace \alpha_2, \ldots, \alpha_r \rbrace$, so that $\Supp \sigma_2$ is of type~$\mathsf C_{r-1}$ and $\sigma_2 = \alpha_2 + 2\alpha_3 + \ldots + 2\alpha_{r-1} + \alpha_r$. Since $\alpha_3 \in \Pi_{\sigma_1}$ and $\langle \alpha_3^\vee, \sigma_2 \rangle \ne 0$, we obtain a contradiction with condition~(\ref{Phi3}) for~$\sigma_1$.

\textit{Case}~2: $\sigma = \alpha_1 + \ldots + \alpha_r$ with $\Supp \sigma$ of type~$\mathsf B_r$ ($r \ge 2$). Then $2\sigma = 2\alpha_1 + \ldots + 2\alpha_r \in \Phi(\Gamma)$ and hence $\alpha_2, \ldots, \alpha_r \in \Gamma^\perp$ by~(\ref{Phi3}). In view of condition $\sigma \in \ZZ \Phi(\Gamma)$ there exists $\tau \in \Phi(\Gamma) \setminus \lbrace 2\sigma \rbrace$ such that $\alpha_1 \in \Supp \tau$. As $\tau \in \ZZ \Gamma$ by~(\ref{Phi1}), condition $\alpha_2 \in \tau^\perp$ implies $\alpha_2 \in \Supp \tau$. Iterating this argument yields $\alpha_3, \ldots, \alpha_r \in \Supp \tau$, therefore $\Supp \sigma \subset \Supp \tau$. Since $\Phi(\Gamma) \subset \overline \Sigma(G)$, an inspection of Table~\ref{table_spherical_roots} shows that conditions $\alpha_1 \notin \Pi_\tau$ and $\alpha_r \in \tau^\perp$ cannot hold for an element $\tau \in \Phi(\Gamma) \setminus \lbrace 2\sigma \rbrace$, a~contradiction.

\textit{Case}~3: $\sigma = 2\alpha_1 + \alpha_2$ with $\Supp \sigma$ of type~$\mathsf G_2$. Then $2\sigma = 4\alpha_1 + 2\alpha_2 \in \Phi(\Gamma)$ and hence $\alpha_2 \in \Gamma^\perp$ by~(\ref{Phi3}). Since $\Phi(\Gamma) \subset \overline \Sigma(G)$, an inspection of Table~\ref{table_spherical_roots} yields that there are no elements $\tau \in \Phi(\Gamma) \setminus \lbrace 2\sigma \rbrace$ such that $\alpha_2 \in \Supp \tau$ and $\alpha_2 \in \tau^\perp$, which contradicts the condition $\sigma \in \ZZ \Phi(\Gamma)$.
\end{proof}

The following proposition is similar to~\cite[Proposition~5.4]{BvS}.

\begin{proposition}
\label{prop_aux_uniqueness}
Every $\sigma \in \Phi(\Gamma)$ satisfies the condition $\sigma \notin \ZZ^+(\Phi(\Gamma) \setminus \lbrace \sigma\rbrace)$.
\end{proposition}

\begin{proof}
Fix an arbitrary $\sigma \in \Phi(\Gamma)$ and assume that $\sigma \in \ZZ^+(\Phi(\Gamma) \setminus \lbrace \sigma \rbrace$). Fix a subset $\Phi_\sigma \subset \Phi(\Gamma) \setminus \lbrace \sigma \rbrace$ such that $\sigma = \sum \limits_{\tau \in \Phi_\sigma} n_\tau \tau$ with all the coefficients $n_\tau$ being positive integers. Then Lemma~\ref{lemma_non-proportional} yields
\begin{equation} \label{eqn_at_least_2}
|\Phi_\sigma| \ge 2.
\end{equation}
Clearly, every $\tau \in \Phi_\sigma$ satisfies
\begin{equation} \label{eqn_condition1}
\Supp \tau \subset \Supp \sigma;
\end{equation}
moreover,
\begin{equation} \label{eqn_condition2}
\Pi_\sigma \subset \tau^\perp
\end{equation}
by~(\ref{Phi3}). Assume that $|\Supp \sigma \setminus \Pi_\sigma| = 1$. As $\Pi_\sigma \subset \sigma^\perp$, for each $\tau \in \Phi_\sigma$ condition~(\ref{eqn_condition2}) implies that $\tau$ is proportional to~$\sigma$, which is impossible by Lemma~\ref{lemma_non-proportional}. Thus in what follows we assume that $|\Supp \sigma \setminus \Pi_\sigma| \ge 2$. Then an inspection of Table~\ref{table_spherical_roots} leaves the following five cases.

\textit{Case}~1: $\sigma = \alpha_1 + \alpha_2$ with $\Supp \sigma$ of type $\mathsf A_2$, $\mathsf B_2$, or $\mathsf G_2$. Condition~(\ref{eqn_at_least_2}) yields $\Phi_\sigma = \lbrace \alpha_1, \alpha_2 \rbrace$, which contradicts Corollary~\ref{crl_Supp_is_not_in_Phi}.

\textit{Case}~2: $\sigma = \alpha + \beta$ for some $\alpha, \beta \in \Pi$ with $\alpha \perp \beta$. Condition~(\ref{eqn_at_least_2}) yields $\Phi_\sigma = \lbrace \alpha, \beta \rbrace$, which contradicts Lemma~\ref{lemma_alpha+beta_in_Phi}.

\textit{Case}~3: $\sigma = \alpha_1 + \ldots + \alpha_r$ with $\Supp \sigma$ of type $\mathsf A_r$ ($r \ge 3$). An easy computation based on conditions~(\ref{eqn_condition1}) and~(\ref{eqn_condition2}) shows that every element $\tau \in \Phi_\sigma$ has the form
\begin{multline} \label{eqn_element}
\tau = k\alpha_1 + (k + d) \alpha_2 + \ldots + (k + (r-1)d)\alpha_r \\
\text{for some } k,d \in \ZZ \text{ with } k \ge 0, k + (r-1)d \ge 0.
\end{multline}
Since $r \ge 3$ and $\Phi_\sigma \subset \overline \Sigma(G)$, an inspection of Table~\ref{table_spherical_roots} yields $\Phi_\sigma = \varnothing$, which contradicts condition~(\ref{eqn_at_least_2}).

\textit{Case}~4: $\sigma = \alpha_1 + \ldots + \alpha_r$ with $\Supp \sigma$ of type $\mathsf B_r$ ($r \ge 3$). The same computation as in Case~3 shows that every element $\tau \in \Phi_\sigma$ has the form~(\ref{eqn_element}). Since $r \ge 3$ and $\Phi_\sigma \subset \overline \Sigma(G)$, an inspection of Table~\ref{table_spherical_roots} yields $\Phi_\sigma = \varnothing$ for $r \ge 5$ and $\Phi_\sigma \subset \lbrace \alpha_{r-2} + 2\alpha_{r-1} + 3\alpha_r \rbrace$ for $r \in \lbrace 3,4 \rbrace$. In any case we obtain a contradiction with~(\ref{eqn_at_least_2}).

\textit{Case}~5: $\sigma = \alpha_1 + 2\alpha_2 + 2\alpha_3 + \ldots + 2\alpha_{r-1} + \alpha_r$ with $\Supp \sigma$ of type~$\mathsf C_r$ ($r \ge 3$). It follows from conditions~(\ref{eqn_condition1}) and~(\ref{eqn_condition2}) that every element $\tau \in \Phi_\sigma$ has the form
\[
\tau = k_1 \alpha_1 + k_2(2\alpha_2 + 2\alpha_3 + \ldots + 2\alpha_{r-1} + \alpha_r)
\]
for some non-negative integers $k_1, k_2$. Since $r \ge 3$ and $\Phi_\sigma \subset \overline \Sigma(G)$, an inspection of Table~\ref{table_spherical_roots} yields $\Phi_\sigma \subset \lbrace \alpha_1 \rbrace$, which contradicts condition~(\ref{eqn_at_least_2}).
\end{proof}

\begin{lemma} \label{lemma_2alpha}
Suppose that $\alpha \in \ZZ \Gamma \cap \Pi$ and $2\alpha \in \Phi(\Gamma)$. Then $\alpha$ is primitive in~$\ZZ \Gamma$.
\end{lemma}

\begin{proof}
Property (\ref{Phi6}) implies that $\iota(\alpha^\vee) / 2 \in \mathcal L$. As $\langle \alpha^\vee /2, \alpha \rangle = 1$, the claim follows.
\end{proof}

\begin{remark}
Lemma~\ref{lemma_non-proportional}, Proposition~\ref{prop_aux_free}, and Proposition~\ref{prop_aux_uniqueness} would follow easily if we knew a priori that the set $\Phi(\Gamma)$ is linearly independent.
\end{remark}

The above remark leads to the following natural question.

\begin{question}
Is the set $\Phi(\Gamma)$ linearly independent?
\end{question}

\subsection{Applications to affine spherical $G$-varieties}
\label{subsec_applications}

Let $X$ be an affine spherical $G$-variety. Consider the corresponding $T_\ad$-orbit closure $C_X \subset \mathrm M_{\Gamma_X}$ (see~\S\,\ref{subsec_Tad_action}) and equip it with its reduced subscheme structure. Recall the root monoid $\Xi_X$ from Definition~\ref{dfn_root_monoid}.

\begin{proposition} \label{prop_indecomposable}
Suppose that $\sigma$ is an indecomposable element of~$\Xi_X$. Then
\begin{enumerate}[label=\textup{(\alph*)},ref=\textup{\alph*}]
\item \label{prop_indecomposable_a}
$\sigma \in \Phi(\Gamma_X)$;

\item \label{prop_indecomposable_b}
$\sigma$ is primitive in the lattice $\ZZ \Xi_X$.
\end{enumerate}
\end{proposition}

\begin{proof}
(\ref{prop_indecomposable_a}) This follows readily from Corollary~\ref{crl_ind_elements} together with the inclusion $T_{X_0} C_X \subset T_{X_0} \mathrm M_{\Gamma_X}$.

(\ref{prop_indecomposable_b}) Part~(\ref{prop_indecomposable_a}) yields $\ZZ \Xi_X \subset \ZZ \Phi(\Gamma_X)$, which implies the required result in view of Proposition~\ref{prop_aux_free}.
\end{proof}

Recall the monoid $\Xi_X^\sat$ and the set $\overline \Sigma_X$ defined in~\S\,\ref{subsec_root_monoid}.

\begin{theorem} \label{thm_root_monoid_is_free}
There is an inclusion $\overline \Sigma_X \subset \Xi_X$. In particular, $\Xi_X =\Xi_X^\sat$ and the monoid $\Xi_X$ is free.
\end{theorem}

\begin{proof}
Take any $\sigma \in \overline \Sigma_X$. Since $\ZZ^+\overline\Sigma_X = \Xi^\sat_X$ and the set $\overline \Sigma_X$ is linearly independent, there exists a positive integer $n$ such that $n\sigma$ is an indecomposable element of~$\Xi_X$. It follows from Proposition~\ref{prop_indecomposable}(\ref{prop_indecomposable_b}) that $n = 1$ and hence $\sigma \in\Xi_X$.
\end{proof}

\begin{corollary} \label{crl_spherical_roots_are_weights}
There is an inclusion $\overline \Sigma_X \subset \Phi(\Gamma_X)$.
\end{corollary}

\begin{proof}
This follows from Theorem~\ref{thm_root_monoid_is_free} and Proposition~\ref{prop_indecomposable}(\ref{prop_indecomposable_a}).
\end{proof}

\begin{corollary} \label{crl_Tad-OC_is_AS1}
The $T_\ad$-orbit closure $C_X \subset \mathrm M_{\Gamma_X}$ is an affine space of dimension~$|\overline \Sigma_X|$.
\end{corollary}

\begin{proof}
Combining Theorems~\ref{thm_orbit_closure_of_X} and~\ref{thm_root_monoid_is_free} we find that $C_X$ is a multiplicity-free affine $T_\ad$-variety whose weight monoid is generated by the linearly independent set $\overline \Sigma_X$. All the claims follow readily.
\end{proof}

\begin{corollary} \label{crl_Tad-OC_is_AS2}
Let $\Gamma \subset \Lambda^+$ be a finitely generated and saturated monoid. Then every $T_\ad$-orbit closure in $\mathrm M_\Gamma$, equipped with its reduced subscheme structure, is an affine space.
\end{corollary}

\begin{proof}
This follows from Theorem~\ref{thm_bijection} and Corollary~\ref{crl_Tad-OC_is_AS1}.
\end{proof}

\begin{theorem} \label{thm_uniqueness1}
Up to a $G$-isomorphism, every affine spherical $G$-variety $X$ is uniquely determined by the pair $(\Gamma_X, \overline \Sigma_X)$.
\end{theorem}

\begin{proof}
Let $X_1, X_2$ be two affine spherical $G$-varieties with $\Gamma_{X_1} = \Gamma_{X_2}$ and $\overline \Sigma_{X_1} = \overline \Sigma_{X_2}$ and assume that $X_1, X_2$ are not $G$-isomorphic. Put $\Gamma = \Gamma_{X_1} = \Gamma_{X_2}$ and $\overline \Sigma = \overline \Sigma_{X_1} = \overline \Sigma_{X_2}$ for brevity. Consider the closed subsets $C_{X_1}$, $C_{X_2}$, and $Z = C_{X_1} \cup C_{X_2}$ in $\mathrm M_\Gamma$ and equip each of them with its reduced subscheme structure. Thanks to Corollary~\ref{crl_Tad-OC_is_AS1},
\[
\dim C_{X_1} = \dim C_{X_2} = \dim Z = |\overline \Sigma|.
\]
It follows from Theorem~\ref{thm_bijection} that $C_{X_1} \ne C_{X_2}$, hence $C_{X_1}$ and $C_{X_2}$ are distinct irreducible components of~$Z$. Consequently, $X_0$ is a singular point of~$Z$, which implies
\begin{equation} \label{eqn_dim+1}
\dim T_{X_0} Z \ge |\overline \Sigma| + 1.
\end{equation}
By Theorems~\ref{thm_orbit_closure_of_X} and~\ref{thm_root_monoid_is_free}, $C_{X_1}$ and $C_{X_2}$ are isomorphic multiplicity-free affine $T_\ad$-varieties with weight monoid $\ZZ^+\overline \Sigma$, therefore all $T_\ad$-weights of the algebra $\Bbbk[Z]$ belong to~$\ZZ^+\overline \Sigma$. In particular,
\[
\lbrace \tau \in \mathfrak X(T_\ad) \mid - \tau \text{ is a $T_\ad$-weight of } T_{X_0} Z \rbrace \subset \ZZ^+ \overline \Sigma.
\]
Since $T_{X_0} Z \subset T_{X_0} \mathrm M_\Gamma$ and $T_{X_0} \mathrm M_\Gamma$ is a multiplicity-free $T_\ad$-module by Theorem~\ref{thm_tangent_space}, inequality~(\ref{eqn_dim+1}) implies that the set $\Phi(\Gamma) \setminus \overline \Sigma$ contains an element that belongs to $\ZZ^+ \overline \Sigma$. The latter is impossible by Proposition~\ref{prop_aux_uniqueness}.
\end{proof}

Recall from~\S\,\ref{subsec_root_monoid} that to every affine spherical $G$-variety~$X$ one assigns the set $\Sigma_X$ of spherical roots of~$X$.

The following result, which strengthens Theorem~\ref{thm_uniqueness1}, was first obtained by Losev in~\cite[Theorem~1.2]{Lo09b}.

\begin{corollary} \label{crl_uniqueness1}
Up to a $G$-isomorphism, every affine spherical $G$-variety $X$ is uniquely determined by the pair~$(\Gamma_X, \Sigma_X)$.
\end{corollary}

\begin{proof}
Thanks to Corollary~\ref{crl_spherical_roots_are_weights} and Proposition~\ref{prop_aux_free}, the set $\overline \Sigma_X$ is uniquely determined by the pair $(\Gamma_X, \Sigma_X)$ as the set of primitive elements $\nu$ of the lattice $\ZZ \Phi(\Gamma_X)$ such that $\QQ^+\nu$ is an extremal ray of the cone $\QQ^+\Sigma_X \subset \ZZ \Gamma_X \otimes_\ZZ \QQ$. It remains to apply Theorem~\ref{thm_uniqueness1}.
\end{proof}

The following corollary is a particular case of Corollary~\ref{crl_finiteness_MF} below, which was first obtained in~\cite[Corollary~3.4]{AB}.

\begin{corollary} \label{crl_finiteness_Sph}
Up to a $G$-isomorphism, there are only finitely many affine spherical $G$-varieties with a prescribed weight monoid.
\end{corollary}

\begin{proof}
Let $X$ be an affine spherical $G$-variety. Combining Corollary~\ref{crl_spherical_roots_are_weights} with condition~(\ref{Phi2}) yields $\overline \Sigma_X \subset \overline \Sigma(G)$. As the set $\overline \Sigma(G)$ is finite, the claim follows from Theorem~\ref{thm_uniqueness1}.
\end{proof}

\begin{corollary} \label{crl_irr_comp_AS}
Suppose that $\Gamma \subset \Lambda^+$ is a finitely generated and saturated monoid. Then every irreducible component of $\mathrm M_\Gamma$, equipped with its reduced subscheme structure, is an affine space.
\end{corollary}

\begin{proof}
It follows from Theorem~\ref{thm_bijection} and Corollary~\ref{crl_finiteness_Sph} that every irreducible component of $\mathrm M_\Gamma$ is a $T_\ad$-orbit closure. Now the claim is implied by Corollary~\ref{crl_Tad-OC_is_AS2}.
\end{proof}

Let $X$ be an affine spherical $G$-variety. For every $\sigma \in \Sigma_X$, let $\overline \sigma$ denote the unique element in the set $\ZZ^+\sigma \cap \overline \Sigma_X$. The following result is a version of \cite[Theorem~2]{Lo09a} for affine spherical $G$-varieties.

\begin{theorem} \label{thm_overline_sigma}
Under the above assumptions, $\overline \sigma \in \lbrace \sigma, 2\sigma \rbrace$ for every $\sigma \in \Sigma_X$. Moreover, $\overline \sigma = 2\sigma$ if and only if one of the following cases occurs:
\begin{enumerate}[label=\textup{(\arabic*)},ref=\textup{\arabic*}]
\item
$\sigma \notin \overline \Sigma(G)$;

\item
$\sigma = \alpha \in \Pi$ and $\QQ^+ \iota(\alpha^\vee)$ is an extremal ray of the cone~$\mathcal K$;

\item
$\sigma = \alpha_1 + \ldots + \alpha_r$ with $\Supp \sigma$ of type~$\mathsf B_r$ \textup($r \ge 2$\textup) and $\alpha_r \in \Gamma_X^\perp$.
\end{enumerate}
\end{theorem}

\begin{proof}
Fix any $\sigma \in \Sigma_X$. Corollary~\ref{crl_spherical_roots_are_weights} yields $\overline \sigma \in \Phi(\Gamma_X)$, which together with Lemma~\ref{lemma_non-proportional} implies that $\overline \sigma$ is the unique element in the set $\ZZ^+ \sigma \cap \Phi(\Gamma_X)$. Next, $\overline \sigma \in \overline \Sigma(G)$ by~(\ref{Phi2}). Since $\sigma \in \mathfrak X(T)$, an inspection of Table~\ref{table_spherical_roots} along with Lemma~\ref{lemma_2alpha} shows that the condition $\sigma \notin \overline \Sigma(G)$ implies $\overline \sigma = 2\sigma$. Hence in what follows we assume $\sigma \in \overline \Sigma(G)$. Inspecting again Table~\ref{table_spherical_roots}, we find that $\overline \sigma = \sigma$ except for, possibly, one of the following two cases.

\textit{Case}~1: $\sigma = \alpha \in \Pi$. Then $\overline \sigma \in \lbrace \alpha, 2\alpha \rbrace$. Comparing conditions~(\ref{Phi7}) and (\ref{Phi8}), we find that $\overline \sigma = 2\alpha$ if and only if $\QQ^+ \iota(\alpha^\vee)$ is an extremal ray of the cone~$\mathcal K$.

\textit{Case}~2: $\sigma = \alpha_1 + \ldots + \alpha_r$ with $\Supp \sigma$ of type $\mathsf B_r$ ($r \ge 2$). Then $\overline \sigma \in \lbrace \sigma, 2\sigma \rbrace$. Comparing conditions (\ref{Phi3}) and (\ref{Phi4}), we find that $\overline \sigma = 2\sigma$ if and only if $\alpha_r \in \Gamma_X^\perp$.
\end{proof}

\subsection{Consequences for multiplicity-free affine $G$-varieties}
\label{subsec_mf_var}

In this subsection, using a simple reduction, we extend some of the results of~\S\,\ref{subsec_applications} to arbitrary multiplicity-free affine $G$-varieties.

Let $X$ be a multiplicity-free affine $G$-variety and let $\widetilde X$ be the normalization of~$X$. Clearly, $\widetilde X$ is an affine spherical $G$-variety and $\Bbbk[X]$ is naturally identified with a $G$-invariant subalgebra of $\Bbbk[\widetilde X]$.

\begin{proposition} \label{prop_mf_uniqueness}
Up to a $G$-isomorphism, $X$ is uniquely determined by $\widetilde X$ and~$\Gamma_X$.
\end{proposition}

\begin{proof}
This follows from the fact that $\Gamma_X$ uniquely determines $\Bbbk[X]$ as a $G$-submodule and hence as a subspace of $\Bbbk[\widetilde X]$.
\end{proof}

\begin{proposition} \label{prop_mfav}
The following assertions hold:
\begin{enumerate}[label=\textup{(\alph*)},ref=\textup{\alph*}]
\item \label{prop_mfav_a}
$\Gamma_{\widetilde X} = \ZZ \Gamma_X \cap \QQ^+ \Gamma_X$ \textup(that is, $\Gamma_{\widetilde X}$ is the saturation of~$\Gamma_X$\textup);

\item \label{prop_mfav_b}
$\Sigma_{\widetilde X} = \Sigma_X$.
\end{enumerate}
\end{proposition}

\begin{proof}
(\ref{prop_mfav_a}) Since the algebra $\Bbbk[\widetilde X]$ is integral over~$\Bbbk[X]$, it follows from \cite[Corollary~2 of Theorem~4]{Po86} that the algebra $\Bbbk[\widetilde X]^U$ is integral over~$\Bbbk[X]^U$. Taking into account Proposition~\ref{prop_normality_criterion} and the equality $\Bbbk(X)^U = \Quot \Bbbk[X]^U$ (see \cite[Theorem~3.3]{PV}), we conclude that $\Bbbk[\widetilde X]^U$ is the integral closure of $\Bbbk[X]^U$ in $\Quot \Bbbk[X]^U$. It remains to apply Proposition~\ref{prop_SGalgebra}.

(\ref{prop_mfav_b}) Since $X$ and $\widetilde X$ contain the same open $G$-orbit, the claim follows from Proposition~\ref{prop_dual_cones} and the definition of the set of spherical roots (see \S\,\ref{subsec_root_monoid}).
\end{proof}

\begin{corollary} \label{crl_uniqueness1_mf}
Up to a $G$-isomorphism, every multiplicity-free affine $G$-variety $X$ is uniquely determined by the pair~$(\Gamma_X, \Sigma_X)$.
\end{corollary}

\begin{proof}
Combining Proposition~\ref{prop_mfav} and Corollary~\ref{crl_uniqueness1}, we find that the pair $(\Gamma_X, \Sigma_X)$ uniquely determines $\widetilde X$ up to a $G$-isomorphism. It remains to apply Proposition~\ref{prop_mf_uniqueness}.
\end{proof}

The following result was first obtained in~\cite[Corollary~3.4]{AB}.

\begin{corollary} \label{crl_finiteness_MF}
Up to a $G$-isomorphism, there are only finitely many multiplicity-free affine $G$-varieties with a prescribed weight monoid.
\end{corollary}

\begin{proof}
This is implied by Proposition~\ref{prop_mfav}(\ref{prop_mfav_a}), Corollary~\ref{crl_finiteness_Sph}, and Proposition~\ref{prop_mf_uniqueness}.
\end{proof}

\subsection{The uniqueness property for spherical homogeneous spaces}
\label{subsec_uniqueness2}

Given a spherical homogeneous space $G/H$, recall its invariants $\Lambda_{G/H}$, $\Pi^p_{G/H}$, $\Sigma_{G/H}$, and $\mathcal D_{G/H}$ from Appendix~\ref{app_invariants}. Our goal in this subsection is to give a new proof of the following theorem, which is a reformulation of~\cite[Theorem~1]{Lo09a}.

\begin{theorem} \label{thm_uniqueness2}
Up to a $G$-isomorphism, every spherical homogeneous space $G/H$ is uniquely determined by the quadruple $\mathscr H_{G/H} = (\Lambda_{G/H}, \Pi^p_{G/H}, \Sigma_{G/H}, \mathcal D_{G/H})$.
\end{theorem}

The main idea of our proof of this theorem is to perform a reduction to the uniqueness property for affine spherical varieties (Corollary~\ref{crl_uniqueness1}). The reduction itself uses tools that go back to~\cite[\S\,6]{Lu01}; see also~\cite[\S\,3.2]{Br07}.

Recall that a subgroup $H \subset G$ is said to be \textit{spherical} if $G/H$ is a spherical homogeneous space. In the proof of Theorem~\ref{thm_uniqueness2} we shall need the following lemma.

\begin{lemma} \label{lemma_root_cone}
Suppose that $H$ and $H'$ are two spherical subgroups of $G$ such that $H \subset H' \subset N_G(H)$.
Then, modulo the inclusion $\Lambda_{G/H'} \hookrightarrow \Lambda_{G/H}$, the equality $\QQ^+\Sigma_{G/H'} = \QQ^+ \Sigma_{G/H}$ holds.
\end{lemma}

\begin{proof}
Restricting valuations along the chain $\Bbbk(G/H) \supset \Bbbk(G/H') \supset \Bbbk(G/N_G(H))$ yields a chain of maps
\begin{equation} \label{eqn_chain1}
\mathcal V_{G/H} \to \mathcal V_{G/H'} \to \mathcal V_{G/N_G(H)}.
\end{equation}
As follows from \cite[\S\,3.2, Corollary~1]{LV} or \cite[Corollary~1.5]{Kn91}, all the maps in~(\ref{eqn_chain1}) are surjective, which induces a chain of inclusions
\begin{equation} \label{eqn_chain2}
\QQ^+ \Sigma_{G/N_G(H)} \hookrightarrow \QQ^+ \Sigma_{G/H'} \hookrightarrow \QQ^+ \Sigma_{G/H}.
\end{equation}
It was shown in \cite[\S\,5.4]{BriP} that the composite map $\mathcal V_{G/H} \to \mathcal V_{G/ N_G(H)}$ in~(\ref{eqn_chain1}) is the quotient by the vector subspace $\mathcal V_{G/H} \cap (-\mathcal V_{G/H})$. It follows that the composite map $\QQ^+ \Sigma_{G/N_G(H)} \hookrightarrow \QQ^+ \Sigma_{G/H}$ in~(\ref{eqn_chain2}) is bijective hence so are all the maps in~(\ref{eqn_chain2}).
\end{proof}

\begin{proof}[{Proof of Theorem~\textup{\ref{thm_uniqueness2}}}]
Without loss of generality, we may assume that $G$ is the product of a simply connected semisimple group with a torus. Fix a spherical subgroup $H \subset G$.

Let $H^\sharp \subset H$ be the common kernel of all characters of~$H$. Clearly, $H^\sharp$ is a normal subgroup of~$H$ and the group $S = H/H^\sharp$ is diagonalizable. Consider the natural map
\begin{equation} \label{eqn_map_phi}
\varphi \colon H \to S, \quad h \mapsto hH^\sharp.
\end{equation}
The definition of $H^\sharp$ implies that the induced map $\varphi^* \colon \mathfrak X(S) \to \mathfrak X(H)$ is an isomorphism.

Consider the homogeneous space $G / H^\sharp$ and equip it with the natural action of the group $G \times S$ given by $((g, hH^\sharp), xH^\sharp) \mapsto g x h^{-1}H^\sharp$. One easily sees that the stabilizer in $G \times S$ of the point $eH^\sharp$ is the subgroup
\[
\widehat H = \lbrace (h, hH^\sharp) \mid h \in H \rbrace \simeq H.
\]
In what follows, we identify the algebra $\Bbbk[G / H^\sharp]$ with $\Bbbk[G]^{H^\sharp}$.

The action of $G \times S$ on $G / H^\sharp$ induces the $(G \times S)$-module structure on the algebra $\Bbbk[G]^{H^\sharp}$ given by
\[
[(g,hH^\sharp)f](x) = f(g^{-1}xh),
\]
where $g,x \in G$, $h \in H$, and $f \in \Bbbk[G]^{H^\sharp}$. It follows from \cite[Theorem~1]{VK78} that the sphericity of $H$ is equivalent to the condition that the $(G \times S)$-module $\Bbbk[G]^{H^\sharp}$ be multiplicity-free. Let $\widehat \Gamma_{G/H}$ be the set of all pairs $(\lambda, \chi) \in \Lambda^+ \oplus \mathfrak X(H)$ such that $\Bbbk[G]^{H^\sharp}$ contains a simple $(G \times S)$-submodule isomorphic to  $V(\lambda) \otimes \Bbbk_\chi$, where $\Bbbk_\chi$ stands for the one-dimensional $S$-module on which $S$ acts via the character~$\chi$. The set $\widehat \Gamma_{G/H}$ is a submonoid in $\Lambda^+ \oplus \mathfrak X(H)$, called the \textit{extended weight monoid} of~$G/H$; see \cite[\S\,2.2]{Av15} for details.

The variety $G / H^\sharp$ is quasi-affine (see, for instance, \cite[Lemma~2.4]{Av15}).
It is thus identified with an open $(G \times S)$-stable subset of the affine $(G \times S)$-variety $X = \Spec \Bbbk[G]^{H^\sharp}$. By the definitions of $\widehat \Gamma_{G/H}$ and~$X$, there is a $(G \times S)$-module isomorphism
\begin{equation} \label{eqn_module_isomorphism}
\Bbbk[X] \simeq \bigoplus \limits_{(\lambda, \chi) \in \widehat \Gamma_{G/H}} V(\lambda) \otimes \Bbbk_\chi.
\end{equation}

Now consider the subgroup $H^0 \subset G$, which is also spherical. Lemma~\ref{lemma_root_cone} yields
\begin{equation} \label{eqn_equality_of_cones1}
\QQ^+ \Sigma_{G/H} = \QQ^+ \Sigma_{G/H^0}.
\end{equation}
By \cite[Corollary~5.2]{BriP}, the group $N_G(H^0) / H^0$ is diagonalizable, hence so is $H / H^0$. It follows that
\begin{equation}
H^0 = \varphi^{-1}(S^0),
\end{equation}
where the map $\varphi$ is given by~(\ref{eqn_map_phi}).

Observe that the group $G \times S^0$ acts transitively on $G/H^\sharp$ and that the stabilizer in $G \times S^0$ of the point $eH^\sharp$ is the subgroup $\widehat H^0$. It follows from the sphericity of $H^0$ that $G/H^\sharp$ is a spherical $(G \times S^0)$-variety, hence so is~$X$. In particular, the algebra $\Bbbk[X]$ is a multiplicity-free $(G \times S^0)$-module. Consequently, the natural map
\begin{equation} \label{eqn_map_psi}
\psi \colon \widehat \Gamma_{G/H} \to \Gamma_{X}, \quad (\lambda, \chi) \mapsto (\lambda, \left.\chi\right|_{S^0}),
\end{equation}
is injective and hence an isomorphism.

It is easy to see that $\widehat H^0 \subset H^0 \times S^0 \subset N_{G \times S^0}(\widehat H^0)$ and $(G \times S^0) / (H^0 \times S^0) \simeq G / H^0$. This together with Lemma~\ref{lemma_root_cone} implies $\QQ^+ \Sigma_{G / H^0} = \QQ^+ \Sigma_{(G \times S^0)/\widehat H^0}$. Combining this equality with~(\ref{eqn_equality_of_cones1}) yields
\begin{equation} \label{eqn_equality_of_cones2}
\QQ^+ \Sigma_{G/H} = \QQ^+ \Sigma_X.
\end{equation}

We are now ready to recover $H$ from $\mathscr H_{G/H}$.
As shown in \cite[\S\,2.3]{Av15}, the datum $\mathscr H_{G/H}$ uniquely determines $\mathfrak X(H)$ as an abstract group and $\widehat \Gamma_{G/H}$ as a submonoid of $\Lambda^+ \oplus \mathfrak X(H)$. Then $S$ is recovered as the diagonalizable group with $\mathfrak X(S) = \mathfrak X(H)$. Next, the weight monoid $\Gamma_X$ is recovered by the formula $\Gamma_X = \psi(\widehat \Gamma_{G/H})$. Further, equality~(\ref{eqn_equality_of_cones2}) together with $\Gamma_X$ uniquely determine the set $\Sigma_X$. According to Corollary~\ref{crl_uniqueness1}, the pair $(\Gamma_X, \Sigma_X)$ uniquely determines $X$ up to a $(G \times S^0)$-isomorphism. As the map $\psi$ is injective, the action of $G \times S^0$ on $X$ uniquely extends to an action of $G \times S$ satisfying~(\ref{eqn_module_isomorphism}).
Therefore $X$ is uniquely determined up to $(G \times S)$-equivariant isomorphism. At last, up to conjugacy, the subgroup $H$ is recovered from $X$ as the projection to $G$ of the stabilizer in $G \times S$ of a point in the open $(G \times S)$-orbit in~$X$.
\end{proof}

\appendix

\section{The structure constants of Chevalley bases}
\label{app_structure_constants}

Computations carried out in~\S\,\ref{subsec_step_2} require the knowledge of signs of the structure constants of Chevalley bases for the simple Lie algebras of types $\mathsf A_r$, $\mathsf B_r$, $\mathsf C_r$, $\mathsf D_r$, $\mathsf F_4$, and~$\mathsf G_2$. The goal of this appendix is to specify a particular choice of the signs for each of the above-mentioned Lie algebras.

Let $\mathfrak g$ be a simple Lie algebra and let $\lbrace h_\alpha \mid \alpha \in \Pi \rbrace \cup \lbrace e_\alpha \mid \alpha \in \Delta \rbrace$ be a Chevalley basis of~$\mathfrak g$. The following relations for $\alpha, \beta \in \Delta^+$ easily follow, for instance, from~\cite[Theorem~4.1.2]{Ca89}:
\begin{gather*}
N_{\alpha, \beta} = - N_{\beta, \alpha}; \\
N_{-\alpha, -\beta} = - N_{\alpha, \beta}; \\
N_{\alpha, - \beta} =
\begin{cases}
-N_{\beta, \alpha - \beta}\frac{(\alpha - \beta, \alpha - \beta)}{(\alpha, \alpha)} & \ \text{if} \ \alpha - \beta \in \Delta^+; \\
N_{\beta - \alpha, \alpha}\frac{(\beta - \alpha, \beta - \alpha)}{(\beta, \beta)} & \ \text{if} \ \beta - \alpha \in \Delta^+.
\end{cases}
\end{gather*}
These relations show that the signs of all the structure constants of $\mathfrak g$ are uniquely determined by the signs of the structure constants $N_{\alpha, \beta}$ with $\alpha, \beta \in \Delta^+$. In what follows we specify these signs for all the Lie algebras in question.

For type $\mathsf F_4$, we use the signs presented in \cite[Table~II]{VaP}.

For type $\mathsf G_2$, we use the signs extracted from \cite[Table~IV]{VaP}.

For each of the types $\mathsf A_r, \mathsf B_r, \mathsf C_r, \mathsf D_r$, a specific choice of the signs is presented below. It can be obtained from explicit matrix realizations of the corresponding simple Lie algebras.

Type $\mathsf A_r$, $r \ge 2$.

For $1 \le i \le j \le r$ set $\alpha_{ij} = \alpha_i + \ldots +
\alpha_j$. Then $\Delta^+ = \lbrace \alpha_{ij} \mid 1 \le i \le j
\le r \rbrace$.

\begin{center}
\begin{tabular}{|c|c|c|}
\hline

Condition & $k = j+1$ & $i = l + 1$ \\

\hline

Sign of $N_{\alpha_{ij}, \alpha_{kl}}$ & $+$ & $-$\\

\hline
\end{tabular}
\end{center}

~

Type $\mathsf B_r$, $r \ge 2$.

For $1 \le i \le j \le r$ set $\alpha_{ij} = \alpha_i + \ldots +
\alpha_j$.

For $1 \le i < j \le r$ set $\beta_{ij} = \alpha_{ir} +
\alpha_{jr}$.

Then $\Delta^+ = \lbrace \alpha_{ij} \mid 1 \le i \le j \le r
\rbrace \cup \lbrace \beta_{ij} \mid 1 \le i < j \le r \rbrace$.

\begin{center}
\begin{tabular}{|c|c|c|c|c|}
\hline

Condition & $k = j+1$ & $i = l + 1$ &
$j = l = r$, $i < k$ & $j = l = r$, $k < i$ \\

\hline

Sign of $N_{\alpha_{ij}, \alpha_{kl}}$ & $+$ & $-$ & $-$ & $+$\\

\hline
\end{tabular}
\end{center}

\begin{center}
\begin{tabular}{|c|c|c|c|}
\hline

Condition & $i = l+1$ & $j = l + 1$, $k < i$ &
$j = l + 1$, $i < k$ \\

\hline

Sign of $N_{\beta_{ij}, \alpha_{kl}}$ & $-$ & $+$ & $-$\\

\hline
\end{tabular}
\end{center}

~

Type $\mathsf C_r$, $r \ge 3$.

For $1 \le i \le j \le r-1$ set $\alpha_{ij} = \alpha_i + \ldots +
\alpha_j$.

For $1 \le i \le r$ set $\beta_{ir} = \alpha_i + \ldots + \alpha_r$.

For $1 \le i \le j < r$ set $\beta_{ij} = \alpha_{i,r-1} + \alpha_r
+ \alpha_{j,r-1}$.

Then $\Delta^+ = \lbrace \alpha_{ij} \mid 1 \le i \le j \le r-1
\rbrace \cup \lbrace \beta_{ij} \mid 1 \le i \le j \le r \rbrace$.

\begin{center}
\begin{tabular}{|c|c|c|}
\hline

Condition & $k = j+1$ & $i = l + 1$ \\

\hline

Sign of $N_{\alpha_{ij}, \alpha_{kl}}$ & $+$ & $-$\\

\hline
\end{tabular}
\end{center}

\begin{center}
\begin{tabular}{|c|c|c|}
\hline

Condition & $i = l+1$ & $j = l + 1$ \\

\hline

Sign of $N_{\beta_{ij}, \alpha_{kl}}$ & $-$ & $-$\\

\hline
\end{tabular}
\end{center}

~

Type $\mathsf D_r$, $r \ge 4$.

For $1 \le i \le j \le r-1$ set $\alpha_{ij} = \alpha_i + \ldots +
\alpha_j$.

For $1 \le i \le r-1$ set $\beta_{ir} = \alpha_{i,r-1} + (\alpha_r -
\alpha_{r-1})$.

For $1 \le i < j \le r-1$ set $\beta_{ij} = \alpha_{i,r-1} +
(\alpha_r - \alpha_{r-1}) + \alpha_{j,r-1}$.

Then $\Delta^+ = \lbrace \alpha_{ij} \mid 1 \le i \le j \le r-1
\rbrace \cup \lbrace \beta_{ij} \mid 1 \le i < j \le r \rbrace$.

\begin{center}
\begin{tabular}{|c|c|c|}
\hline

Condition & $k = j+1$ & $i = l + 1$ \\

\hline

Sign of $N_{\alpha_{ij}, \alpha_{kl}}$ & $+$ & $-$\\

\hline
\end{tabular}
\end{center}

\begin{center}
\begin{tabular}{|c|c|c|c|}
\hline

Condition & $i = l+1$ & $j = l + 1$, $k < i$ &
$j = l + 1$, $i < k$ \\

\hline

Sign of $N_{\beta_{ij}, \alpha_{kl}}$ & $-$ & $+$ & $-$\\

\hline
\end{tabular}
\end{center}


\section{Invariants of spherical homogeneous spaces}
\label{app_invariants}

In this appendix we recall combinatorial invariants of spherical homogeneous spaces and their equivariant embeddings used in this paper. In what follows, $G/H$ is an arbitrary spherical homogeneous space.

Let $P$ denote the stabilizer of the open $B$-orbit in $G/H$. Clearly, $P$ is a parabolic subgroup of $G$ containing~$B$. We set
\[
\Pi^p_{G/H} = \lbrace \alpha \in \Pi \mid e_{-\alpha} \in \mathfrak p \rbrace.
\]

The next invariants of $G/H$ are the \textit{weight lattice}
\[
\Lambda_{G/H} = \lbrace \lambda \in \mathfrak X(T) \mid \Bbbk(G/H)^{(B)}_\lambda \ne \lbrace 0 \rbrace \rbrace
\]
and the corresponding dual $\QQ$-vector space
\[
\mathcal Q_{G/H} = \Hom_\ZZ(\Lambda_{G/H}, \QQ).
\]
For every $\lambda \in \Lambda_{G/H}$, we fix a nonzero rational function $f_\lambda \in \Bbbk(G/H)^{(B)}_\lambda$. Since $G/H$ contains an open $B$-orbit, it follows that $\Bbbk(G/H)^{(B)}_\lambda = \Bbbk f_\lambda$ for all $\lambda \in \Lambda_{G/H}$.

Every discrete $\QQ$-valued valuation $v$ of $\Bbbk(G/H)$ vanishing on~$\Bbbk^\times$ determines an element $\rho_v \in \mathcal Q_{G/H}$ such that $ \langle \rho_v, \lambda \rangle = v(f_\lambda)$ for all $\lambda \in \Lambda_{G/H}$. It is known (see~\cite[7.4]{LV} or~\cite[Corollary~1.8]{Kn91}) that the restriction of the map $v \mapsto \rho_v$ to the set of $G$-invariant $\QQ$-valued valuations of $\Bbbk(G/H)$ vanishing on~$\Bbbk^\times$ is injective; we denote its image by~$\mathcal V_{G/H}$. It was proved in~\cite[\S\,3]{Br90} that $\mathcal V_{G/H}$ is a cosimplicial cone in~$\mathcal Q_{G/H}$. Consequently, there is a uniquely determined linearly independent set $\Sigma_{G/H}$ of primitive elements in~$\Lambda_{G/H}$ such that
\[
\mathcal V_{G/H} = \lbrace q \in \mathcal Q_{G/H} \mid \langle q, \sigma \rangle \le 0 \ \text{for all} \ \sigma \in \Sigma_{G/H} \rbrace.
\]
Elements of $\Sigma_{G/H}$ are called \textit{spherical roots} of~$G/H$ and $\mathcal V_{G/H}$ is called the \textit{valuation cone} of~$G/H$.

Let $\mathcal D_{G/H}$ denote the set of $B$-stable prime divisors in~$G/H$. Elements of $\mathcal D_{G/H}$ are called \textit{colors} of~$G/H$. For every $D \in \mathcal D_{G/H}$, let $v_D$ be the valuation of $\Bbbk(G/H)$ defined by~$D$, that is, $v_D(f) = \ord_D(f)$ for every $f \in \Bbbk(G/H)$. Let $\rho_{G/H} \colon \mathcal D_{G/H} \to \mathcal Q_{G/H}$ be the map given by $\rho_{G/H}(D) = \rho_{v_D}$ for all $D \in \mathcal D_{G/H}$. We regard $\mathcal D_{G/H}$ as an abstract set equipped with the map~$\rho_{G/H}$.

For an arbitrary irreducible $G$-variety $X$ containing $G/H$ as an open $G$-orbit, one defines the same invariants $\Pi^p_X$, $\Lambda_X$, $\mathcal Q_X$, $\mathcal V_X$, $\Sigma_X$, $\mathcal D_X$, and $\rho_X$ of $X$ as those of $G/H$.

For a multiplicity-free affine $G$-variety~$X$, the set $\Sigma_X$ defined right above coincides with the set $\Sigma_X$ defined in \S\,\ref{subsec_root_monoid}. This follows from the following proposition, which is a particular case of~\cite[Lemma~6.6, iii)]{Kn96}.

\begin{proposition} \label{prop_dual_cones}
Suppose that $X$ is a multiplicity-free affine $G$-variety. Then the cone $\QQ^+ \Xi_X$ is dual to $-\mathcal V_X$.
\end{proposition}



\begin{thebibliography}{10000000}


\setlength{\itemsep}{-\parsep}



\bibitem[AB05]{AB}
V.~Alexeev, M.~Brion, \textit{Moduli of affine schemes with reductive group action}, J. Algebraic Geom. \textbf{14} (2005), no.~1, 83--117; see also
\href{http://arxiv.org/abs/math/0301288}%
{\texttt{arXiv:math/0301288\,[math.AG]}}.

\bibitem[Av15]{Av15}
R.~Avdeev, \textit{Strongly solvable spherical subgroups and their combinatorial invariants}, Selecta Math. (N.\,S.) \textbf{21} (2015), no.~3, 931--993; see also
\href{http://arxiv.org/abs/1212.3256}%
{\texttt{arXiv:1212.3256\,[math.AG]}}.

\bibitem[ACF18]{ACF18}
R.~Avdeev, S.~Cupit-Foutou, \textit{On the irreducible components of moduli schemes for affine spherical varieties}, Transform. Groups \textbf{23} (2018), to appear, \href{https://doi.org/10.1007/s00031-017-9443-8}
{\texttt{DOI:10.1007/s00031-017-9443-8}}; see also
\href{http://arxiv.org/abs/1406.1713v4}%
{\texttt{arXiv:1406.1713v4\,[math.AG]}}.

\bibitem[Bo68]{Bo}
N.~Bourbaki, \'El\'ements de math\'ematique. Groupes et Alg\`ebres de Lie. Chapitre IV: Groupes de Coxeter et Syst\`emes de Tits. Chapitre V: Groupes engendr\'es par des r\'eflexions. Chapitre VI: Syst\`emes de racines. Actualit\'es Scientifiques et Industrielles, No. 1337 Hermann, Paris, 1968.

\bibitem[BCF08]{BraCu08}
P.~Bravi, S.~Cupit-Foutou, \textit{Equivariant deformations of the affine multicone over a flag variety}, Adv. Math. \textbf{217} (2008), no.~6, 2800--2821; see also
\href{http://arxiv.org/abs/math/0603690}%
{\texttt{arXiv:math/0603690\,[math.AG]}}.

\bibitem[BVS16]{BvS}
P.~Bravi, B.~Van Steirteghem, \textit{The moduli scheme of affine spherical varieties with a free weight monoid}, Int. Math. Res. Not. IMRN \textbf{2016} (2016), no.~15, 4544--4587; see also
\href{http://arxiv.org/abs/1406.6041}%
{\texttt{arXiv:1406.6041\,[math.AG]}}.

\bibitem[Br90]{Br90}
M.~Brion, \textit{Vers une g\'en\'eralisation des espaces sym\'etriques}, J.~Algebra \textbf{134} (1990), no.~1, 115--143.

\bibitem[Br07]{Br07}
M.~Brion, \textit{The total coordinate ring of a wonderful variety}, J.~Algebra \textbf{313} (2007), 61--99.

\bibitem[Br13]{Br13}
M. Brion, \textit{Invariant Hilbert schemes}, Handbook of Moduli, Vol. I, Adv. Lect. in Math. \textbf{24}, 63--118, International Press, 2013; see also
\href{http://arxiv.org/abs/1102.0198}%
{\texttt{arXiv:1102.0198\,[math.AG]}}.

\bibitem[BP87]{BriP}
M.~Brion, F.~Pauer, \textit{Valuations des espaces homog\`enes sph\'eriques}, Comment. Math. Helv. \textbf{62} (1987), no.~2, 265--285.

\bibitem[Ca89]{Ca89}
R.\,W.~Carter, \textit{Simple groups of Lie type}, Wiley Classics Library, John Wiley \& Sons, 1989.

\bibitem[CF09]{Cu}
S.~Cupit-Foutou, \textit{Wonderful varieties: a geometrical realization}, preprint, see \\
\href{http://arxiv.org/abs/0907.2852}%
{\texttt{arXiv:0907.2852\,[math.AG]}}.

\bibitem[Ha67]{Ha67}
D.~Hadziev, \textit{Some questions in the theory of vector invariants}, Math. USSR-Sb. \textbf{1} (1967), no.~3, 383--396.

\bibitem[Ha77]{Ha77}
R.~Hartshorne, \textit{Algebraic geometry}, Graduate texts in mathematics, no.~\textbf{52}, Springer, New York Heidelberg, 1977.

\bibitem[Ja07]{Jan}
S.~Jansou, \textit{D\'eformations des c\^ones de vecteurs primitifs}, Math. Ann. \textbf{338} (2007), no.~3, 627--667; see also
\href{http://arxiv.org/abs/math/0506133}%
{\texttt{arXiv:math/0506133\,[math.AG]}}.

\bibitem[KKMS73]{KKMS73}
G.~Kempf, F.~Knudsen, D.~Mumford, B.~Saint-Donat, \textit{Toroidal embeddings I}, Lecture Notes in Mathematics, vol.~\textbf{339}, Springer, Berlin--New York, 1973.

\bibitem[Kn91]{Kn91}
F.~Knop, \textit{The Luna-Vust theory of spherical embeddings}, Proceedings of the Hyderabad Conference on Algebraic Groups (Hyderabad, India, 1989), Manoj Prakashan, Madras, 1991, 225--249.

\bibitem[Kn94]{Kn94}
F.~Knop, \textit{A Harish-Chandra homomorphism for reductive group actions}, Ann. of Math. (2) \textbf{140} (1994), no.~2, 253--288.

\bibitem[Kn96]{Kn96}
F.~Knop, \textit{Automorphisms, root systems, and compactifications of homogeneous varieties}, J.~Amer. Math. Soc. \textbf{9} (1996), no.~1, 153--174.

\bibitem[Li02]{Li02}
Q.~Liu, \textit{Algebraic geometry and arithmetic curves}, Oxford graduate texts in mathematics, \textbf{6}, Oxford University Press, Oxford, 2002.

\bibitem[Lo09a]{Lo09a}
I.~Losev, \textit{Uniqueness property for spherical homogeneous spaces}, Duke Math. J. \textbf{147} (2009), no.~2, 315--343; see also
\href{http://arxiv.org/abs/math/0703543}%
{\texttt{arXiv:math/0703543\,[math.AG]}}.

\bibitem[Lo09b]{Lo09b}
I.\,V.~Losev, \textit{Proof of the Knop conjecture}, Ann. Inst. Fourier \textbf{59} (2009), no.~3, 1105--1134; see also
\href{http://arxiv.org/abs/math/0612561}%
{\texttt{arXiv:math/0612561\,[math.AG]}}.

\bibitem[Lu01]{Lu01}
D.~Luna, \textit{Vari\'et\'es sph\'eriques de type~A}, Inst. Hautes \'Etudes Sci. Publ. Math. \textbf{94} (2001), 161--226.

\bibitem[LuV83]{LV}
D.~Luna, Th.~Vust, \textit{Plongements d'espaces homog\`enes}, Comment. Math. Helv. \textbf{58} (1983), no.~2, 186--245.

\bibitem[PVS12]{PvS12}
S.\,A.~Papadakis, B.~Van Steirteghem, \textit{Equivariant degenerations of spherical modules for groups of type $\mathsf A$}, Ann. Inst. Fourier \textbf{62} (2012), no.~5, 1765--1809; see also
\href{http://arxiv.org/abs/1008.0911}%
{\texttt{arXiv:1008.0911\,[math.AG]}}.

\bibitem[PVS16]{PvS16}
S.\,A.~Papadakis, B.~Van Steirteghem, \textit{Equivariant degenerations of spherical modules: part II}, Algebr. Represent. Theory \textbf{19} (2016), no.~5, 1135--1171; see also
\href{http://arxiv.org/abs/1505.07446}%
{\texttt{arXiv:1505.07446\,[math.AG]}}.

\bibitem[Po86]{Po86}
V.\,L.~Popov, \textit{Contraction of the actions of reductive algebraic groups}, Math. USSR-Sb. \textbf{58} (1987), no.~2, 311--335.

\bibitem[PoV94]{PV}
V.\,L.~Popov, E.\,B.~Vinberg, \textit{Invariant theory}, Algebraic geometry. IV: Linear algebraic groups, invariant theory, Encycl. Math. Sci., vol. \textbf{55}, 1994, pp. 123--278.

\bibitem[Ti11]{Tim}
D.\,A.~Timashev, \textit{Homogeneous spaces and equivariant embeddings}, Encycl. Math. Sci., vol.~\textbf{138}, Springer, Berlin Heidelberg, 2011.

\bibitem[VaP96]{VaP}
N.~Vavilov, E.~Plotkin, \textit{Chevalley groups over commutative rings: I. Elementary calculations}, Acta Appl. Math. \textbf{45} (1996), 73--113.

\bibitem[ViK78]{VK78}
E.\,B.~Vinberg, B.\,N.~Kimel'fel'd, \textit{Homogeneous domains on flag manifolds and spherical subgroups of semisimple Lie groups}, Funct. Anal. Appl. \textbf{12} (1978), no.~3, 168--174.

\bibitem[ViP72]{VP72}
E.\,B.~Vinberg, V.\,L.~Popov, \textit{On a class of quasihomogeneous affine varieties}, Math. USSR-Izv. \textbf{6} (1972), no.~4, 743--758.

\bibitem[Vu76]{Vu76}
Th.~Vust, \textit{Sur la th\'eorie des invariants des groupes classiques}, Ann. Inst. Fourier (Grenoble) \textbf{26} (1976), no.~1, 1--31.

\end{thebibliography}
\end{document}